\documentclass[11pt]{article}
\title{A non-backtracking method for long matrix and tensor completion}
\date{}
\usepackage{fullpage}
\usepackage{mathtools}
\usepackage{graphicx}
\usepackage{etoolbox}
\usepackage{amsmath, amssymb, amsthm}
\usepackage{hyperref}
\hypersetup{colorlinks, citecolor=blue, unicode}
\usepackage[dvipsnames]{xcolor}
\usepackage{tikz}
\usetikzlibrary{patterns,shapes.arrows, arrows.meta, calc, shapes.geometric, arrows}

\usepackage{mathtools}
\mathtoolsset{showonlyrefs}
\usepackage{times}
\usepackage{xcolor}

\usepackage{algorithm, algpseudocode}
\algrenewcommand\algorithmicrequire{\textbf{Input:}}
\algrenewcommand\algorithmicensure{\textbf{Output:}}

\usepackage{mathtools}
\usepackage{graphicx}
\usepackage{etoolbox}
\usepackage{amsmath, amssymb}

\newcommand{\dR}{\mathbb{R}}

\newcommand{\dN}{\mathbb{N}}
\newcommand{\dQ}{\mathbb{Q}}

\newcommand{\dF}{\mathbb{F}}
\newcommand{\dE}{\mathbb{E}}
\newcommand{\dP}{\mathbb{P}}

\newcommand{\cE}{\mathcal{E}}
\newcommand{\cF}{\mathcal{F}}
\newcommand{\cG}{\mathcal{G}}
\newcommand{\cI}{\mathcal{I}}

\newcommand{\cL}{\mathcal{L}}
\newcommand{\cM}{\mathcal{M}}
\newcommand{\cN}{\mathcal{N}}

\newcommand{\cS}{\mathcal{S}}
\newcommand{\cT}{\mathcal{T}}

\newcommand{\ind}{\mathbf{1}}

\providecommand{\given}{}
\DeclarePairedDelimiterXPP{\Pb}[1]{\mathbb{P}}{\lparen}{\rparen}{}{\renewcommand{\given}{\nonscript{}\:\delimsize\vert\nonscript{}\:\mathopen{}} #1}
\DeclarePairedDelimiterXPP{\E}[1]{\mathbb{E}}[]{}{\renewcommand{\given}{\nonscript{}\:\delimsize\vert\nonscript{}\:\mathopen{}} #1}
\DeclarePairedDelimiterX{\Set}[1]\lbrace\rbrace{\renewcommand{\given}{\nonscript{}\:\delimsize\vert\nonscript{}\:\mathopen{}} #1}

 \DeclareMathOperator{\diag}{diag}
\DeclareMathOperator{\tr}{tr}

\DeclareMathOperator{\im}{im}
\DeclareMathOperator{\vect}{vect}

\DeclareMathOperator{\sign}{sign}

\DeclarePairedDelimiterX{\norm}[1]\lVert\rVert{\ifblank{#1}{\: \cdot \:}{#1}}

\DeclareMathOperator{\Poi}{Poi}
\DeclareMathOperator{\Bin}{Bin}

\DeclareMathOperator{\Unif}{Unif}
\DeclareMathOperator{\dtv}{d_{TV}}

\newcommand{\bilingualcommand}[3]{%
	\newcommand{#1}[1][\ ]{%
		##1%
		\iflanguage{english}{\text{#2}}{%
			\iflanguage{french}{\text{#3}}{}%
		}%
		##1%
	}%
}

\bilingualcommand{\where}{where}{où}
\bilingualcommand{\textif}{if}{si}
\bilingualcommand{\textand}{and}{et}
\bilingualcommand{\textiff}{if and only if}{si et seulement si}
\bilingualcommand{\otherwise}{otherwise}{sinon}

\newcommand{\eps}{\varepsilon}

\newcommand{\quand}{\quad \textand{} \quad}
\newcommand{\qquand}{\qquad \textand{} \qquad}

 \newtheorem{theorem}{Theorem}
 \newtheorem{lemma}{Lemma}
 \newtheorem{proposition}{Proposition}
 \newtheorem{definition}{Definition}
\newtheorem{corollary}{Corollary}

\author{Ludovic Stephan\\
Information, Learning and Physics (IdePHICS)\\
École Polytechnique Fédérale de Lausanne (EPFL) \\
\texttt{ludovic.stephan@epfl.ch}
\and 
Yizhe Zhu\\
Department of Mathematics\\
University of California, Irvine\\
\texttt{yizhe.zhu@uci.edu}}

\begin{document}

\maketitle

\begin{abstract}%
We consider the problem of low-rank rectangular matrix completion in the regime where the matrix $M$  of size $n\times m$ is ``long", i.e., the aspect ratio $m/n$ diverges to infinity. Such matrices are of particular interest in the study of tensor completion, where they arise from the unfolding of a low-rank tensor. In the case where the sampling probability is $\frac{d}{\sqrt{mn}}$, we propose a new spectral algorithm for recovering the singular values and left singular vectors of the original matrix $M$ based on a variant of the standard non-backtracking operator of a suitably defined bipartite weighted random graph, which we call a \textit{non-backtracking wedge operator}. When $d$ is above a Kesten-Stigum-type sampling threshold, our algorithm recovers a correlated version of the singular value decomposition of $M$ with quantifiable error bounds. This is the first result in the regime of bounded $d$ for weak recovery and the first for weak consistency when $d\to\infty$ arbitrarily slowly without any polylog factors. As an application, for low-CP-rank orthogonal $k$-tensor completion,  we efficiently achieve weak recovery with sample size $O(n^{k/2})$ and weak consistency with sample size $\omega(n^{k/2})$. A similar result is obtained for low-multilinear-rank tensor completion with $O(n^{k/2})$ many samples.

\end{abstract}

\section{Introduction}

Matrix completion is the problem of reconstructing a matrix from a (usually random) subset of entries, leveraging prior structural knowledge such as rank and incoherence. When the two dimensions $n, m$ of the matrix $M$ are comparable, this problem has been widely studied in the past decades, with a fairly general pattern emerging. Roughly speaking, if the ground truth matrix is low-rank and eigenvectors are sufficiently incoherent (or delocalized), sampling $\Omega(n\log n)$ many entries uniformly at random with a certain efficient optimization procedure, one can exactly recover the ground truth matrix \cite{keshavan.montanari.ea_2009_matrix,keshavan2010matrix,candes.tao_2010_power,candes.plan_2010_matrix,recht2011simpler,candes2012exact,jain2013low}. This is not the case when the sample size is only $O(n)$, where full completion is provably impossible under uniform sampling. However, an approximation of the ground truth matrix is still possible \cite{heiman2014deterministic,gamarnik2017matrix,brito2022spectral} with efficient algorithms, and more refined estimates of the eigenvectors were studied in \cite{bordenave2020detection}.

In recent years, attention has shifted to ``long'' matrix models where the row and column sizes of the matrix are not proportional. This is the case, for example, in the spiked long matrix and tensor models \cite{montanari.wu_2022_fundamental,ding.hopkins.ea_2020_estimating,arous2021long,montanari.richard_2014_statistical}. This situation also occurs in bipartite graph clustering and community detection, where the goal is to recover the community structure in the smaller vertex set \cite{florescu2016spectral,ndaoud2021improved,cai2021subspace,braun2022minimax,braun2023strong} for a bipartite stochastic block model, and in subspace recovery problems \cite{cai2021subspace,zhou2021rate}.

Our work is closely related to tensor completion, which is the higher-order analog of matrix completion. Indeed, tensor problems are highly related to their long matrix counterparts through the unfolding operation for an order-$k$ tensor of size $n^k$:
\[
    \mathrm{unfold}_{a, b}: \dR^{n^k} \to \dR^{n^a \times n^b},
\]
where $a + b = k$. For example, an order-3 tensor of size $n$ yields an unfolded matrix of size $n \times n^2$; in general, uneven unfoldings (i.e., with $a \neq b$) of a tensor give rise to long rectangular matrices.
In those tensor problems, the ground truth is assumed to have a low CP-rank, which implies that the unfolded matrix is low-rank. In low-rank $k$-tensor completion, the best known polynomial-time algorithms require $\tilde O(n^{k/2})$ revealed entries \cite{jain2014provable,montanari.sun_2018_spectral,xia2019polynomial,cai2021nonconvex}, while information theoretically $\tilde{O}(n)$ suffices \cite{ghadermarzy2019near,harris2021,harris2023spectral}. In \cite{barak.moitra_2016_noisy}, the authors showed lower bounds
for a family of algorithms based on the sum-of-squares hierarchy for order-3 noisy tensor completion, where $\Omega(n^{3/2})$ many samples are needed,  by connecting
it to the literature on refuting random 3-SAT. They also conjectured such a sample size is needed for any polynomial-time algorithm. This is aligned with other tensor estimation problems where a diverging statistical-to-computational gap is expected \cite{montanari.richard_2014_statistical,wein2019kikuchi,jagannath2020statistical,arous2020algorithmic,dudeja2022statistical,auddy2021estimating}. However, so far no known work has reached the exact threshold $\Omega(n^{3/2})$, only approaching it up to polylog factors, in a similar fashion to the matrix completion phase diagram described above.
For exact recovery results on tensor completion, all the methods in the literature are either convex optimization or spectral initialization with iterative optimization steps. Our method achieves weak consistency without any log factors, which could improve the sample size for spectral initialization in the work \cite{jain2014provable,cai2021nonconvex,xia2019polynomial,xia2021statistically}. 

The applications we mentioned above are related to the same question: 

\emph{Can we improve the sample complexity for estimating a one-sided structure in a  long matrix?}

Specifically, we explore the use of the \textit{non-backtracking operator} in this context. The non-backtracking operator is an important object in the study of spectral graph theory \cite{bass_iharaselberg_1992,terras2010zeta}. Recently, it has been a powerful tool for understanding the spectrum of a very sparse random matrix \cite{bordenave2018nonbacktracking,stephan2020non, bordenave2020new,bordenave2019eigenvalues, brito2022spectral,benaych2020spectral,alt2021extremal,alt2021delocalization,dumitriu2022extreme,stephan2022sparse} and for designing efficient algorithms in community detection and matrix completion \cite{krzakala2013spectral,bordenave2018nonbacktracking,stephan2020non,bordenave2020detection,stephan2022sparse,zhu2023non}. In the bipartite graph setting, the non-backtracking spectrum was analyzed in \cite{brito2022spectral,dumitriu2022extreme,dumitriu2021spectra} when the two sides have a bounded aspect ratio. For the spectrum of the adjacency matrix of a random bipartite graph with a diverging aspect ratio, only a few results were obtained in \cite{zhu2022second,dumitriu2020global} for random bipartite biregular graphs, and in \cite{guruswami2022l_p} for the adjacency matrix of bipartite random graphs with random sign flips.

In this work, we aim to address this gap by defining a non-backtracking operator suitable for sparse matrix estimation with a diverging aspect ratio,  which we call \textit{a non-backtracking wedge matrix}. Our work sheds light on the fundamental limits and practical algorithms for sparse long matrix estimation with a focus on one-sided structure recovery. To the best of our knowledge, this is the first result for a long $n\times m$ low-rank matrix completion with sampling probability $p=\frac{d}{\sqrt{mn}}$ in the regime of bounded $d$ for weak recovery and the first result for weak consistency when $d\to\infty$ arbitrarily slowly without any polylog factors. These two types of recovery guarantees are well studied in the literature of community detection \cite{abbe_2018_community}. In fact, a direct adaptation of the bipartite non-backtracking operator we define here leads to a simpler spectral algorithm for the detection of the bipartite stochastic block model studied in \cite{florescu2016spectral}. 
Below, We summarize our results informally. See Theorems~\ref{thm:eigenvalues} and \ref{thm:eigenvectors} for the precise statements.
\begin{theorem}[long matrix completion, informal]\label{thm:informal}
    With $O(\sqrt{mn})$ many samples, if the left singular vectors of a long $n\times m$ matrix $(m\gg n)$ are sufficiently delocalized, and the singular values are above a specific Kesten-Stigum threshold,  then the non-backtracking wedge operator can recover the ground truth singular values and find correlated estimates of the left singular vectors.
\end{theorem}

The \textit{Kesten-Stigum threshold} in Theorem~\ref{thm:informal} is analogous to a similar threshold for community detection in the stochastic block model, which characterizes a threshold of the signal-to-noise ratio for the existence of polynomial time algorithms to achieve weak recovery; see \cite{abbe_2018_community} for more details.  We discuss the Kesten-Stigum bound in matrix completion in Section~\ref{sec:discussion}.

We then apply our main results to the low-rank tensor completion problem under an orthogonality assumption on the $r$ components in a tensor decomposition. This is the first efficient algorithm to achieve weak recovery of a low-rank tensor with $\Omega(n^{k/2})$ samples without any extra $\mathrm{polylog}(n)$ factors. In particular, an explicit  $n^{k/2}$ threshold is obtained for rank-1  tensor completion when the entries are $\pm 1$. The weak recovery results can serve as a good spectral initialization step for other optimization procedures \cite{montanari.sun_2018_spectral,jain2014provable,cai2021nonconvex}. Unlike the common approach of unfolding a tensor to a matrix as square as possible ( i.e., applying $\mathrm{unfold}_{\lfloor k/2\rfloor, \lceil k/2\rceil}$ to $T$) in \cite{mu2014square,montanari.sun_2018_spectral}, we unfold $T$ to an $n \times n^{k-1}$ matrix in the most unbalanced way, and this allows us to directly estimate individual components of the tensor decomposition with sample complexity $O(n^{k/2})$. Note that simply applying the results in \cite{bordenave2020detection} to $\mathrm{unfold}_{\lfloor k/2\rfloor, \lceil k/2\rceil}(T)$ will have sample complexity $O(n^{\lceil k/2\rceil})$, and one cannot directly estimate any individual component in $T$. Therefore, our algorithm has a better computational complexity for tensor completion for all integer values $k$. The long matrix completion result also has a direct application in low-multilinear-rank tensor completion. Detailed statements and discussion on tensor completion are given in Sections \ref{sec:tensor_completion} and \ref{sec:tucker}.

Although there is a strong connection between the detection problem for matrix completion and community detection in the stochastic block model \cite{angelini.caltagirone.ea_2015_spectral,stephan2020non,bordenave2018nonbacktracking,bordenave2020detection}, the statistical-to-computational gap in the hypergraph stochastic block model behaves very differently from tensor completion \cite{Pal_2021,stephan2022sparse,gu2023weak}. In hypergraph community detection, efficient spectral methods can achieve weak recovery with the sample complexity down to the information-theoretical threshold up to a constant factor in all cases. However, in tensor completion, there is a conjectured diverging gap between the information-theoretical sample complexity and sample complexity for efficient algorithms \cite{barak.moitra_2016_noisy}.

The non-backtracking wedge operator we defined in Section~\ref{sec:setting} is tailored for the situation where the underlying sparse bipartite graph has a diverging aspect ratio. It is similar in spirit to the non-backtracking operator for hypergraphs \cite{angelini.caltagirone.ea_2015_spectral,liu2022statistical,stephan2022sparse}, and the nomadic walk operator in \cite{mohanty2020sdp,guruswami2022l_p} for (signed) random bipartite biregular graphs. The main difference compared to \cite{guruswami2022l_p,mohanty2020sdp} is that we work with a non-centered random matrix model without degree homogeneity, and we only count one-sided non-backtracking walks. In addition, a detailed eigenvector analysis is provided while \cite{guruswami2022l_p,mohanty2020sdp} are only about extreme eigenvalues. In terms of tensor estimation with path-counting arguments, another relevant work is the estimation of a spiked tensor model under heavy-tail noise \cite{ding.hopkins.ea_2020_estimating}, where the authors use  $n\times n$ symmetric
matrices to count the number of self-avoiding walks of length $c\log n$. However, those matrices are usually more expensive
to compute and require a new
parameter (the length of the paths considered) that needs tuning in practice.

\section{Setting and main results}\label{sec:main}

\subsection{Detailed setting for long matrix completion}\label{sec:setting}
We consider a rectangular matrix $M$ of size $n \times m$, with aspect ratio $\alpha = m/n$. We write the singular value decomposition of $M$ as
\begin{equation}\label{eq:SVD}
    M = \sum_{i = 1}^n \nu_i \phi_i \psi_i^\top, \notag
\end{equation}
where the $\nu_i$ are the singular values of $M$, and the $(\phi_i)_{i\in[r]}$ (resp. the $(\psi_i)_{i\in[r]}$) are an orthonormal family of its left (resp. right) singular vectors. We order the singular values of $M$ in decreasing order, that is
$ \nu_1 \geq \dots \geq \nu_n.$
Our observed matrix is of the form
\begin{equation}\label{eq:weighted_A}
    A = \frac{\sqrt{mn}}d (X \circ M),  
\end{equation}
where $X$ is an i.i.d Bernoulli matrix:
\begin{equation}
    X_{ij} \stackrel{i.i.d}{\sim} \  \mathrm{Ber}\left( \frac{d}{\sqrt{mn}}\right). \notag 
\end{equation}
We will use the shortcut $p = d/\sqrt{mn}$ as the sampling probability. The scaling of $A$ is chosen so that the expected value of $A$ is exactly $M$. Our goal is to weakly recover the eigendecomposition: \[MM^\top=\sum_{i=1}^n \nu_i^2 \phi_i \phi_i^\top.\]

\paragraph{The non-backtracking wedge matrix}
We can view the matrix $A$ as the weighted biadjacency matrix of a bipartite graph $G = (V_1 \cup V_2, E)$ with $|V_1| = n$, $|V_2| = m$. We call $V_1$ a \textit{left} vertex set and $V_2$ a \textit{right} vertex set. A path of length 2 in $G$ denoted by $(e_1,e_2,e_3)$ is called a \textit{wedge} if $e_1,e_3\in V_1, e_2\in V_2,$ and $e_1\not=e_3$. Define the set of \emph{oriented wedges} $\vec{E}$ as
\begin{equation}\label{eq:wedge}
    \vec{E} = \Set*{(e_1, e_2, e_3) \in V_1 \times V_2 \times V_1 : \{e_1, e_2\} \in E, \{e_2, e_3\} \in E, e_1 \neq e_3}.
\end{equation}
The \emph{non-backtracking wedge} matrix $B$ associated to $G$ is defined on $\vec{E} \times \vec{E}$ by
\begin{equation}\label{eq:defB}
    B_{ef} = \ind_{e \to f} A_{f_1 f_2} A_{f_3 f_2}, \quad \text{where} \quad \ind_{e \to f} = \ind_{e_3 = f_1} \ind_{e_2 \neq f_2}. 
\end{equation}
Equivalently, $\ind_{e\to f}$ is one if and only if $(e_1, e_2, e_3, f_2, f_3)$ is a non-backtracking walk of length 4 in $G$. See Figure~\ref{fig:NBW} for an example. Note that the non-backtracking wedge matrix $B$ is not symmetric: its index set
is the oriented wedge set defined in \eqref{eq:wedge}, and $B_{ef} \not= B_{fe}$ in \eqref{eq:defB}.

\begin{figure}
    \centering
    \begin{tikzpicture}[scale=0.60]
       \node[draw,circle, fill=orange](A) at (0,0) {$e_1$};
       \node[draw, minimum size=.6cm, fill=green](B) at (2, 3) {$e_2$};
       \node[draw,circle, fill=orange](C) at (4,0) {$e_3$};
       \node[draw, minimum size=.6cm, fill=green](D) at (6, 3) {$f_2$};
       \node[draw, circle, minimum size=.6cm, fill=orange](E) at (8, 0) {$f_3$};
       \draw [-{Latex}] (A) -- (B);
       \draw [-{Latex}] (B) -- (C);
       \draw [-{Latex}] (C) -- (D);
       \draw [-{Latex}] (D) -- (E);
    \end{tikzpicture}
    \caption{Example of a non-backtracking path between two wedges $(e_1, e_2, e_3)$ and $(f_1, f_2, f_3)$.}\label{fig:NBW}
\end{figure}
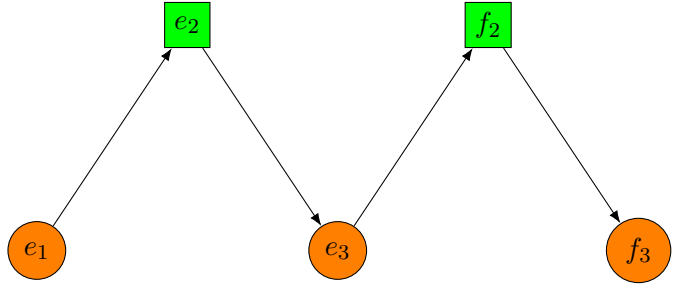

For comparison, we recall the definition of the \textit{non-backtracking operator} for weighted graphs as follows. Let $A$ be a weighted adjacency matrix of a graph $G=(V,E)$. Define the oriented edge set $\vec{E}$ for $G$ as $\vec{E}= \{ (i,j): \{ i,j\} \in E\}$. 
      The \textit{non-backtracking operator} $B$ of associated with $A$ is a non-Hermitian operator of size $|\vec E| \times |\vec E|$. For any $(u,v), (x,y)\in \vec{E}$, $B$ is defined as 
      \begin{align}\label{eq:standardB}
B_{(u,v),(x,y)} = \begin{cases}
     A_{xy} &\text{$v=x$, $u \neq y$}, \\
     0 &\text{otherwise.}
\end{cases}
\end{align} 
In the case of an $n\times m$ bipartite weighted graph defined  in \eqref{eq:weighted_A}, the non-backtracking wedge operator only counts non-backtracking walks starting from a vertex in $V_1$ while the non-backtracking operator defined in \eqref{eq:standardB} counts non-backtracking walks from both $V_1$ and $V_2$.

\paragraph{Intuition for the non-backtracking wedge matrix}
It has been shown that the non-backtracking operator, as a non-hermitian matrix, outperforms the standard (weighted) adjacency matrix when the sample size is very small in many graph clustering and matrix completion problems \cite{bordenave2018nonbacktracking,bordenave2020detection,stephan2020non,stephan2022sparse}. The key difference is that the spectrum of the non-backtracking operator is more informative and stable when the underlying graph is very sparse with a constant average degree. It is also a general phenomenon that the spectrum of non-hermitian random matrices is less sensitive to a row or column with a large $\ell_2$-norm \cite{coste2023sparse,bordenave2022convergence}.

In terms of the sample size, the non-backtracking matrix in \eqref{eq:standardB}  was used in \cite{bordenave2020detection,stephan2020non} to achieve the optimal scale $O(n)$ for square matrix completion. However, for a low-rank tensor in a matrix form of size $n\times n^2$, the work \cite{bordenave2020detection} only achieves weak recovery with sample size $O(n^2)$, while with the non-backtracking wedge operator, we only need $O(n^{1.5})$. The standard way in \cite{montanari.sun_2018_spectral,cai2021nonconvex} to obtain $\tilde O(n^{1.5})$ sample size is to apply a spectral method on the hollowed Gram matrix $Z=AA^T-\text{diag}(AA^T)$. For $i\not=j$, we have \[Z_{ij} =  \sum_{\text{wedge } i \to k \to j} A_{ik} A_{kj}.\] 
Hence, the spectrum of $Z$ depends on wedge walk counts. Our new non-backtracking operator only counts the wedge walks that are \textit{non-backtracking}, which avoids the localization effect from high-degree vertices \cite{benaychgeorges.bordenave.ea_2019_largest} in the weighted graph corresponding to $Z$.  This new operator introduces a  tool to  ``regularize'' the spectrum of a heavy-tailed long random matrix without any tuning parameter, which might be useful for other applications.

\subsection{Quantifying the convergence}
\label{sec:parameters}
The alignment between the spectral decomposition of $B$ and the SVD of $A$ will be governed by a wealth of different quantities. For ease of understanding, we divide those parameters into two categories, depending on their role in the convergence.

\paragraph{Threshold parameters} These parameters describe which parts of the SVD of $M$ will be reflected in the spectral decomposition of $B$.
\begin{enumerate}
    \item \textbf{Variance parameter}:
    \begin{equation}\label{eq:def_Q_rho}
        \rho = \norm{Q}, \quad \text{where} \quad Q = \sqrt{mn}(M \circ M).
    \end{equation}

    \item \textbf{Amplitude parameter}:
    \begin{equation}
        L = \sqrt{mn} \max_{x\in [n], y \in [m]} |M_{xy}| . \notag 
    \end{equation}

    \item \textbf{Detection thresholds}
    \begin{equation}
        \theta = \max(\theta_1, \theta_2), \notag
    \end{equation}
    where the two parameters $\theta_1$ and $\theta_2$ are defined as
    \begin{equation}
     \theta_1 = \sqrt{\frac{\rho}{d}} \quand    \theta_2 =  \frac L d . \notag
    \end{equation}

    \item \textbf{Detection rank}:
    \begin{equation}
        r_0 = \max \Set{i\in [n]: \nu_{i} > \theta}. \notag
    \end{equation}
    Equivalently, it is the number of singular values of $M$ above the detection threshold $\theta$.

    \item \textbf{Inverse signal-to-noise ratios (inverse SNRs)}: for $i \in [r_0]$,
    \begin{equation}
        \tau_i = \frac{\theta}{\nu_i} \in (0, 1). \notag
    \end{equation}
    The largest of those inverse SNRs, $\tau_{r_0}$, will be denoted by $\tau$.
\end{enumerate}

\paragraph{Complexity parameters} Those parameters quantify how intrinsically ``easy'' the matrix $M$ is to recover. In particular, contrary to the previous ones, they are all invariant under any rescaling of $M$.
\begin{enumerate}
    \item \textbf{Rescaled amplitude parameter}:
    \begin{equation}\label{eq:def_K}
        K = \frac{L^2}{\rho}.
    \end{equation}

    \item \textbf{Aspect log-ratio}:
    \begin{equation}\label{eq:def_eta}
        \eta = \frac{\log(m)}{\log(n)}-1.
    \end{equation}
    Equivalently, it is the solution of the equation $m = n^{1+\eta}$. When $M$ is obtained from the $(a, b)$ unfolding of a tensor, we get $\eta = \frac{b-a}{a}$.   
    \item \textbf{Rank}:
    \begin{equation}
        r = \mathrm{rank}(M) = \max \Set{i\in [n]: \nu_{i} \neq 0}. \notag
    \end{equation}

    \item \textbf{Delocalization parameter}:
    \begin{equation}
        \kappa = \sqrt{n} \max_i \norm{\phi_i}_\infty. \notag
    \end{equation}
    Note that $\kappa\geq 1$ by definition. We require no such conditions on the right singular vectors $\psi_i$.
\end{enumerate}

\paragraph{Simplifying assumptions} We also make a few simplifying assumptions throughout the paper; although our results are still valid when those assumptions aren't satisfied, the proofs become much messier. 
Recall the aspect ratio is defined by $\alpha=\frac{m}{n}\geq 1$.  We first assume that
\begin{equation}\label{eq:d_lowerbound}
   \liminf_{n} d>1 \quand d\leq n.
\end{equation}
 We also assume that 
\begin{equation}\label{eq:assumption_alpha}
    \eta \geq \frac{2\log(d)+2\log\log_d(n)}{\log(n)}, \quad \text{or equivalently} \quad \alpha \geq d^2(\log_d n)^2.
\end{equation} 
We note that if  $d>n$, our results become trivially true (see the definition of $\ell$ in Theorem \ref{thm:eigenvalues}, and the bounds thereafter).

\subsection{Main results for long matrix completion}\label{sec:main_matrix}

Our first result shows that the eigenvalues of $B$ exhibit a structure similar to those observed in \cite{stephan2020non, bordenave2020detection}: namely, most eigenvalues of $B$ lie in a circle of known radius, with a few outlier eigenvalues correlated with the singular values of $M$.

\begin{theorem}[Singular value estimation]\label{thm:eigenvalues}
    Let
    \begin{equation}
        \ell = \lfloor\eps \log_d(n) \rfloor, \quad \text{with} \quad \eps = \frac{(\eta / 2) \wedge 1}{25}. \notag 
    \end{equation}
    There exist constants $C_0,n_0\geq 1 $ that depend polynomially on $r, K,\kappa, d, \log n$  such that, if $n \geq n_0$ and $C_0 \tau^{2\ell} \leq 1$, the following holds on an event with probability at least $1 - cd^4n^{-1/4}$: 
    \begin{enumerate}
        \item There exists an ordering $\lambda_1, \dots, \lambda_{r_0}$ of the top $r_0$ eigenvalues of $B$ such that for any $i \in [r_0]$,
    \begin{equation}
        |\lambda_i - \nu_{i}^2| \leq C_0 \nu_{i}^2\tau_i^{2\ell}.  \notag 
    \end{equation}
    \item All other eigenvalues of $B$ have modulus at most $C_0^{1/\ell}\theta^2$. \label{thm:radius_bound}
    \end{enumerate}
\end{theorem}

Our second result concerns quantifying the overlaps between the top eigenvectors of $B$ and the singular vectors of $M$. For $i \in [r_0]$, we define the relative eigengap
    \begin{equation}\label{eq:delta}
        \delta_{i, \ell} = \min_{\nu_j \neq \nu_i} \left| 1 - \left( \frac{\nu_j}{\nu_i} \right)^{2\ell}\right|.
    \end{equation}
    Let $\xi_i^R, \xi_i^L$ a right (resp. left) unit eigenvector of $B$ associated with $\lambda_i$, and define
    \begin{equation}\label{eq:reduced_eigenvec}
         \zeta_i^R(x) := \sum_{e\in \vec{E}: e_1 = x} A_{e_1 e_2} A_{e_3 e_2} \xi_i^R(e) \quand \zeta_i^L(x) := \sum_{e\in \vec{E}: e_1 = x}\xi_i^L(e).
    \end{equation}

\begin{theorem}[Left singular vector estimation]\label{thm:eigenvectors}
    Under the same setting and conditions as Theorem \ref{thm:eigenvalues},  with probability at least $1-cd^4n^{-1/4}$, for $i\in [r_0]$, there exists a unit left singular vector $\phi_i'$ of $M$ associated with $\nu_i$ such that
    \begin{equation}
        \left| \langle  \zeta_i^R, \phi_i' \rangle - \frac{1}{\sqrt{\gamma_i}} \right| \leq \frac{ C_0 \tau_i^{2\ell}}{\delta_{i, \ell}} \quand \left| \langle  \zeta_i^L, \phi_i' \rangle - \frac{1}{\sqrt{\gamma_i}} \right| \leq \frac{C_0  \tau_i^{2\ell}}{\delta_{i, \ell}}, \label{eq:overlap_gamma}
    \end{equation} 
    where $\gamma_i$ is defined as 
    \begin{equation}\label{eq:def_gamma}
        \gamma_i = \left\langle \ind, \left(I - \frac{QQ^*}{\nu_i^4 d^2} \right)^{-1} \left( \phi_i \circ \phi_i \right) \right\rangle.
    \end{equation}
\end{theorem}

\subsection{Discussion}\label{sec:discussion}

\paragraph{Comparison to \cite{bordenave2020detection}} In the seminal work of \cite{bordenave2020detection}, the authors provide a spectral algorithm to weakly recover a low-rank rectangular matrix of size $m\times n$ with $O(m+n)$ many samples, where a Kesten-Stigum bound for estimation singular values and both left and right eigenvectors are obtained. In our low-rank long matrix setting $(m\gg n)$, their sample complexity scales as $O(m)$. In contrast, our non-backtracking wedge operator is able to estimate singular values and left singular vectors of a long matrix with sample size $O(\sqrt{mn})$. This improvement in sample complexity to recover one-sided structure is useful in  tensor completion; see Section~\ref{sec:tensor_completion} for more details.

\paragraph{Overlap convergence and weak consistency} In full generality, the overlaps $\gamma_i$ in Theorem \ref{thm:eigenvectors} have a complicated expression. however, this can be  simplified under a homogeneity assumption, namely
\begin{equation} 
    D_x := \sum_{y \in V_1, z\in V_2} Q_{xz} Q_{yz} \notag 
\end{equation}
is independent of $x$. In this case, the Perron-Frobenius theorem  implies that $D_x = \rho^2$, and hence
\begin{equation}
    \langle  \zeta_i^R, \phi_i \rangle=\sqrt{1 - \tau_i^4} + o(1), \quad  \langle  \zeta_i^L, \phi_i \rangle = \sqrt{1 - \tau_i^4} + o(1). \notag
\end{equation}
This implies a \textit{weak recovery} of the singular vectors $\phi_i$. In general, as $d \to \infty$, we have $(\gamma_i -1)\sim C/d^2$, and hence \eqref{eq:overlap_gamma} implies the eigenvectors of $B$ are asymptotically aligned with the left singular vectors of $M$.  Under the additional assumption that all singular values of $M$  are distinct, the top $r$ eigenvectors of $B$ are asymptotically aligned with all left singular vectors of $M$, and eigenvalues are consistent estimators for $\nu_i, i\in [r]$: this is a form of \textit{weak consistency} for matrix completion.

\paragraph{Optimality of the bounds and thresholds} We have made no effort towards getting closed-form exponents in the polynomials $C_0$ and $n_0$, which would probably not be sharp. An explicit bound on $C_0, n_0$ in the proportional case for matrix completion is worked out in \cite[Theorem 5.4]{bordenave2020detection}.  Similarly, although we do expect (when $d, \eta$ are fixed) that the rate of convergence is indeed of the form $n^{-c}$, we didn't try to optimize the value of such $c$ as much as possible.

Our work shows a threshold behavior for long matrix completion: one can weakly recover the components in   $MM^\top$ whose signal-to-noise ratio is above the threshold. 
On the other hand, the optimality of $\theta$ is still an open and interesting question since if this threshold can be lowered, more information about $M$ can be recovered. The first threshold $\theta_1$ is inherent to the problem and can be found in many similar problems (e.g., the \textit{Kesten-Stigum threshold} in community detection \cite{krzakala.moore.ea_2013_spectral,bordenave2018nonbacktracking,mossel_reconstruction_2015,mossel.neeman.ea_2018_proof}):  it is related to an intrinsic property on the existence of a pseudo-eigenvector in a  Galton-Watson tree with Poisson offspring distribution of mean $d$, see Section \ref{sec:tree_pseudo} for details.  But whether $\theta_2$ is needed or is an artifact of proof is still unclear. We note that in a similar problem regarding the configuration model \cite{coste_2021_spectral}, such a seemingly spurious threshold turns out to be the sharp one (albeit only for one eigenvalue). Note that as $d$ grows, $\theta_2$ decays faster than $\theta_1$, therefore $\theta=\theta_1$ for sufficiently large $d$.

\paragraph{Recovering right singular vectors} We draw attention to the fact that our result is almost completely independent of the right singular vectors $\psi_i$ of $M$, both on the level of assumptions (except for their influence on $K$) and results. We believe this is not an artifact of proof: indeed, most of the right vertices (a proportion $1 - \alpha^{-1}$ of them) are isolated, and it is impossible to recover any signal from those. This is a completely different behavior compared with \cite{bordenave2020detection}, where the authors show that when $\alpha$ is bounded, both left and right singular vectors of $M$ can be recovered. However, it matches recent results from \cite{montanari.wu_2022_fundamental} in the framework of long matrix reconstruction showing different reconstruction thresholds for the left and right singular vectors. 
Our work also shows that one can effectively estimate the rank of a long matrix of size $n\times m$ with $O(\sqrt{mn})$ random samples when the rank $r=O(1)$. When $m,n$ are proportional, a similar result was obtained in \cite{saade2015matrix} based on the Bethe-Hessian matrix.

\paragraph{Parameter scaling} All of the results in Theorems \ref{thm:eigenvalues} and \ref{thm:eigenvectors} are non-asymptotic, and with more care in the proof, we would be able to give explicit expressions for $C_0, n_0$ (see, e.g., \cite{bordenave2020detection}). However, the given bounds are meaningful only when $\ell \gg 1$, and $C_0 \tau^{2\ell} \ll 1$. These can easily be seen to be satisfied whenever
\begin{equation}
   K, \kappa, r = n^{o(1)}, \quand \log(d)^2 = o(\log(n)). \notag 
\end{equation}

\subsection{Application in low-CP-rank tensor completion}\label{sec:tensor_completion}

Throughout this part, our goal will be to recover a  tensor $T$ of order $k\geq 3$, which has  a CP decomposition (see, e.g., \cite{kolda2009tensor}) with $r$ components:
\begin{equation}\label{eq:low_cp_rank}
    T = \sum_{i=1}^r \nu_i \,w_i^{(1)}\otimes w_i^{(2)}\otimes\cdots \otimes w_i^{(k)},
\end{equation}
where $\|w_i^{(j)}\|_2=1$ for all $i\in [r], j\in [k]$. The smallest integer $r$ for the existence of the decomposition \eqref{eq:low_cp_rank} is called the \textit{CP-rank} of $T$, and we will call it the rank of $T$ for simplicity in this section.

It is known that for a general low-rank tensor $T$, finding its CP-decomposition is NP-hard \cite{hillar2013most}. Even with the complete information of the ground truth $T$, identifying each component in  $\{w_{i}^{(j)}, i\in [r], j\in [k]\}$ is not feasible. Since we consider asymmetric tensor completion, it will be useful to consider a general version of unfolding:
\begin{definition}
    Let $\cS$, $\cT$ be a partition of $[k]$ in two subsets. The $(\cS, \cT)$-unfolding of $T$; denoted as $\mathrm{unfold}_{\cS, \cT}(T)$, is a matrix $M \in \dR^{n^\cS \times n^{\cT}}$ defined as follows: for any multi-index $(i_s)_{s\in [k]}$,
    \begin{equation}
        M_{(i_s)_{s\in \cS}, (i_s)_{s\in \cT}} = T_{(i_s)_{s\in [k]}}. \notag 
    \end{equation}
\end{definition}
In particular, when $T$ is as in \eqref{eq:low_cp_rank}, we have
\begin{equation}\label{eq:low_rank_unfolding}
    \mathrm{unfold}_{\cS, \cT}(T) = \sum_{i = 1}^r \nu_i \left( \bigotimes_{j\in \cS} w_j \right) \left( \bigotimes_{j\in \cT} w_j\right)^*.
\end{equation}
For simplicity, we define the mode-$i$ unfolding of $T$ as  $\mathrm{unfold}_{i}(T) := \mathrm{unfold}_{\{i\}, [k] \setminus \{i\}}(T)$.

However, we make an important remark: Equation \eqref{eq:low_rank_unfolding} is not necessarily the SVD of the tensor unfolding. Indeed, the CP decomposition does not always have orthogonal components, and hence, the vectors in \eqref{eq:low_rank_unfolding} are not necessarily orthogonal.
Generally, for a given tensor $T$, the spectra of the unfolded matrices are not the same for different unfoldings, and there is a hierarchy of unfolded matrix norms \cite{wang2017operator}. In the following, we will denote by $\tilde T$ an observed tensor from $T$ with sampling probability $p = \frac{d}{n^{k/2}}$. For an unfolding mode $i$, we let $M^{(i)} = \mathrm{unfold}_i(\tilde T)$, and $B^{(i)}$ be its associated non-backtracking wedge matrix.

\subsubsection{The rank-1 case} When $r = 1$,    we write 
\begin{align}\label{eq:rank-1-T}
    T=w^{(1)}\otimes \cdots \otimes w^{(k)},  
\end{align}
where we choose $\nu_1 = 1$ in \eqref{eq:low_cp_rank} for convenience. Then we have
\begin{align}
 \mathrm{unfold}_{i}(T) := \mathrm{unfold}_{\{i\}, [k] \setminus \{i\}}(T) =  w^{(i)} (w^{(1)} \otimes \cdots \otimes w^{(i-1)} \otimes w^{(i+1)} \otimes \cdots \otimes w^{(k)})^\top.  \notag 
\end{align}
We now write down the corresponding model parameters for $M^{(i)}$. It can be checked that they  are the same for all $i$, independent of the choice for unfolding, except for $\kappa$:
\begin{align}
    \theta_1 &= \sqrt{\frac{n^{k/2} \prod_{j=1}^k \norm*{w^{(j)}}_4^2}d}, &\theta_2 &= \frac{n^{k/2}\prod_{j=1}^k \norm*{w^{(j)}}_\infty}d \label{eq:tensor_theta}\\
      K &= n^{k/2} \prod_{j=1}^k \frac{\norm{w^{(j)}}_{\infty}^2}{\norm{w^{(j)}}_4^2}, 
     &\kappa^{(i)} &=\sqrt{n}  \|w^{(i)}\|_{\infty} \leq\sqrt{n}  \max_i \|w^{(i)}\|_{\infty} \eqqcolon \kappa_{\mathrm{max}}.\notag 
\end{align}
Note that we always have $\theta_1^2\leq \theta_2$. Applying Theorems \ref{thm:eigenvalues} and \ref{thm:eigenvectors} to $\mathrm{unfold}_{i}(T)$, we obtain the following corollary:
\begin{corollary}[Rank-1 tensor completion] \label{cor:rank-1} Let $k\geq 3$, and define $\ell$ as
    \begin{equation}
        \ell = \left\lfloor c_1 \log_d(n) \right\rfloor, \quad c_1=\frac{(k/2-1)\wedge 1}{25}. \notag 
    \end{equation} For any constant $\eps>0$,
    assume that
    \begin{equation}\label{eq:rank1_d_condition}
        d > (1+\eps) n^{k/2}  \prod_{j=1}^k \norm*{w^{(j)}}_{\infty}.
    \end{equation}
     Then there exists $C_0, n_0\geq 1$ that depend polynomially on $ K,\kappa_{\mathrm{max}}, d, \log n$ such that for any $i \in [k]$, $n\geq n_0$ and $C_0 \theta^{2\ell}\leq 1$, the following holds on an event with probability at least $1 - cd^4n^{-1/4}$: 
    \begin{enumerate}
        \item The top eigenvalue $\lambda^{(i)}$ of  $B^{(i)}$   satisfies \begin{equation}
        |\lambda^{(i)} - 1| \leq C_0\theta^{2\ell},  \notag 
    \end{equation}
    where $\theta = \max(\theta_1, \theta_2)$ with $\theta_1, \theta_2$ in \eqref{eq:tensor_theta}.
  \item All other eigenvalues of $B^{(i)}$ have modulus at most $C_0^{1/\ell} \theta^2$.
 \item There exist associated left and right eigenvectors $\xi^{(i), L}, \xi^{(i), R}$ of $B^{(i)}$ such  that 
    \begin{equation}
        \left| \langle   \zeta^{(i), L}, w^{(i)} \rangle - \sqrt{1 - \theta^2} \right| \leq  C_0 \theta^{2\ell} \quand \left| \langle    \zeta^{(i), R}, w^{(i)} \rangle - \sqrt{1 - \theta^2} \right| \leq  C_0 \theta^{2\ell},
    \end{equation}
    where $\zeta^{(i), L},\zeta^{(i), R}$ are defined in \eqref{eq:reduced_eigenvec}.
    \end{enumerate}
\end{corollary}

Hence, successively applying tensor unfolding $k$ times to the observed tensor is enough to estimate all of the components of $T$. 
When all components $w_1,\dots, w_k$ are vectors in $\frac{1}{\sqrt n} \{ \pm 1\}^n$, with $(1+\eps)n^{k/2}$ samples,  we have 
\[ \langle \zeta^{(i),L}, w^{(i)}\rangle =\frac{\sqrt{(1+\eps)^2-1}}{1+\eps}+O(n^{-c}).\]

\subsubsection{Low-rank orthogonal components estimation}

Our long matrix completion results aim to estimate individual left singular vectors above a Kesten-Stigum-type threshold. However, in the tensor case, without orthogonal assumption on the vectors $[w_{i}^{(1)}, \dots, w_{i}^{(k)}]$, such a reduction to the estimation of left singular vectors in long matrix completion is impossible. 
Therefore, we consider the same setting as in \cite{jain2014provable,potechin2017exact}, and assume that $T$ can be written as
\begin{equation}\label{eq:orthogonal_condition}
    T = \sum_{j=1}^r \nu_j w_j^{(1)}\otimes w_j^{(2)}\otimes\cdots \otimes w_j^{(k)},
\end{equation}
where for each $i\in [k]$,
$\left\{ w_{1}^{(i)}, \dots, w_{r}^{(i)}\right\}$ forms an orthonormal basis, and $\nu_1\geq \nu_2\geq \cdots \geq \nu_{r}>0$. In contrast with the rank-one case, we don't have access to an orthogonal CP-decomposition of $T \circ T$, and hence most parameters used in Theorem \ref{thm:eigenvalues} depend on the unfolding mode $i$; we shall denote them by $\rho^{(i)}, Q^{(i)}$, and so on. We can still provide the following estimates of $\rho^{(i)}$ and $L^{(i)}$:
\begin{equation}
    n^{k/2}\norm{T}_2 \leq \rho^{(i)} \leq n^{k/2}\norm{T}_F, \quand L^{(i)} = L \coloneqq n^{k/2} \norm{T}_\infty. \notag 
\end{equation}

\begin{corollary}[Low-rank orthogonal tensor completion]\label{cor:orthogonal}
    Let $T$ be a $k$-tensor with an orthogonal decomposition \eqref{eq:orthogonal_condition} and $k\geq 3$, and define $\ell$ as
    \begin{equation}
        \ell = \left\lfloor c_1 \log_d(n) \right\rfloor, \quad  c_1=\frac{(k/2-1)\wedge 1}{25}. \notag 
    \end{equation}
    For any $i \in [k]$, there exists  $C_0^{(i)}, n_0^{(i)}$, that depend  polynomially on $K^{(i)}, \kappa^{(i)}, r^{(i)}, d^{(i)}, \log(n)$, such that, if $n \geq n_0$ and $C_0^{(i)} \left(\tau_j^{(i)}\right)^{2\ell} \leq 1$, the following holds    with probability at least $1 - cd^4n^{-1/4}$: 
 \begin{enumerate}
        \item There exists an ordering $\lambda_1^{(i)}, \dots, \lambda_{r_0}^{(i)}$ of the top $r_0$ eigenvalues of $B^{(i)}$ such that for any $j \in [r_0^{(i)}]$,
    \begin{equation}
        |\lambda_j - \nu_{j}^2| \leq C_0^{(i)} \nu_{j}^2\left(\tau_j^{(i)}\right)^{2\ell}. \notag 
    \end{equation}
    \item All other eigenvalues of $B^{(i)}$ have modulus at most $\left(C_0^{(i)}\right)^{1/\ell}(\theta^{(i)})^2$.
    \item For $j\in [r_0^{(i)}]$, there exist associated left and right eigenvectors $\xi_j^{ L}, \xi_j^{ R}$ of $B^{(i)}$  such that
    \begin{align*}
        \left| \langle \zeta_j^{(i),R}, w_{j}^{(i)} \rangle - \frac{1}{\sqrt{\gamma_j^{(i)}}} \right| &\leq \frac{ C_0^{(i)} \left(\tau_j^{(i)}\right)^{2\ell}}{\delta_{j, \ell}}, \quad 
        \left| \langle \zeta_j^{(i),L}, w_{j}^{(i)} \rangle - \frac{1}{\sqrt{\gamma_j^{(i)}}} \right|  \leq \frac{C_0^{(i)} \left(\tau_j^{(i)}\right)^{2\ell}}{\delta_{j, \ell}},
    \end{align*}
    where $\gamma_j$ is defined as 
    \begin{equation}
        \gamma_j^{(i)} = \left\langle \ind, \left(I - \frac{Q^{(i)} Q^{(i)*}}{\nu_j^4 d^2} \right)^{-1} \left( w_{j}^{(i)} \circ w_{j}^{(i)} \right) \right\rangle, \notag 
    \end{equation}
   $\delta_{j,\ell}$ is defined in \eqref{eq:delta}, and  $\zeta_j^{(i), R}, \zeta_j^{(i), L}$ are defined in \eqref{eq:reduced_eigenvec}.
    \end{enumerate}
\end{corollary}
We obtain a weak recovery of an orthogonal tensor decomposition with sample size $O(n^{k/2})$, which removes the $\mathrm{polylog}(n)$ factors in \cite{jain2014provable,potechin2017exact}. Same as in the discussion for long matrix completion in Section \ref{sec:discussion}, under the additional assumption that $\{\nu_i, i\in [r]\}$ are separated, with sample size $\omega(n^{k/2})$, weak consistency for  $T$ can be achieved:
\begin{theorem}\label{thm:frobenius_bound}
    Let $T$ have an orthogonal decomposition as in \eqref{eq:orthogonal_condition}, such that $\delta_{i, \ell} \geq c$ for some constant $c> 0$. For $i \in [k], j \in [r]$, define the estimator
    \begin{equation}
        \hat T = \sum_{j=1}^r \hat\nu_j \,\hat w_j^{(1)}\otimes\cdots \otimes \hat w_j^{(k)}\quad \text{for} \quad \hat w_j^{(i)} = \zeta_j^{(i), R} \quand \hat \nu_j = \sqrt{\lambda_i^{(1)}}. \notag 
    \end{equation}
    Then, with probability at least $1 - cd^4n^{-1/4}$, for large enough $d$, 
    \begin{equation}
        \frac{\norm{T - \hat T}_F}{\norm{T}_F} \leq \frac{\sqrt{k r} K\tau}{d} + C_1\left( \frac\tau d \right)^{\ell-1},  \quad      \text{where} \quad  \tau = \nu_r^{-2} \max_i \rho^{(i)}.\notag 
    \end{equation}

\end{theorem}
Theorem \ref{thm:frobenius_bound} is proved in Appendix \ref{sec:app:frobenius}. To the best of our knowledge, this is the first result that shows weak recovery (in the sense $\norm{T - \hat T}_F = o(1)$) in the regime $d \to \infty$, as opposed to $d \gtrsim \log(n)^c$.

\subsubsection{Discussion on tensor completion}

\paragraph{Comparison to nearly square unfolding} Corollaries \ref{cor:rank-1} and \ref{cor:orthogonal} provide an efficient algorithm that outputs an estimate for each component $w^{(i)}$. This directly implies that one can have a correlated estimation of $T$ with sample size $O(n^{k/2})$. This improves the sample complexity $O(n^{\lceil k/2 \rceil })$ when directly applying the matrix completion algorithm in \cite{bordenave2020detection} to $\mathrm{unfold}_{\lfloor k/2\rfloor, \lceil k/2\rceil}(T)$ when $k$ is odd. 
This mode-$i$ unfolding approach is more computationally efficient than the nearly square unfolding \cite{montanari.sun_2018_spectral} for component-wise estimation.  For example, in the rank-1 case, applying the matrix completion algorithm to  $\mathrm{unfold}_{\lfloor k/2\rfloor, \lceil k/2\rceil}(T)$, one can find a correlated estimation of $w^{(1)}\otimes \cdots \otimes  w^{(\lfloor k/2 \rfloor)}$. However, a further recursive unfolding is needed to recover individual components.  The advantage of single-mode unfolding was also discussed in \cite{arous2021long} for the spiked tensor model.

\paragraph{Optimal sample complexity for rank-one completion} We note that although our method does reach the $O(n^{k/2})$ scaling for rank-one tensor completion, this is not the optimal sample complexity. Indeed, for a boolean rank-one tensor, completion is equivalent to $k$-XORSAT, and hence \cite{creignou_approximating_2003} implies the following theorem:
\begin{theorem}\label{thm:rank_one_optimal}
    Assume that $T = x^{\otimes k}$ for a given vector $x \in \dR^n$ such that $\norm{x}_2 = 1$. Then there exists a constant $c_k < 1$ such that if $p \geq c_k n^{-(k-1)}$, there exists an algorithm that outputs an estimator $\hat x$ for $x$ satisfying
    \begin{equation}
        \frac{\langle \hat x, x \rangle}{\norm{\hat x}_2} \geq \frac1{\sqrt{n}\norm{x}_\infty}. \notag 
    \end{equation}
\end{theorem}
Hence, for the specific case of rank-one tensors with bounded entries, a near-optimal algorithm can reach a sample complexity of $O(n)$. However, the mapping to a XORSAT problem disappears for tensors of rank $ r > 1$, as well as with the addition of any type of noise (see, e.g., \cite{barak.moitra_2016_noisy}). Therefore, the conjectured $n^{k/2}$ threshold in \cite{barak.moitra_2016_noisy} for general low-rank tensor completion does not apply to this specific noiseless example. The proof of Theorem \ref{thm:rank_one_optimal} is given in Appendix \ref{sec:appendix_A}.

\paragraph{Subspace recovery for general low-rank tensors}
Without the orthogonal decomposition assumption of the tensor $T$, one cannot directly associate the singular vectors of the unfolded matrix to the components of $T$. In this general case, the best one can hope for is to recover the subspace spanned by the spikes from $T$.  Although our result is of independent interest in long matrix completion, it can also be used directly in non-orthogonal tensor completion. First, when $d \to \infty$, the results of \cite{montanari.sun_2018_spectral} imply that a good reconstruction of the left singular space of $M$ translates directly to weak consistency for tensor reconstruction. On the other hand, several algorithms such as the ones in \cite{montanari.sun_2018_spectral, liu2020tensor, cai2021nonconvex} make use of a first spectral estimate of the CP factors of the tensor $T$ before a refined optimization step; a popular choice for such an estimate is usually a truncated SVD of $A$. We conjecture that even for fixed $d$ when the SVD of $A$ contains no information about the one of $M$, the embedded eigenvectors of $B$ can be used to perform an informed initialization of those algorithms.

\subsection{Application in low-multilinear-rank tensor completion}\label{sec:tucker}
Besides CP-decomposition, another popular choice for tensor decomposition is the Tucker decomposition \cite{tucker1966some}, which is also called higher-order singular value decomposition (HOSVD) \cite{de2000multilinear}. For an order $k$-tensor $T$, we write its HOSVD as
\begin{align}\label{eq:HOSVD}
T=\sum_{q_1=1}^{r_1}\sum_{q_2=1}^{r_2}\cdots \sum_{q_k=1}^{r_k}s_{q_1,\dots,q_k} \phi_{q_1}^{(1)}\otimes \cdots \otimes \phi_{q_k}^{(k)},
\end{align}
where $(r_1,\dots,r_k)$ is called the \textit{multilinear rank} of $T$, $U_j=[\phi_1^{(j)},\dots, \phi_{r_j}^{(j)}]\in \mathbb R^{n\times r_j}$ is a matrix with orthonormal columns,   and $S\in \mathbb R^{r_1\times \cdots\times r_k}$ is the core tensor. One can show that the columns of $U^{(j)}$ are the left singular vectors of the mode-$j$ unfolding of $T$ denoted by $\mathrm{unfold}_j(T)\in \mathbb R^{n\times n^{k-1}}$. See \cite{kolda2009tensor} for more details.  
Tensor completion with low-multilinear-rank was studied in \cite{kressner2014low,yuan2017incoherent, montanari.sun_2018_spectral}.  Recovering an HOSVD  under additive noise was considered in \cite{lebeau2024random}.

We can apply the result for long matrix completion in Section~\ref{sec:main_matrix} and  to $\mathrm{unfold}_j(T)$  to recover the mode-$j$ components of the HOSVD.  We denote the SVD of $\mathrm{unfold}_j(T)$  as
\begin{align}
\mathrm{unfold}_j(T)=\sum_{i=1}^{r_j}\nu_i^{(j)}\phi_i^{(j)}\psi_i^{(j)}. \notag 
\end{align}
Let $B^{(j)}$ be the non-backtracking wedge operator for the observed matrix of $\mathrm{unfold}_j(T)$ after sampling each entry with probability $p=\frac{d}{n^{k/2}}$. The following corollary follows directly from Theorem~\ref{thm:eigenvectors}.
\begin{corollary}[Low-multilinear-rank tensor completion]\label{cor:tucker}
Let $T$ be a tensor  defined in \eqref{eq:HOSVD}.  With the same notations as in Section~\ref{sec:parameters} and Theorem~\ref{thm:eigenvectors} applied to $\mathrm{unfold}_j(T)$ and $B^{(j)}$, there exist constants $C_0,n_0\geq 1 $ that depend polynomially on $r, K,\kappa, d, \log n$  such that, if $n \geq n_0$ and $C_0 \tau^{2\ell} \leq 1$, the following holds on an event with probability at least $1 - cd^4n^{-1/4}$:  there exists a unit left singular vector ${\phi_i'}^{(j)}$ of $\mathrm{unfold}_j(T)$ associated with $\nu_i^{(j)}$ such that
    \begin{equation}
        \left| \langle  \zeta_i^R, \phi_i'^{(j)} \rangle - \frac{1}{\sqrt{\gamma_i}} \right| \leq \frac{ C_0 \tau_i^{2\ell}}{\delta_{i, \ell}} \quand \left| \langle  \zeta_i^L, {\phi_i'} ^{(j)}\rangle - \frac{1}{\sqrt{\gamma_i}} \right| \leq \frac{C_0  \tau_i^{2\ell}}{\delta_{i, \ell}}. \notag 
    \end{equation} 
\end{corollary}
By applying Corollary~\ref{cor:tucker} to different unfoldings, we can achieve weak recovery for the orthonormal components $U^{(j)}, 1\leq j\leq k$,  with $O(n^{k/2})$ many samples. Our method, however, does not provide estimations for the core tensor $S$, the same as the work on low-multilinear-rank tensor approximation \cite{lebeau2024random}.

\section*{Organization of the appendix} 
Numerical simulations are provided in Appendix~\ref{sec:simulation}.
Appendix~\ref{sec: structure} is devoted to stating the main technical ingredient of Theorems \ref{thm:eigenvalues} and \ref{thm:eigenvectors}. It consists of a detailed spectral structure of the matrix $B$ (or, more precisely, of its power $B^\ell$), with precise expressions given for the pseudo-eigenvectors of $B$ and their relationships with each other. We also provide a sketch of the proof of how this technical result implies our main theorems. 
  The rest of the sections are then devoted to proving Theorem \ref{thm:spectral_structure_Bl}. In short:
 \begin{enumerate}
 \item  In Appendix~\ref{sec:inequality}, we collect a few simple inequalities that will be useful in the proof.
     \item in Appendix \ref{sec:local_study}, we perform a local study of the neighborhoods of the weighted bipartite random graph $G$; we show a local convergence (in a stronger sense than the one of Benjamini and Schramm \cite{benjamini.schramm_2011_recurrence}) of $G$ to an appropriately defined  Galton-Watson tree $T$. The structure of this tree is tailored to our bipartite setting and is similar to the one of \cite{florescu2016spectral}. In short, the tree $T$ encodes the ``local” properties of $G$, in the sense that most neighborhoods of $G$ have the same distribution as the tree $T$.
     
     \item in Appendix \ref{sec:tree_pseudo}, we study this tree in-depth to extract information about specific processes on $T$. Those processes give rise to the pseudo-eigenvectors of $B$. The main ingredient of this Appendix is a reduction to the setting of \cite{stephan2020non}, which approximates the bipartite graph $G$ with a randomly weighted (but non-bipartite) version.
     
     \item Appendices~\ref{sec:bulk_radius} and ~\ref{sec: trace_method} are devoted to bounding the bulk eigenvalues of $G$. The method used -- a so-called \emph{tangle-free} decomposition and a variant of the trace method \cite{furedi1981eigenvalues} -- is similar to the one of \cite{bordenave2018nonbacktracking, bordenave2020detection, stephan2020non, stephan2022sparse}. One of the major technical novelty is in the trace method proof of Theorem \ref{thm:eigenvalues}\ref{thm:radius_bound}. The  $B$ matrix counts non-backtracking wedges, but the appearance of different wedges in a random graph $G$ is dependent. It is a big challenge to estimate the high-power trace of a random matrix with dependent entries compared to the standard moment method proof techniques in the random matrix theory literature, and such a technical challenge does not appear in \cite{bordenave2018nonbacktracking, bordenave2020detection, stephan2020non}.  Additional proofs based on the trace method are included in Appendix~\ref{sec:additional_trace}.
 \end{enumerate}
 Finally, we include the proof of Theorem~\ref{thm:frobenius_bound} in Appendix~\ref{sec:app:frobenius}, and the proof of Theorem~\ref{thm:rank_one_optimal} in Appendix~\ref{sec:appendix_A}.

\subsection*{Acknowledgments}Y.Z. is partially supported by NSF-Simons Research Collaborations on the Mathematical and Scientific
Foundations of Deep Learning and an AMS-Simons Travel Grant.

\bibliographystyle{plain}
\bibliography{db}

\appendix


\section{Numerical experiments} \label{sec:simulation}

To illustrate our results, we considered a rectangular matrix $M$ of size $n \times m$, of the form
\begin{equation}\label{eq:example_M}
    M = \frac{1}{\sqrt{mn}} \left(\nu_1 \, \ind_n \ind_m^\top + \nu_2 \, u v^\top \right)
\end{equation} 
where $u$ and $v$ are vectors with entries in $\{-1, 1\}$ orthogonal to $\ind_n$ and $\ind_m$, respectively. In this setting, the thresholds $\theta_1$ and $\theta_2$ are 
$\theta_1 = \sqrt{\frac{\nu_1^2 + \nu_2^2}{d}}$ and  $\theta_2 = \frac{\nu_1 + \nu_2}d$.
For all our experiments, we fixed $d = 3, \nu_1 = 1, \nu_2 = 0.9$, and $m = n^2$ to mimic the unfolding of a tensor of order 3.

The left part of Figure \ref{fig:spectrum_nb} shows the spectrum of the non-backtracking wedge matrix $B$ constructed as in \eqref{eq:defB}, with $n = 4000$. The structure shown in Theorem \ref{thm:eigenvalues} is easily apparent: all but two of the eigenvalues are confined in a circle of radius close to $\theta^2$, with two real \emph{outlier} eigenvalues around $\nu_1^2$ and $\nu_2^2$, respectively. In the right part, we plot the reduced eigenvector $\zeta_2^L$ associated with $\lambda_2$ for $n = 30000$ against the theoretical vector $u$. We can check that $\frac{\langle \sign(\zeta_2^L), u \rangle}n \approx 0.845$,
which corresponds to recovering 92.3\% of the entries of $u$. On the other hand, Figure \ref{fig:baseline} shows the same plots for the matrix $Z=AA^T-\mathrm{diag}(AA^T)$, widely used in tensor completion works \cite{montanari.sun_2018_spectral, cai2021nonconvex}. The spectrum of $Z$ is less well-behaved than the one of $B$, and in particular, there is no visible outlier. This is confirmed by plotting the second eigenvector $\xi_2$ of $Z$ against $u$: the former is very localized, with no visible correlation with $u$. This time, the sign correlation between $\xi_2$ and $u$ is 
$\frac{\langle \sign(\xi_2), u \rangle}n \approx 0.01 \approx \frac{1.65}{\sqrt{n}}$,
which is indistinguishable from the variance of a random guess.
\begin{figure}
    \centering
    \includegraphics[width=0.49\textwidth]{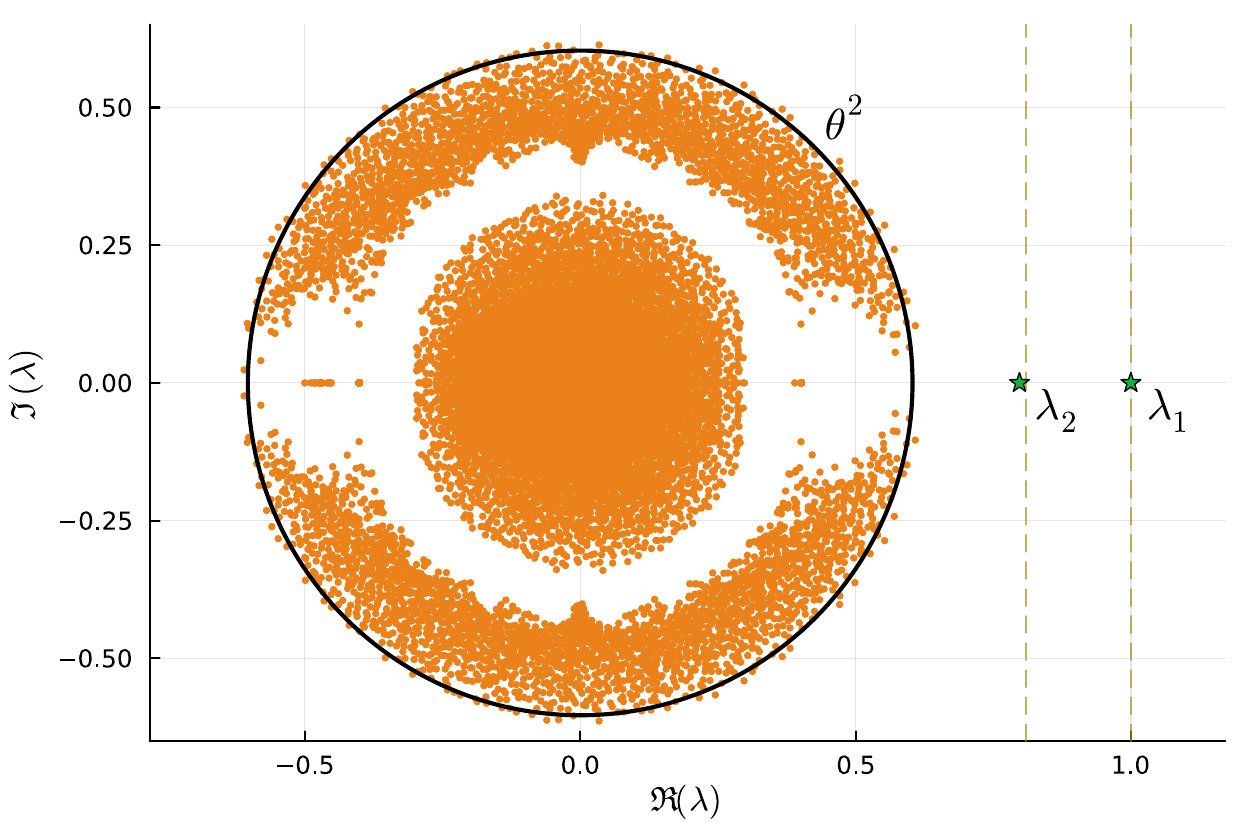}
        \includegraphics[width=0.49\textwidth]{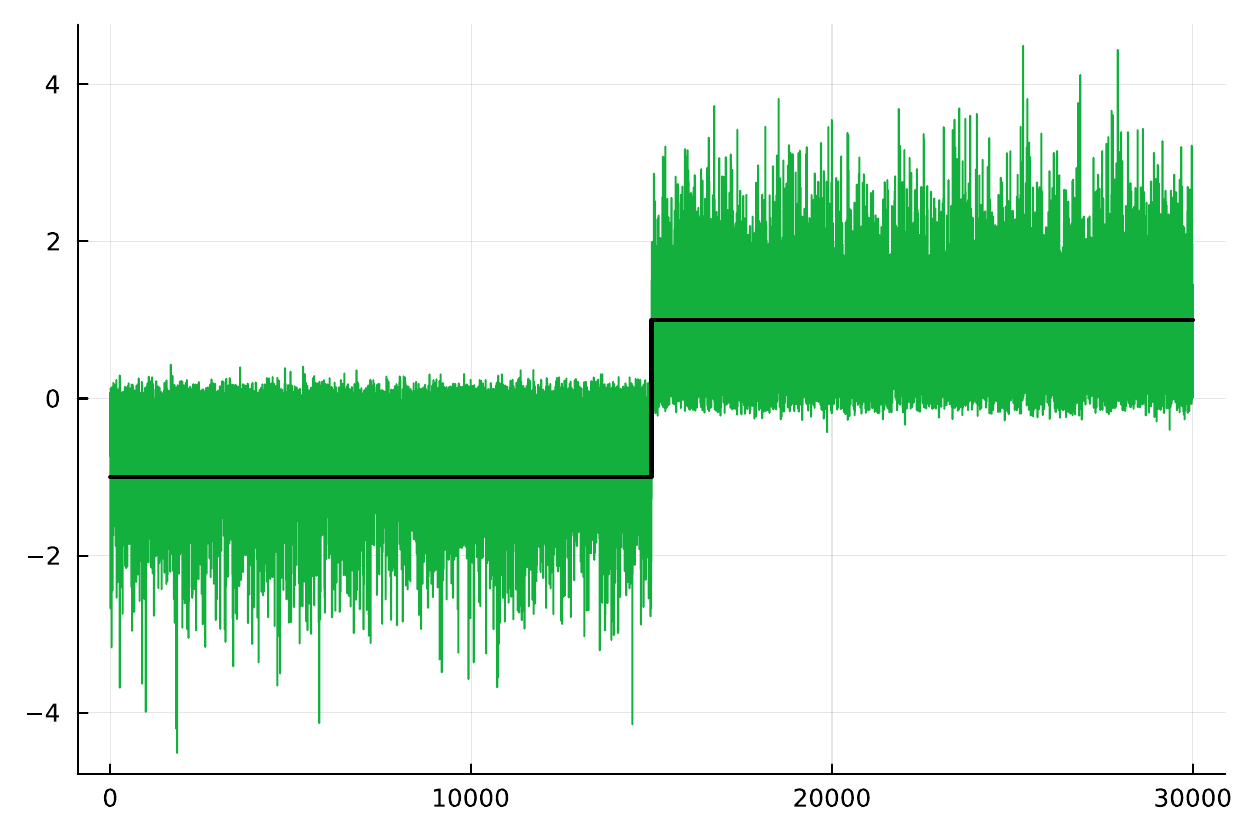}
    \caption{\textbf{Left:} Spectrum of the non-backtracking wedge matrix of the matrix $M$ in \eqref{eq:example_M}. The bulk eigenvalues are in orange, while the outliers $\lambda_1$ and $\lambda_2$ are in green. The circle of radius $\theta^2$ and the theoretical locations at $\nu_1^2$ and $\nu_2^2$ are plotted in solid black and dashed light purple, respectively.
    \textbf{Right:} Reduced second eigenvector $\zeta_2^L$ of $B$ (see \eqref{eq:reduced_eigenvec} for the embedding), in green. The theoretical singular vector $u$ is plotted in black. Both vectors are rescaled to have norm $\sqrt{n}$.}
    \label{fig:spectrum_nb}
\end{figure}

\begin{figure}
    \centering
    \includegraphics[width=0.49\textwidth]{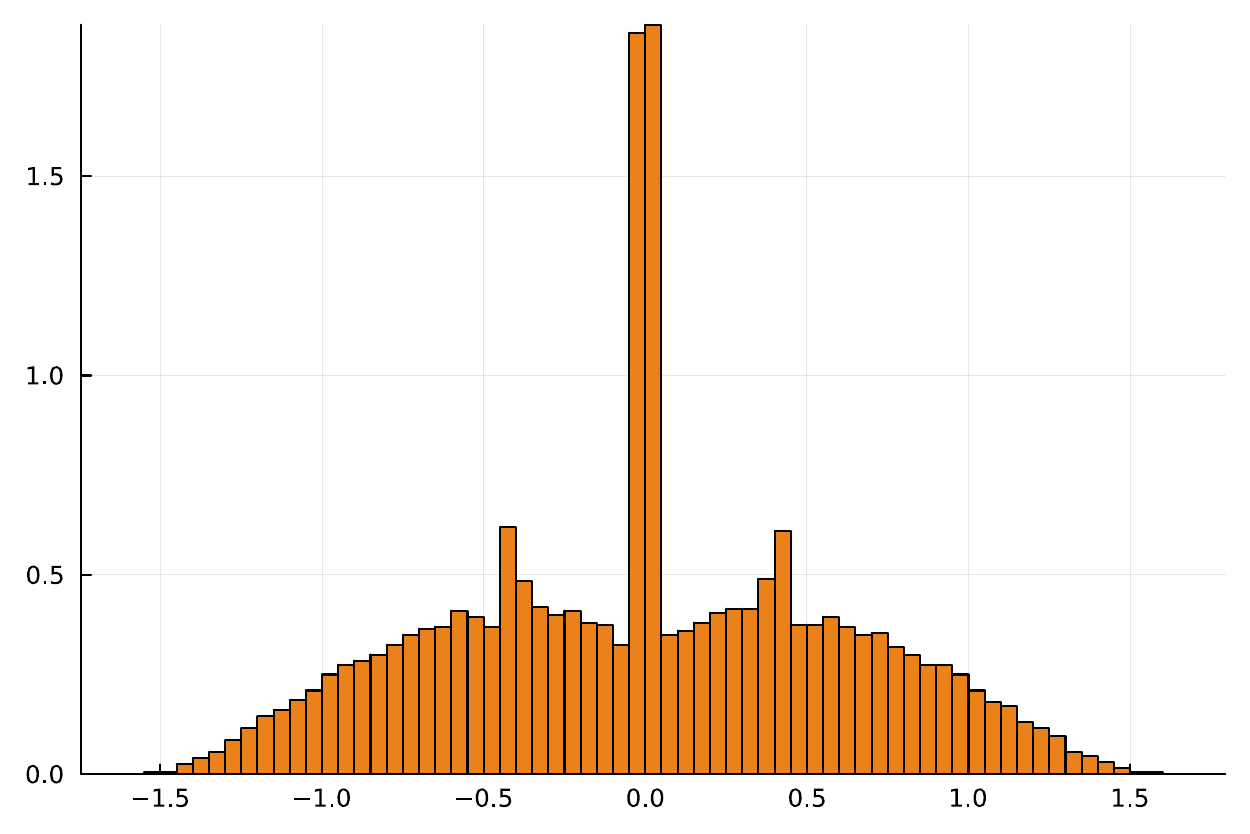}
    \includegraphics[width=0.49\textwidth]{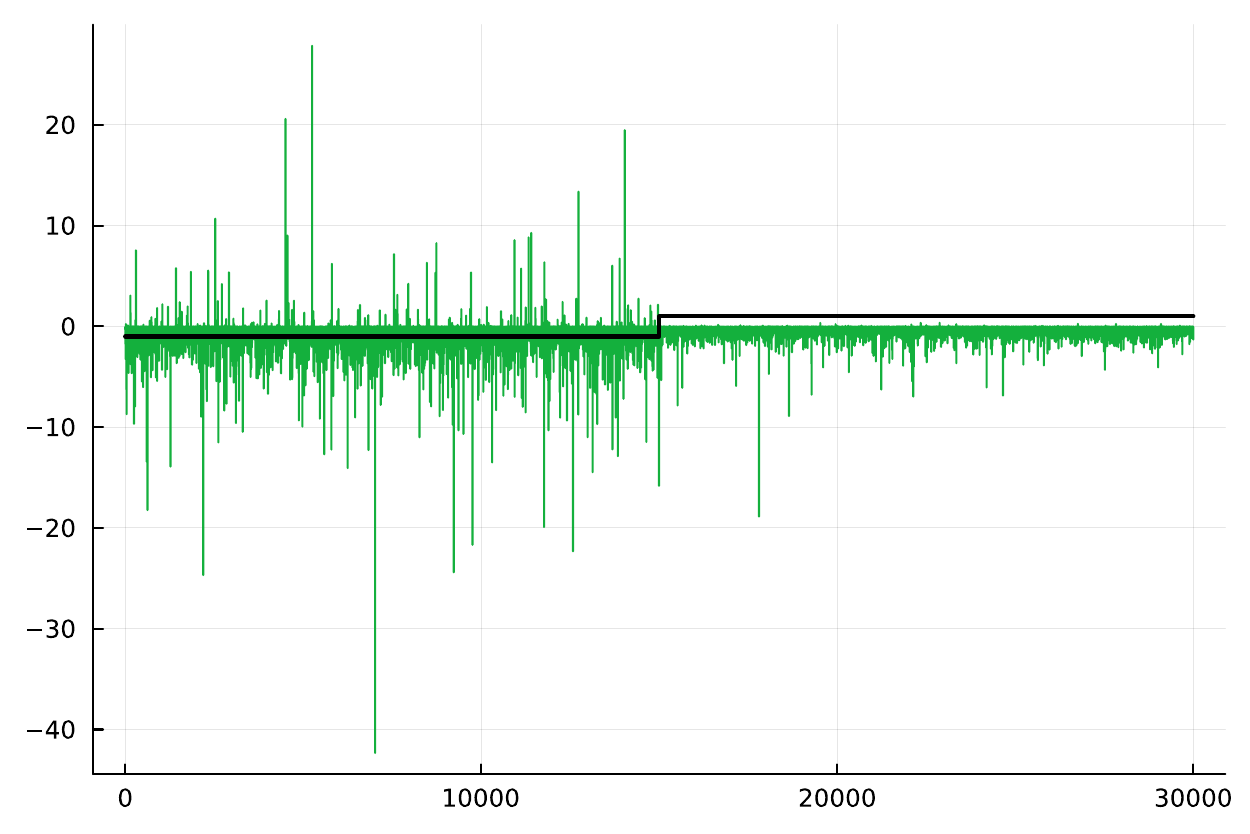}
    \caption{Spectrum density (left) and second eigenvector (right) of the matrix $Z=AA^T-\mathrm{diag}(AA^T)$. The second eigenvector of $Z$ is normalized and plotted against the vector $u$ as in Figure \ref{fig:spectrum_nb}.}
    \label{fig:baseline}
\end{figure}

\section{Spectral structure of the matrix $B$}\label{sec: structure}

In this section, we state the main intermediary result used in our paper, which characterizes almost fully the structure of the top eigenvectors of $B$. Theorems \ref{thm:eigenvalues} and \ref{thm:eigenvectors} then stem from a perturbation analysis already done in \cite{stephan2020non, bordenave2020detection, stephan2022sparse}, which we will sketch here.

\subsection{Preliminaries}
We begin with a few important definitions and properties. Define the \emph{start} and \emph{terminal} matrices $S \in \dR^{V_1 \times \vec E}$ and $T \in \dR^{\vec{E} \times V_1}$ by
\begin{equation}
    S_{xe} = \ind_{e_1 = x} \qquand T_{ex} = \ind_{e_3 = x}. \notag 
\end{equation}
This implies that for any vector $\phi \in \dR^{V_1}$, $[S^*\phi](e) = \phi_{e_1}$ and $[T\phi](e) = \phi(e_3)$. We shall also need the so-called \emph{edge inversion operator} $J$ and a diagonal weight operator $\Delta$, defined as
\begin{equation}
    J_{e, f} = \mathbf{1} \{ f=e^{-1}\} \qquand \Delta_{e, e} = A_{e_1 e_2} A_{e_3 e_2}, \notag 
\end{equation}
where $e^{-1} = (e_3, e_2, e_1)$. As a shortcut, we will denote 
\begin{align}\label{eq:matrix_relation2}
  S_\Delta = S \Delta, \quad T_\Delta = \Delta T,\quad  J_\Delta = \Delta J = J \Delta  
\end{align}
the weighted versions of $S, T, J$ respectively; and for any vector $\chi \in \dR^{\vec{E}}$, we let 
\begin{align}\label{def:check}
  \check \chi = J_\Delta \chi.  
\end{align} Those matrices are related by the following identities:
\begin{align}\label{eq:matrix_relations}
    S &= T^*J, & T &= S^*J, \\
    S_\Delta T &= S T_\Delta = AA^*, & B &= TS_\Delta - J_\Delta.
\end{align}
In particular, those identities imply the following property, named the \emph{parity-time symmetry} in \cite{bordenave2018nonbacktracking}:
\begin{equation}\label{eq:parity_time}
    J_\Delta B = B^* J_\Delta.
\end{equation}
Since we aim to recover the left singular vectors of $M$, it will be useful to consider the ``folded'' equivalents of $M$ and $Q$, namely \emph{signal} and \emph{variance} matrices:
\begin{align}\label{eq:defPQ}
P = MM^* =  \sum_{i=1}^r \nu_i^2 \phi_i \phi_i^* \quand \Phi = QQ^*
\end{align}
Finally, we define the oriented equivalents 
\begin{align}\label{eq:def_chi}
 \chi_i = T \phi_i   
\end{align} of the left singular vectors of $M$.

\subsection{Main technical results}

We are now in a position to state our main technical result. Let
\begin{equation}
    \ell = \left\lfloor \eps \log_d(n)\right\rfloor \quad \text{with} \quad \eps= \frac{(\eta/2) \wedge 1}{25}. \notag 
\end{equation}
The candidate right (resp. left) eigenvectors of $B^\ell$ are the following: for $i \in [r_0]$,
\begin{equation}\label{eq:def_u_uhat}
    u_i = \frac{B^\ell \chi_i}{\nu_i^{2\ell}} \quand \hat u_i = \frac{(B^*)^\ell \check \chi_i}{\nu_i^{2\ell+2}}
\end{equation}
with associated (pseudo)-eigenvalue $\nu_i^{2\ell}$. We shall collect those eigenvalues and eigenvectors in the matrices $ \Sigma^\ell=\diag(\nu_i^{2\ell}) \in \dR^{r_0 \times r_0}, U, \hat U \in \dR^{\vec{E} \times r_0}$. Finally, we define for all $t \geq 0$ the matrix $\Gamma^{(t)} \in \dR^{r_0 \times r_0}$ as
\begin{equation}\label{eq:def_Gamma}
    \Gamma_{ij}^{(t)} = \sum_{s = 0}^t \frac{\langle \ind, \Phi^s (\phi_i \circ \phi_j) \rangle}{(\nu_i \nu_j d)^{2s}}.
\end{equation}
Then our results are as follows:  
\begin{theorem}\label{thm:spectral_structure_Bl}
    There exists an absolute constants $c, C_0 > 0$ such that for $n\geq C_0K^{12}$, with probability at least $1 - cd^4 n^{-1/4}$, the following inequalities hold:
    \begin{align}
        \norm*{U^*U - d^2 \Gamma^{(\ell)}} &\leq C_1 n^{-\eps/2}\tau^{4\ell}, & \norm*{\hat{U}^*\hat{U} - \left(\Gamma^{(\ell+1)} - I_{r_0}\right)} &\leq C_1 n^{-\eps/2}\tau^{4\ell} \label{eq:UU_VV} ,  \\
        \norm*{U^* \hat{U} - I_{r_0}} &\leq C_1 n^{-\eps/2}\tau^{4\ell} ,
        & \norm*{\hat{U}^*B^\ell U -\Sigma^{\ell} } &\leq C_1n^{-\eps/2} \tau^{4\ell} \theta^{2\ell} \label{eq:UV_VBU} ,  \\
        \norm*{B^\ell P_{\im(\hat U)^\bot}} &\leq C_2 \theta^{2\ell}, &  \norm*{P_{\im(U)^\bot} B^\ell } &\leq C_2 \theta^{2\ell} \label{eq:BPU_PVB},
    \end{align}
    with $C_1 = c r \kappa^2 d^6\log(n)^{7/2}$ and $C_2 = cr\kappa^2 d^6 K^{20}\log^{14}(n)$, where $P_{\im(U)^\bot}$ is the projection operator onto the orthogonal complement of the vector space spanned by column vectors of $U$.
\end{theorem}
The statements \eqref{eq:UU_VV} and \eqref{eq:UV_VBU} are proved in Section~\ref{sec:pseudo_eigenvectors}, and  \eqref{eq:BPU_PVB} is proved in Section \ref{sec:bulk_radius}.

\subsection{Sketch of proof for Theorems \ref{thm:eigenvalues} and \ref{thm:eigenvectors}}

Theorems \ref{thm:eigenvalues} and \ref{thm:eigenvectors} stem from Theorem \ref{thm:spectral_structure_Bl} via non-Hermitian perturbation arguments. Details can be found in \cite{stephan2020non, bordenave2020detection}.  Following the steps in \cite[Section 5.2]{stephan2020non}, the bounds in Theorem \ref{thm:spectral_structure_Bl} imply that $B$ is nearly diagonalized by $U, \hat U$, in the sense that for a constant $C_3$ depending polynomially on $r, K,\kappa, d, \log n$, 
\begin{equation}
    \norm*{B^\ell - U^* \Sigma^\ell \hat U} \leq C_3 \theta^{2\ell} \notag 
\end{equation}
and $U, \hat U$ are nearly orthogonal. In turn, an application of the Bauer-Fike Theorem \cite{bauer.fike_1960_norms} yields that for $i\in [r_0]$,
\begin{equation}
    \lambda_i^\ell = \nu_i^{2\ell} + C_4 \theta^{2\ell} = \nu_i^{2\ell}\left(1 + C_4\left( \frac{\theta}{\nu_i}\right)^{2\ell}\right), \notag 
\end{equation}
and taking the $1/\ell$-th power on both sides, we get
\begin{equation}
    \lambda_i = \nu_i^2 \left(1 + C_5\left( \frac{\theta}{\nu_i}\right)^{2\ell}\right),  \notag 
\end{equation}
where $C_4, C_5$ depend polynomially on $r, K,\kappa, d, \log n$. All the other eigenvalues of $B$ satisfy \[ |\lambda| \leq C_5^{1/\ell} \theta^2.\]
This corresponds to Theorem \ref{thm:eigenvalues}. 

Now, let $\sigma = C_4 \theta^{2\ell}$. For $i \in [r_0]$, we define $\cM_i = \{ j \in [r_0] : \nu_j = \nu_i\}$, and the eigengap
\begin{equation}
    \delta_i = \min_{j \neq \cM_i} |\nu_i^{2\ell} - \nu_j^{2\ell}|. \notag 
\end{equation}
We will assume in the following that $\delta_i \geq  2\sigma$; otherwise, the bounds we show are trivial. Recall the definition of $u_i$ from \eqref{eq:def_u_uhat}, and the definition of $\chi_i$ from \eqref{eq:def_chi}. By the proof of Theorem 8 in \cite{stephan2020non}, if $\xi_i$ is an eigenvector of $B$ associated to $\lambda_i$, there exists a vector $\tilde u_i$ in $\vect(u_j: j \in \cI)$ such that
\begin{equation}
    \left\|\xi_i - \frac{\tilde u_i}{\norm{\tilde u_i}}\right\| \leq \frac{3\sigma }{\delta_i - \sigma} \leq \frac{6\sigma}{\delta_i}. \notag 
\end{equation}
From  the definition of $u_i, i\in [r_0]$ in \eqref{eq:def_u_uhat}, by linearity,  we can write 
\[\tilde u_i = \frac{B^\ell T \tilde  \phi_i}{\nu_i^{2\ell}},\]
where $\tilde \phi_i \in \vect(\phi_j: j \in \cI)$ is a unit vector. For simplicity, we assume $\tilde \phi_i = \phi_i$ (and hence $\tilde u_i = u_i$); the proof holds, \emph{mutatis mutandis}, in the general case.  

Note that using the definition of $S_{\Delta}$ from \eqref{eq:matrix_relation2}, $S_{\Delta} S_{\Delta}^*$ is a diagonal matrix such that 
\begin{align}
   (S_{\Delta} S_{\Delta}^* )_{xx}&=\sum_{e\in \vec{E}: x=e_1} (A_{e_1e_2}A_{e_3e_2})^2 \leq \left( \frac{L}{d}\right)^4 \sum_{e_2\in [m],e_3\not=x} X_{xe_2}X_{e_3e_2}. \notag 
\end{align}
Here $\sum_{e_2,e_3} X_{xe_2}X_{e_3e_2}$ counts the number of wedges in $G$ starting from $x$, which is bounded by $cd^2 \log n$ for all $x\in [n]$ for an absolute constant $c>1$ with probability at least $1-n^{-1}$ from Proposition~\ref{prop:neighbourhood_size_bounds}. Therefore $\| S_{\Delta}\| \leq  cK\theta^{2}\sqrt{\log n}$ and

\begin{equation}
    \left\|S_\Delta \xi_i - S_\Delta \frac{u_i}{\norm{u_i}}\right\| \leq \frac{c\sigma K^2\theta^2\sqrt{\log n}}{\delta_i}. \notag 
\end{equation}
On the other hand, using Lemma \ref{lem:UV_UBV} and \eqref{eq:SBtSBt} in Lemma \ref{lem:UU_VV}, and recalling  matrix relations an definitions in \eqref{eq:matrix_relation2}, \eqref{def:check}, \eqref{eq:def_chi},  we have 
\begin{align}
    \langle S_\Delta u_i, \phi_i \rangle  &= \langle u_i ,\check \chi_i \rangle = \nu_i^2  (1+ O(C_0n^{-\eps/2}\tau_i^{2\ell})) \notag  \\
    \norm{S_\Delta u_i}^2 & = \nu_i^4\left( \Gamma_{ii}^{(\ell+1)} + O(C_0 n^{-\eps/2} \tau_i^{4\ell}) \right) \notag 
\end{align}
for a constant $C_0$ depending polynomially on $r, K,\kappa, d, \log n$.
From \eqref{eq:BtBt} in Lemma \ref{lem:UU_VV},
\begin{align}
    \norm{u_i}^2=  d^2\Gamma_{ii}^{(\ell)}+ O(C_0n^{-\eps/2}\tau_i^{4\ell}). \notag 
\end{align}
 
Since $\Phi$ is a positive matrix with spectral radius $\rho^2$, from \eqref{eq:def_Gamma},
\begin{equation}
    \Gamma_{ii}^{(\ell)} = \left\langle \ind, \left(I - \frac{\Phi}{\nu_i^4 d^2} \right)^{-1} \left( \phi_i \circ \phi_i \right) \right\rangle + O(\tau_i^{4\ell})=\gamma_i + O(\tau_i^{4\ell}). \notag 
\end{equation}
Putting all of the previous bounds together, recalling the definition of $\zeta_i^R$ from \eqref{eq:reduced_eigenvec}, we do find
\begin{equation}
  \left\langle  \frac{\zeta_i}{\norm{\zeta_i}}, \phi_i \right\rangle= \left\langle   \frac{ S_\Delta\xi_i }{\norm{S_{\Delta}\xi_i}} , \phi_i\right \rangle = \frac{1}{\sqrt{\gamma_i}} + O\left( \frac{2K^2d \sqrt{\log n}~\sigma}{\delta_i} \right), \notag 
\end{equation}
which is the statement of Theorem \ref{thm:eigenvectors} upon diving both $\delta_i$ and $\sigma$ by $\nu_i^{2\ell}$. The same holds for the left eigenvectors by replacing $u_i$ by $\hat{u}_i$ in the proof.

\section{Parameter inequalities}\label{sec:inequality}
In this section, we collect a few simple inequalities that will be useful in the proof.
First, since \[\rho=\|Q\|\leq \sqrt{nm} \max_{xy}|Q_{xy}|,\] we have $K\geq 1$.  
Further, using the variational formula for singular values,
\[ \rho\geq \frac{1}{\sqrt{mn}} \langle \mathbf{1} , Q\mathbf{1}\rangle =\frac{1}{\sqrt{nm}}\sum_{xy}Q_{xy}=\sum_{x,y}M_{xy}^2=\|M\|_F^2\geq \nu_1^2.
\]
Hence 
\[ 1\leq K=\frac{L^2}{\rho}\leq \frac{L^2}{\nu_1^2}.\]

Next, we derive some entrywise bound on $\Phi:=QQ^*$. For any $x,y$,
\begin{align} \label{eq:entrywisePhi}
\Phi_{xy}=(QQ^*)_{xy}=\sum_{z} Q_{xz}Q_{yz}\leq \frac{K^2\rho^2}{n}. 
\end{align}
Let $(v_k)$ be a set of orthonormal left singular vectors for $Q$ with singular values $\mu_1,\dots, \mu_n$. For $t\geq 1$,  
\begin{align}
    (\Phi^t)_{xy}&=[(QQ^*)^t]_{xy}=\sum_{k}\mu_k^{2t}v_k(x)v_k(y)\\
    &\leq \rho^{2t-2}\sum_{k}\mu_k^{2}|v_k(x)||v_k(y)|\\
    &\leq \rho^{2t-2}\sqrt{(\Phi)_{xx}}\sqrt{(\Phi)_{yy}}\leq \frac{K^2 \rho^{2t}}{n},\label{eq:Phi_bound}
\end{align}
where the last line follows from Cauchy-Schwartz inequality and \eqref{eq:entrywisePhi}.  Similarly, let $\tilde\Phi=Q^* Q$, we have 
\begin{align}\label{eq:tildeQ}
(\tilde\Phi^{t})_{xy}\leq \frac{K^2\rho^{2t}}{m}.
\end{align}
And 
\begin{align}\label{eq:bipartitemoment}
    (\Phi^tQ)_{xy}&\leq \frac{K^2 \rho^{2t+1}}{\sqrt{nm}}, \quad  (\tilde\Phi^tQ^* )_{xy}\leq \frac{K^2 \rho^{2t+1}}{\sqrt{nm}}.
\end{align}
Moreover, for any $t\geq 1$ and $v,w\in\mathbb R^n$,
\begin{align}\label{eq:Phi_scalar_bound}
\langle v, \Phi^t w\rangle =\sum_{xy} (\Phi^t)_{xy}v(x)w(y)\leq \frac{K^2\rho^{2t}}{n} \|v\|_1 \|w\|_1.
\end{align}

\section{Local study of $G$}\label{sec:local_study}

\subsection{Preliminaries}

\paragraph{Rooted graphs and trees} We begin with a rigorous definition for the objects considered in this section:

\begin{definition}[Labeled rooted graphs]
A labeled rooted graph is a triplet $g_\star = (g, o, \iota)$ that consists in:
\begin{itemize}
    \item a graph $g = (V, E)$,
    \item a distinguished vertex (or \emph{root}) $o \in V$,
    \item a labeling function $\iota: V \to \dN$.
\end{itemize}
The set of labeled rooted graphs is denoted by $\cG_\star$.
\end{definition}
When the labeling function $\iota$ is clear, we will use the shortcut $(g, o)$ for $g_\star$. In particular, this will be the case whenever $V \subseteq \dN$ and $\iota(x) = x$. The notion of induced subgraph extends naturally to rooted and/or labeled graphs, provided that the induced subgraph does contain the root $o$.

An important object of study will be the bipartite Galton-Watson tree defined as follows:
\begin{definition}[Bipartite Galton-Watson tree]\label{def:gw_tree}
    The bipartite Galton-Watson tree $(T, \rho)$ is a random tree $(T, o, \iota)$ defined as follows:
    \begin{itemize}
        \item the root of the tree (depth 0) has label $\iota(o) = \rho$,
        \item the number of children of each vertex $x \in T$ is independent, and is defined as follows:
        \begin{itemize}
            \item if $x$ has even depth $2h$, $x$ has $\Poi(d^2)$ children at depth $2h+1$,
            \item if $x$ has odd depth $2h+1$, $x$ has exactly one child at depth $2h+2$.
        \end{itemize}
        \item each vertex $x$ of even depth has a label $\iota(x) \sim \Unif([n])$, and each vertex of odd depth has a label $\iota(x) \sim \Unif([m])$.
    \end{itemize}
\end{definition}

\paragraph{Tangle-free neighbourhoods and graphs} For a rooted graph $(g, o)$, we denote by $(g, o)_{t}$ the subgraph induced by the  vertices within distance $t$ from $o$. The boundary of this subgraph, spanned by the vertices at distance exactly $t$ from $o$, will be denoted by $\partial(g, o)_t$. The notion of \emph{tangle-freeness}, which will be crucial in our analysis, is defined as follows:
\begin{definition}[Tangle-freeness]
    A rooted graph $(g, o)$ is said to be tangle-free if it contains at most one cycle, otherwise, it is called tangled. For a given $t \in \dN$, a graph $g$ is said to be \emph{tangle-free} if for any vertex $x \in g$, the neighborhood $(g, x)_t$ is tangle-free.
\end{definition}

\subsection{Bounding the neighborhood sizes}

The main issue in bounding the neighborhood sizes is that a typical left vertex $x$ has approximately $mp =d\sqrt{\alpha}$ neighbors, which diverges as $\alpha \to \infty$. However, most of those neighbors have degree one and thus do not matter in the spectrum of $B$! We formalize this by defining the bipartite graph $\bar G = (V_1, \bar V_2, \bar E)$ induced by $V_1$ and the non-leaf vertices in $V_2$:
\begin{equation}
    \bar V_2 := \Set*{y \in V_2 : \deg(y) \geq 2}. \notag 
\end{equation}

We are then able to show the following proposition:
\begin{proposition}\label{prop:neighbourhood_tail_bound}
 There exists constants $c_0, c_1 > 0$ such that for all $x \in [n]$ and $s > 0$,
 \begin{equation}
     \Pb*{\forall t \geq 0, \quad |\partial(\bar G, x)_{2t}| \leq s d^{2t} \text{ and }  |\partial(\bar G, x)_{2t+1}| \leq s d^{2t+2}} \geq 1 - c_0 e^{-c_1 s} \notag 
 \end{equation}
\end{proposition}

\begin{proof} 
 We define
\begin{equation}
    \tilde p = \Pb{\Bin(n, p) \geq 1}=1-(1-p)^n, \quad \tilde d_1 = mp\tilde p, \quand \tilde d_2 = \tilde d_1 (1 + n p). \notag 
\end{equation}
We shall actually show the following bound: for some constants $c_0, c_1, c_2, c_3, c_4$, and $s > 1$, 
\begin{equation}\label{eq:neighbourhood_tail_ugly}
    \Pb*{\forall\  t \geq 0, \quad |\partial(\bar G, x)_{2t}| \leq c_1 s \tilde d_2^{t} \text{ and }  |\partial(\bar G, x)_{2t+1}| \leq c_2 s \tilde d_1 \tilde d_2^{t}} \geq 1 - c_3 e^{-c_4 s}
\end{equation}
We first show how this implies Proposition \ref{prop:neighbourhood_tail_bound}: for $s > 1$ the bound is vacuous as soon as $t > c \log_d(n)$ for some constant $c$, so we can always assume otherwise. Second, we have
\begin{equation}
    \tilde d_1 = d^2 + O\left( d^3/\sqrt{\alpha} \right) \quand  \tilde d_2 = d^2 + O\left( d^3/\sqrt{\alpha} \right), \notag 
\end{equation}
and hence for the timescale considered for $t$ we have $\tilde d_2^t \leq c d^{2t}$ under the assumption \eqref{eq:assumption_alpha}. Finally, we can adjust $c_0$ and $c_1$ such that $c_0 e^{-c_1} \geq 1$ and hence the bound also becomes vacuous when $s \leq 1$.

We now move to prove \eqref{eq:neighbourhood_tail_ugly}. Define the following sequences:
\begin{align*}
    \eps_{2t} &= \max\left(\tilde d_1^{-1} \tilde d_2^{-t+1} (2t), \sqrt{np} \tilde d_1^{-1/2} \tilde d_2^{-(t-1)/2} \sqrt{2t} \right), \\
    \eps_{2t+1} &= \tilde d_2^{-t/2} \sqrt{2t+1} \quand f_t = \prod_{t' = 1}^t (1+\eps_{t'}).
\end{align*}
Due to Assumption \eqref{eq:d_lowerbound}, by definition of $\eps_t, f_t$, there exists some constants $c_0, c_1, c_2$ such that
\begin{equation}
    c_0 \leq f_t \leq c_1 \quand \eps_t \leq c_1. \notag 
\end{equation}
We also define the following upper-bound sequence:
\begin{equation}
    S_{2t} = f_{2t} \tilde d_2^t \quand S_{2t+1} = f_{2t+1} \tilde d_1 \tilde d_2^{t},  \notag 
\end{equation}
which satisfies the following recursion equation:
\begin{equation}
    S_{2t+1} = (1 + \eps_{2t+1}) \tilde d_1 S_{2t} \quand S_{2t+2} = (1 + \eps_{2t+2})(1 + np)S_{2t+1}. \notag 
\end{equation}
We will also need the following lemma, which stems from a classical Chernoff bound:
\begin{lemma}\label{lem:chernoff_binomial_bound}
Let $(Y_1, \dots, Y_T)$ be i.i.d binomial random variables, with mean $\mu$. Then, for any $u \geq 0$,
\begin{equation}
    \Pb*{\sum_{i=1}^T  Y_i \geq (1+u) \mu T} \leq e^{- T \mu \gamma(u)}, \notag 
\end{equation}
where $\gamma(u) = (1+u)\log(1+u) - u$. The function $\gamma(u)$ is strictly increasing, and satisfies the inequality
\begin{equation}\label{eq:bound_legendre_transform}
    \gamma(u) \geq \min\left( \frac{u^2}4, u \right)
\end{equation}
\end{lemma}
For a given vertex $x \in \partial(\bar G, x)_{2t}$, the number of neighbors of $x$ in $\bar V_2 \setminus (\bar G, x)_{2t}$ is dominated by $\Bin(m, p \tilde p)$. For any $t \geq 0$, we can apply Lemma \ref{lem:chernoff_binomial_bound} with $u = \eps_{2t+1}$ to get
\begin{equation}
    \Pb*{|\partial(\bar G, x)_{2t+1}| \geq s S_{2t+1} \given |\partial(\bar G, x)_{2t}| \leq s S_{2t}} \leq \exp\left( -s f_{2t} \tilde d_1 \tilde d_2^{t} \gamma(\eps_{2t+1})  \right). \notag 
\end{equation}
Equation \eqref{eq:bound_legendre_transform} implies that when $u \leq c_1$, $\gamma(u) \geq c_2 u^2$ for some constant $u^2$, and hence
\begin{equation}
    \Pb*{|\partial(\bar G, x)_{2t+1}| \geq s S_{2t+1} \given |\partial(\bar G, x)_{2t}| \leq s S_{2t}} \leq \exp\left( -c s (2t+1))  \right). \notag 
\end{equation}
Now, the number of neighbors of a vertex in $\partial(\bar G, x)_{2t+1}$ in $V_1$ is dominated by a $X \sim \Bin(n, p)$ random variable, conditioned on $X \geq 1$, which is itself dominated by a $1 + \Bin(n, p)$ variable (see, e.g., \cite[Proposition 1.1]{broman2011stochastic}). This time, we apply Lemma \ref{lem:chernoff_binomial_bound} by noticing that
\begin{equation}
    \Pb*{\sum_{i=1}^T  (1 + Y_i) \geq (1+u) (1 +\mu) T} = \Pb*{\sum_{i=1}^T  Y_i \geq \left(1+u + \frac{u}{\mu}\right) \mu T}, \notag 
\end{equation}
so that
\begin{equation}
     \Pb*{|\partial(\bar G, x)_{2t+2}| \geq s S_{2t+2} \given |\partial(\bar G, x)_{2t+1}| \leq s S_{2t+1}} \leq \exp\left( -s np f_{2t+1} \tilde d_1 \tilde d_2^t \gamma(\eps_{2t+2} / np)  \right). \notag 
\end{equation}
We can now notice that by the bound on $\gamma$ in \eqref{eq:bound_legendre_transform} and definition of $\eps_{2t+2}$, we have
\begin{equation}
    np \tilde d_1 \tilde d_2^t \gamma(\eps_{2t+2} / np) \geq c (2t+2). \notag 
\end{equation}
Hence, we have shown that for any $t \geq 0$,
\begin{equation}
    \Pb*{|\partial(\bar G, x)_{t+1}| \geq s S_{t+1} \given |\partial(\bar G, x)_{t}| \leq s S_{t}} \leq \exp(- cs (t+1)). \notag 
\end{equation}
This implies that
\begin{equation}
    \Pb*{\exists t \geq 0, |\partial(\bar G, x)_{t}| \geq S_t} \leq \sum_{t \geq 0} \exp(-cs(t+1)) = \frac{e^{-cs}}{1 - e^{-cs}}, \notag 
\end{equation}
from which \eqref{eq:neighbourhood_tail_ugly} ensues since $s > 1$.
\end{proof}

With the same arguments as in \cite[Lemma 29]{bordenave2018nonbacktracking}, this tail bound implies the following proposition:

\begin{proposition}\label{prop:neighbourhood_size_bounds}
 With probability at least $1 - 1/n$, we have for all $x \in [n]$,
 \begin{equation}
    |(\bar G, x)_{2t}| \leq c \log(n) d^{2t} \quand |(\bar G, x)_{2t+1}| \leq c \log(n) d^{2t+2}. \notag 
 \end{equation}
 Further, we have for all $x\in V_1$, $p \geq 1$,
 \begin{equation}
     \E*{|(\bar G, x)_{2t}|^{p}}^{1/p} \leq c p d^{2t} , \notag 
 \end{equation}
 and
 \begin{equation}
     \E*{\max_{x\in [n], t \geq 0} |(\bar G, x)_{2t}|^{p}} \leq (cp)^p + (c\log(n))^p. \notag 
 \end{equation}
\end{proposition}

\subsection{Tree approximation}

We now couple the neighborhoods of $G$ with an appropriately defined tree process. Central to this section is the following breadth-first exploration process of a vertex $x \in [n]$:
\begin{itemize}
    \item start with $A_0 = \{ x\}$
    \item at step $t$, take any vertex $x_t \in A_t$ among those closest to $x$
    \begin{itemize}
        \item if $x_t \in V_2$, we let $N_t$ be the set of all neighbors of $x_t$ in $V_1 \setminus \bigcup_{s \leq t} A_s$
        \item if $x \in V_1$, $N_t$ is the set of neighbors of $x_t$ in $V_2 \setminus \bigcup_{s \leq t} A_s$ \emph{that have a neighbour in $V_1 \setminus \bigcup_{s \leq t} A_s$}.
    \end{itemize}
    \item update $A_{t+1} = A_t \cup N_t \setminus \{x_t\}$ and continue.
\end{itemize}
We denote by $\cF_t$ the filtration adapted to this process and call an edge $\{ e_1, e_2 \}$ a \emph{discovered edge} if for some $t \geq 0$ we have $x_t = e_1$ and $e_2 \in N_t$. 
We begin with a result about the tree-like properties of neighborhoods in $G$:

\begin{proposition}\label{prop:locally_tree_like}
 For a given vertex $x$, and $h \geq 0$, we have
 \begin{equation}
     \Pb*{(G, x)_{2h} \text{ contains a cycle}} \leq \frac{c d^{2h+2}}{n}. \notag 
 \end{equation}
 Further, with probability at least $1 - cd^{4h+4}/n$, the graph $G$ is $h$-tangle-free: for every vertex $x$, $(G, x)_{2h}$ contains at most one cycle.
\end{proposition}

\begin{proof}
Consider the exploration process outlined above, and let $\tau$ be the first time at which all the vertices in $(\bar G, x)_h$ have been revealed. By construction, $\tau$ is a stopping time for the filtration $\cF_t$, and given $\cF_\tau$, the discovered edges of $(G, x)_h$ form a spanning tree. Hence, there are two possibilities for a cycle in $(G, x)_h$:
\begin{itemize}
    \item there is an edge between a vertex of $(G, x)_h \cap V_1$ and $(G, x)_h \cap V_2$,
    \item there exists a vertex $z \in V_2 \setminus (G, x)_h$ and two vertices $x_1, x_2 \in (G, x)_h \cap V_1$ such that $z$ is connected to both $x_1$ and $x_2$. 
\end{itemize}
The number of such events of the first type is stochastically dominated by a $\Bin(|(G, x)_h|^2, p)$ variable and the second by a $\Bin(|(G, x)_h|^2, mp^2)$. Further, given $\cF_{\tau}$, those two types of events are independent.
Recall the following simple bounds:
\begin{equation}
    \Pb*{\Bin(N, q) \geq 1} \leq Nq \quand \Pb*{\Bin(N, q) \geq 2} \leq N^2 q^2. \notag 
\end{equation}
Then,
\begin{equation}
    \Pb*{(G, x)_h \text{ contains a cycle} \given \cF_\tau} \leq |(\bar G, x)_h|^2 (mp^2 + p). \notag 
\end{equation}
and similarly
\begin{equation}
    \Pb*{(G, x)_h \text{ is tangled} \given \cF_\tau} \leq |(\bar G, x)_h|^4 (m^2p^4 + p^2).\notag 
\end{equation}
Taking expectations in both cases and using Proposition \ref{prop:neighbourhood_size_bounds},
\begin{equation}
        \Pb*{(G, x)_{2h} \text{ contains a cycle}} \leq \frac{c d^{2h+2}}{n} \quand \Pb*{(G, x)_{2h} \text{ is tangled}} \leq \frac{c d^{4h + 4}}{n^2}. \notag 
\end{equation}
The second statement of the proposition easily ensues from a union bound.
\end{proof}

We now prove a coupling result between the neighborhoods of $G$ and the Galton-Watson trees introduced in Definition \ref{def:gw_tree}. Recall that the total variation between two probability measures $\dP_1, \dP_2$ is defined as
\begin{equation}
    \dtv(\dP_1, \dP_2) = \min_{\dP \in \pi(\dP_1, \dP_2)} \Pb{X_1 \neq X_2}, \notag 
\end{equation}
where $\pi(\dP_1, \dP_2)$ is the set of all \emph{couplings} of $\dP_1$ and $\dP_2$, i.e. joint distributions of random variables $(X_1, X_2)$ such that $X_1 \sim \dP_1$ and $X_2 \sim \dP_2$. 

\begin{proposition}\label{prop:tree_coupling}
There exists a constant $c \geq 0$ such that for any $h \geq 0$, $x \in V_1$, we have
 \begin{equation}
    \dtv\left( \cL((G, x)_h), \cL((T, x)_h) \right) \leq c \log(n)^2 \frac{d^{4h+3}}{n \wedge \sqrt{\alpha}}. \notag 
 \end{equation}
\end{proposition}

\begin{proof}
We proceed in two steps: we first couple the unlabeled versions of the graphs and then move on to the labels. We shall use the following classical bounds for total variation distances:
\begin{equation}\label{eq:total_variation_poisson}
    \dtv(\Bin(n, p), \Poi(np)) \leq p \quand \dtv(\Poi(\lambda), \Poi(\lambda')) \leq |\lambda - \lambda'|.
\end{equation}

Firstly, from Proposition \ref{prop:locally_tree_like}, $(G, x)_h$ is a tree with probability $1 - c d^{2h+2}/n$, so we only need to couple the offspring distributions. Consider the exploration process $(x_t)_{t \geq 0}$ defined above; at each step $t$, we let $n_t$ (resp. $m_t$) be the number of vertices in $V_1 \setminus \bigcup_{t' \leq t} A_{t'}$ (resp. $V_2 \setminus \bigcup_{t' \leq t} A_{t'}$). We aim to couple it with the same exploration process $(x'_t)_{t \geq 0}$ on $(T, x)$. Let $\dP_t$ (resp. $\dQ_t$) be the offspring distribution of $x_t$ (resp. $x'_t$). By definition, we have an explicit definition for $\dP_t$ and $\dQ_t$:
\begin{itemize}
    \item if $x_t$ has even depth, $\dP_t = \Bin(m_t, p \tilde p_t)$ where $\tilde p_t = \dP(\Bin(n_t, p) \geq 1)$, and $\dQ_t = \Poi(d^2)$;
    \item if $x_t$ has odd depth, $\dP_t = \Bin(n_t, p)$ conditioned on being at least 1, and $\dQ_t = 1$ a.s.
\end{itemize}
We begin with the second case, as it is the simplest. As before, $\dP_t$ is dominated by a distribution equal to $1 + \Bin(n, p)$, and hence
\begin{equation}
    \dP_t(|N_t| > 1) \leq \dP(\Bin(n, p) \geq 1) \leq np = d\alpha^{-1/2}. \notag 
\end{equation}
We now move to the first case. We have
\begin{equation}
    \tilde p_t = 1 - (1 - p)^{n_t} = p n_t + O(p^2 n^2) = O(np). \notag 
\end{equation}
Using the bounds in eq. \eqref{eq:total_variation_poisson},
\begin{equation}
    \dtv(\dP_t, \dQ_t) \leq p \tilde p_t + |m_t p \tilde p_t - d^2| \leq \frac{c d^2}{m} + |m_t n_t p^2 - d^2| + \frac{cd^3}{\sqrt \alpha}. \notag 
\end{equation}
The middle term can be further decomposed by noticing that $nmp^2 = d^2$:
\begin{equation}
    |m_t n_t p^2 - d^2| \leq |m_t - m|np^2 + m|n_t - n|p^2   \leq \frac{2d^2 |(G, x)_{2h}|}{n} \notag 
\end{equation}
All in all, we showed that for any $t \geq 0$,
\begin{equation}\label{eq:dTvPtQt}
    \dtv(\dP_t, \dQ_t) \leq cd^2 \left(\frac{|(G, x)_{2h}|}{n} \vee \frac{d}{\sqrt \alpha} \right).
\end{equation}
 Let $\tilde \cL$ denote the unlabeled distribution of the graphs.
By the high-probability bounds of Proposition \ref{prop:neighbourhood_size_bounds}, we obtain
\begin{align*}
    \dtv\left( \tilde \cL((G, x)_{2h}), \tilde \cL((T, x)_{2h}) \right) &\leq  \underbrace{ \frac{cd^{2h+2}}{n}  }_{\text{$(G, x)_h$ is not a tree}} + \underbrace{ \frac1n }_{\text{Prop. \ref{prop:neighbourhood_size_bounds} does not hold}}  \\
    &+c \log(n) ^2d^{2h+2} \left( \frac{d^{2h}}{n} \vee d\alpha^{-1/2} \right) \\
    &\leq c' \log(n)^2 \frac{d^{4h+3}}{n \wedge \sqrt{\alpha}}, 
\end{align*} 
where in the third part of the first inequality, we apply \eqref{eq:dTvPtQt} with a union bound over all vertices in $(G,x)_{2h}$.  This bound is lower than the one in Proposition \ref{prop:tree_coupling} since $d > 1$.

Consider now the distribution of the labels. Given a coupling between the unlabeled versions of $(G, x)_{2h}$ and $(T, x)_{2h}$, the labels of the even (resp. odd) depth vertices of $(G, x)_{2h}$ are obtained by sampling without replacement from $[n]$ (resp. $[m]$), whereas those of $(T, x)_{2h}$ are sampled with replacement. From \cite{freedman_2012_remark}, the total variation distance between sampling with and without replacement $k$ elements from a population of $N$ elements is bounded above by $k^2/N$, and hence
\begin{equation}
    \dtv(\dP, \dQ) \leq 2\frac{|(G, x)_{2h}|^2}{n}, \notag 
\end{equation}
where $\dP, \dQ$ are the labels distributions for $(G, x)_{2h}, (T, x)_{2h}$, respectively. A final application of Proposition \ref{prop:neighbourhood_size_bounds} completes the proof.
\end{proof}

\subsection{Quantitative bounds for local functionals}

The bounds obtained in Proposition \ref{prop:tree_coupling} can be viewed as a quantitative version of the local convergence of Benjamini and Schramm \cite{benjamini.schramm_2011_recurrence}, for possibly diverging depth. It is known that this coupling also implies weak convergence in the probability measure sense, i.e., the convergence of graph functionals. We now provide a quantitative version of this weak convergence, but only for a specific class of functionals, called \emph{local} functionals: 

\begin{definition}[Local functionals]
    A functional $f: \mathcal G_* \to \mathbb R$ is said to be $t$-local if $f(G,o)$ is only a function of $(G,o)_t$.
\end{definition}

\begin{proposition}\label{prop:functional_concentration}
 Let $f: \cG_\star \to \dR$ be a $2h$-local function for some $h \geq 0$, such that $f(g, o) \leq a |(g, o)_{2h}|^b$ for some $a, b$. Then with probability at least $1-n^{-1}$,
 \begin{equation}
    \left| \sum_{x \in V_1} f(G, x) - \sum_{x \in [n]} \E*{f(T, x)} \right| \leq ca\log(n)^{3/2 + b}\frac{d^{2h(1+b)}}{\sqrt{n} \wedge \alpha^{1/4}}. \notag 
 \end{equation}
\end{proposition}

\begin{proof}
Notice that the bounds on the neighborhood sizes match exactly the ones of \cite{bordenave2020detection}, and hence we can borrow the following concentration result from it (see the proof of Theorem 12.5 in \cite{bordenave2020detection}): with probability at least $1 - n^{-1}$,
\begin{equation}\label{eq:graph_concentration_bound}
    \left| \sum_{x \in V_1} f(G, x) - \sum_{x \in [n]} \E*{f(G, x)} \right| \leq c a \log(n)^{3/2 + b}  d^{2h(1 + b)} \sqrt{n}.
\end{equation}
On the other hand, we have for any $x \in V_1$
\begin{equation}
    \left|\dE f(T, x) - \dE f(G, x) \right| =  \E*{\left|f(T, x) - f(G, x) \right| \ind_{\cE(x)}}, \notag 
\end{equation}
where $\cE(x)$ denotes the event of the coupling between $(G, x)_{2h}$ and $(T, x)_{2h}$ fails. From the Cauchy-Schwarz inequality,
\begin{align*}
    \left|\dE f(T, x) - \dE f(G, x) \right| &\leq \sqrt{\Pb*{\cE(x)}} \sqrt{\E*{(f(T, x) - f(G, x))^2}} \\
    &\leq \sqrt{\dtv\left( \cL((G, x)_h), \cL((T, x)_h) \right)} \left( \sqrt{\E*{|f(G, x)|^2}} + \sqrt{\E*{|f(T, x)|^2}} \right).
\end{align*}
Proposition \ref{prop:tree_coupling} bounds the first factor.
It is easy to check that the bounds for $|(G, x)_{2h}|$ in Proposition \ref{prop:neighbourhood_size_bounds} also apply to $|(T, x)_{2h}|$, which yields
\begin{equation}\label{eq:tree_graph_exp_bound}
    \left|\dE f(T, x) - \dE f(G, x) \right| \leq c a \log(n)    d^{2b h+1} \frac{ d^{2h}}{\sqrt{n}\wedge \alpha^{1/4}} .
\end{equation}
Combining \eqref{eq:graph_concentration_bound} and \eqref{eq:tree_graph_exp_bound} yields the desired bound.
\end{proof}

\section{From tree functionals to pseudo-eigenvectors} \label{sec:tree_pseudo}

Since we now have a way to translate quantities from the Galton-Watson tree $T$ to the graph $G$, we aim to apply these results to specific functionals. We begin with a few preliminaries: given a functional $f: \cG_\star \to \dR$, define its \emph{Galton-Watson transform} $\overline f$ as
\begin{equation}
    \overline f(x) = \E*{f(T, x)}. \notag 
\end{equation}
This definition is consistent since from Definition \ref{def:gw_tree}, the distribution of a Galton-Watson tree only depends on its root label $x$. We shall also need the following definition: given a matrix $W \in \dR^{n \times m}$, and a functional $f: \cG_\star \to \dR$, define $\partial_W f$ as
\begin{equation}
    \partial_W f(g, o) = \sum_{o', o'' \in \cN_2(o)} W_{\iota(o) \iota(o')} W_{\iota(o'') \iota(o')} f(g \setminus \{o, o'\}, o'') \notag 
\end{equation}
where $o', o''$ runs over the set of non-backtracking wedges starting from $o$. In particular, when $g$ is a rooted tree, $o'$ is uniquely determined by $o''$, and the graph $(g \setminus \{o, o'\}, o'')$ is the subtree rooted at $o''$.

The functionals we will consider are of the following form: for a vector $\phi$,
\begin{equation}\label{eq:def_tree_functional}
    f_{\phi, t}(g, o) = \left( \frac{mn}{d^2} \right)^t\sum_{o_1, \dots, o_{2t}} \left(\prod_{s=0}^{t-1} M_{\iota(o_{2s}), \iota(o_{2s+1})}M_{\iota(o_{2s+2}), \iota(o_{2s+1})}  \right) \phi_{\iota(o_{2t})}
\end{equation}
where the sum runs over all non-backtracking paths of length $2t$ starting from $o$. Then, the following proposition holds:
\begin{proposition}\label{prop:functional_expectations}
 For any $t \geq 0$, and $i, j \in [r]$,
 \begin{align*}
    \overline{f_{\phi_i, t}} &= \nu_i^{2t} \phi_i, \\
    \overline{f_{\phi_i, t} f_{\phi_j, t}} &= (\nu_i \nu_j)^{2t} \left(\sum_{s=0}^{t} \frac{\Phi^s}{(\nu_i \nu_j d)^{2s}} \right) (\phi_i \circ \phi_j)
 \end{align*}
 and if $F_{\phi_i, t} = (f_{\phi_i, t+1} - \nu_i^2 f_{\phi_i, t})^2$, then
 \begin{equation}
    \overline{F_{\phi_i, t}} = \frac{\Phi^t (\phi_i \circ \phi_i)}{d^{2t}} \notag 
 \end{equation}
 Additionally, for any functional $f: \cG_\star \to \dR$, and any weight matrix $W$,
 \begin{equation}
    \overline{\partial_W f} = \frac{d^2(WW^*)}{mn} \,\overline{f}. \notag 
 \end{equation}
\end{proposition}

\begin{proof}
We proceed by reduction to \cite{stephan2020non}. Let $\tilde T$ be an ordinary Galton-Watson tree with offspring distribution $\Poi(d^2)$, such that each vertex has a random label $\Unif([n])$, and define the random weights
\begin{equation}
    \tilde W_{xy} = \frac{mn}{d^2} M_{xZ}M_{yZ}, \quad Z \sim \Unif([m]) \notag 
\end{equation}
The equivalent tree functional $\tilde f_{\phi, t}$ is defined as
\begin{equation}
    \tilde f_{\phi, t}(g, o) = \sum_{o_1, \dots, o_{t}} \left(\prod_{s=0}^{t-1} \tilde W_{\iota(o_{s}), \iota(o_{s+1})}  \right) \phi_{\iota(o_{t})}. \notag 
\end{equation}
It is easy to see that
\begin{equation}
    f_{\phi, t}(T, x) \stackrel{d}{=} \tilde f_{\phi, t}(\tilde T, x).\notag 
\end{equation}
On the other hand, the process $\tilde f$ on $(\tilde T, x)$ was already studied in depth in \cite{stephan2020non}. It corresponds to the signal matrix
\begin{equation}
    \tilde Q = \frac{d^2 \ind \ind^\top}{n} \circ \left(\frac{mn}{d^2} \cdot \frac{MM^*}{m} \right) = P\notag 
\end{equation}
and the variance matrix
\begin{equation}
    \tilde K = \frac{d^2 \ind \ind^\top}{n} \circ \left( \frac{m^2 n^2}{d^4} \frac{(M \circ M)(M \circ M)^*}{m} \right) = \frac{\Phi}{d^2}. \notag 
\end{equation}
Proposition \ref{prop:functional_expectations} then ensues from Propositions 6 and 7 in \cite{stephan2020non}.
\end{proof}

\subsection{Pseudo-eigenvectors}\label{sec:pseudo_eigenvectors}

We now make use of Proposition \ref{prop:functional_concentration} to translate the expectations on $T$ to pseudo-eigenvectors of $B$. This corresponds to Equations \eqref{eq:UU_VV}-\eqref{eq:UV_VBU} of Theorem \ref{thm:spectral_structure_Bl}. We shall prove stronger versions of those bounds in the following two lemmas and show how they imply the desired inequalities.
\begin{lemma}\label{lem:UV_UBV}
    Let $h = c \log_d(n)$. With probability at least $1 - c'd^4 n^{12c-1}$, for any $t \leq 3h$, we have
    \begin{equation}
        \left| \langle B^t \chi_i, \check \chi_j \rangle -  \nu_i^{2t+2} \delta_{ij}  \right| \leq c' \kappa^2\theta_2^{2t+2}d^{12h+4}\frac{\log (n)^{5/2}}{\sqrt n \wedge \alpha^{1/4}}. \notag 
    \end{equation}
\end{lemma}
\begin{proof} 
    Consider the following functional:
    \begin{equation}
        f(g, o) = \ind_{(g, o)_t \text{ is tangle-free}} \phi_j(o) f_{\phi_i, t+1}(g, o), \notag 
    \end{equation}
    where $f_{\phi_i, t}$ was defined in \eqref{eq:def_tree_functional}. It is easy to check that $f$ is $(2t+2)$-local, and when $(g, o)$ is $3h$-tangle-free (which happens with probability at least $1 - c'd^4 n^{12c-1}$ from Proposition \ref{prop:locally_tree_like}), there are at most two non-backtracking paths between $o$ and any vertex of $(g, o)$, hence from \eqref{eq:def_tree_functional}, \begin{equation}
        |f(g, o)| \leq 2 \frac{\kappa^2}{n} \theta_2^{2t+2} |(g, o)_{2t+2}|.
    \end{equation}
    We are, therefore, in the setting of Proposition \ref{prop:functional_concentration}. It remains to compute the corresponding quantities on $(G, x)$ and $(T, x)$. First, by definition,
    \begin{equation}
        \sum_{x\in[n]} f(G, x) = \langle B^t \chi_i, \check \chi_j \rangle, \notag 
    \end{equation}
    and by Proposition \ref{prop:functional_expectations} we have
    \begin{equation}
        \sum_{x\in[n]} f(T, x) = \sum_{x \in [n]} \phi_j(x) \nu_i^{2t+2} \phi_i(x) =\nu_i^{2t+2} \delta_{ij}, \notag 
    \end{equation}
    which completes the proof.
\end{proof}

Recall the definition of $\Gamma_{ij}^{(t)}$ from \eqref{eq:def_Gamma}. The following bounds hold:  
\begin{lemma}\label{lem:UU_VV}
    Let $h = c \log_d(n)$. With probability at least $1 - c'd^4 n^{4c-1}$, for any $t \leq h$, we have
    \begin{align}
        \left|\langle B^t \chi_i, B^t \chi_j \rangle - \nu_i^{2t}\nu_j^{2t}d^2\Gamma^{(t)}_{ij} \right| &\leq c' d^{6} \kappa^2 \log(n)^{7/2} \theta_2^{4t} \frac{d^{6t}}{\sqrt{n}\wedge \alpha^{1/4} }, \label{eq:BtBt} \\
        \left|\langle (B^*)^t \check \chi_i, (B^*)^t \check \chi_j \rangle - \nu_i^{2t+2}\nu_j^{2t+2}\left(\Gamma^{(t+1)}_{ij} - \delta_{ij} \right) \right| &\leq c' d^{6} \kappa^2 \log(n)^{7/2} \theta_2^{4t+4} \frac{d^{6t}}{\sqrt{n}\wedge \alpha^{1/4} }, \label{eq:B*tB*t}\\
        \left|\langle S_\Delta B^t  \chi_i, S_\Delta B^t  \chi_j \rangle - \nu_i^{2t+2}\nu_j^{2t+2} \Gamma^{(t+1)}_{ij} \right| &\leq c' d^6 \kappa^2 \log(n)^{7/2} \theta_2^{4t+4} \frac{d^{6t}}{\sqrt{n}\wedge \alpha^{1/4} }. \label{eq:SBtSBt}
    \end{align}
\end{lemma}

\begin{proof}
    In the following, it will be useful to define the following functional, this time acting on wedge-rooted graphs:
    \begin{equation}\label{eq:def_edge_functional}
        \vec f_{\phi, t}(g, e) = \left( \frac{mn}{d^2} \right)^t\sum_{o_2 = e_3, \dots, o_{2t}} \left(\prod_{s=1}^{t} M_{\iota(o_{2s}), \iota(o_{2s+1})}M_{\iota(o_{2s+2}), \iota(o_{2s+1})}  \right) \phi_{\iota(o_{2t+2})},
    \end{equation}
    where $e$ is a wedge and  the sum runs over all tuples $(o_i)$ such that $e_1, e_2, o_2, \dots, o_{2t+2}$ is a non-backtracking path of length $2t+2$.
    First, let
    \begin{equation}
        f(g, o) = \sum_{e: e_1 = o} \vec f_{\phi_i, t}(g, e) \vec f_{\phi_j, t}(g, e). \notag 
    \end{equation}
    Again, $f$ is $(2t+2)$-local, and
    \begin{equation}
        |f(g, o)| \leq 2 \frac{\kappa^2}{n} \theta_2^{4t} |(g, o)_{2t+2}|^2. \notag 
    \end{equation}
    Additionally, it is easy to check that
    \begin{equation}
        \sum_{x\in V_1} f(G, x) = \langle B^t \chi_i, B^t \chi_j \rangle, \notag 
    \end{equation}
    and that, letting $W$ be the all-one matrix,
    \begin{equation}
        f(T, x) = \partial_W( f_{\phi_i, t}f_{\phi_j, t}), \notag 
    \end{equation}
    hence using Proposition \ref{prop:functional_expectations},
    \begin{equation}
        \overline f = \frac{d^2 \ind \ind^*}{n} (\nu_i \nu_j)^{2t} \left(\sum_{s=0}^{t} \frac{\Phi^s}{(\nu_i \nu_j d)^{2s}} \right) (\phi_i \circ \phi_j). \notag 
    \end{equation}
    Taking the scalar product of the above equation with $\ind$, and using Proposition \ref{prop:functional_concentration}, concludes the proof of \eqref{eq:BtBt}. 
    
    For \eqref{eq:B*tB*t}, the parity-time symmetry \eqref{eq:parity_time} implies that
    \begin{equation}
        \langle (B^*)^t \check \chi_i, (B^*)^t \check \chi_j \rangle = \langle \Delta B^t \chi_i, \Delta B^t \chi_j \rangle. \notag 
    \end{equation}
    The corresponding graph functional is now given by
    \begin{equation}
        f(g, o) = \frac{mn}{d^2}\sum_{e: e_1 = o} M_{e_1 e_2} M_{e_3 e_2} \vec f_{\phi_i, t}(g, e) \vec f_{\phi_j, t}(g, e),  \notag 
    \end{equation}
    which yields
    \begin{equation}
        \overline f = \Phi (\nu_i \nu_j)^{2t} \left(\sum_{s=0}^{t} \frac{\Phi^s}{(\nu_i \nu_j d)^{2s}} \right) (\phi_i \circ \phi_j) = (\nu_i \nu_j)^{2t+2} \left( \Gamma_{ij}^{(t+1)} - \delta_{ij} \right). \notag 
    \end{equation}
    The rest of the proof proceeds as above.   Finally, the last equation corresponds to \[f(g, o) = f_{\phi_i, t}(g, o) f_{\phi_j, t}(g, o),\] and proceeds identically.
\end{proof}

Now we are ready to prove Equations \eqref{eq:UU_VV}-\eqref{eq:UV_VBU} of Theorem \ref{thm:spectral_structure_Bl}. 
\begin{proof}[Proof of  \eqref{eq:UU_VV} and \eqref{eq:UV_VBU}]
Recall the definition of $u_i, \hat{u}_i$ in \eqref{eq:def_u_uhat}.   For any $r_0\times r_0$ matrix $M$, we have 
   \begin{align}\label{eq:spectral_entrywise}
       \|M\|\leq \|M\|_F\leq r_0 \max_{ij}|M_{ij}|.
   \end{align}
   
   The first statement in \eqref{eq:UU_VV} follows from \eqref{eq:spectral_entrywise} and the entrywise bounds in  \eqref{eq:BtBt} by taking $c=\eps$ and $t=\ell$. The second statement in \eqref{eq:UU_VV} follows in the same way due to \eqref{eq:B*tB*t}.
 For any $i,j\in [r_0]$,  
 \begin{align}
     \langle u_i, \hat{u}_j\rangle =\frac{\langle B^{\ell} \chi_i, 
 (B^*)^{\ell} \check{\chi}_j\rangle}{\nu_i^{2\ell}\nu_j^{2\ell+2}}=\frac{\langle B^{2\ell} \chi_i, 
  \check{\chi}_j\rangle}{\nu_i^{2\ell}\nu_j^{2\ell+2}}. \notag 
 \end{align}
Then the first statement of \eqref{eq:UV_VBU} follows from Lemma \ref{lem:UV_UBV} by taking $c=\eps$ and $t=2\ell$.  Similarly, 
 \begin{align}
     \langle \hat{u}_i, B^{\ell} u_j\rangle =\frac{\langle B^{3\ell}\chi_j, \check\chi_i\rangle}{\nu_i^{2\ell+2}\nu_j^{2\ell} }. \notag 
 \end{align}
The second statement of \eqref{eq:UV_VBU} is proved by taking $t=3\ell$ in Lemma \ref{lem:UV_UBV}. This completes the proof.
\end{proof}

Finally, we show a result on a pseudo-eigenvector property of $\chi_i$ that shall be useful for the next step of the proof.

\begin{lemma}\label{lem:partial_telescope}
    Let $h = c \log_d(n)$. With probability at least $1 - c' d^4 n^{4c-1}$, for any $t \leq h$,
    \begin{equation}
        \norm{B^{t+1} \chi_i - \nu_i^2 B^t \chi_i}^2 \leq K^2 d^2 \theta_1^{4t} + c'd^{12}\kappa^2\theta_2^{4t}\log(n)^{7/2}\frac{d^{6t}}{\sqrt{n}\wedge \alpha^{1/4}}. \notag 
    \end{equation}
\end{lemma}

\begin{proof}
Recalling the definition of $\vec{f}$ in \eqref{eq:def_edge_functional}, we define
\begin{equation}
    f(g, o) = \sum_{e:e_1 = o} \left(\vec f_{\phi_i, t+1}(g, e) - \nu_i^2\vec f_{\phi_i, t}(g, e)\right)^2 \notag 
\end{equation}
Then $f$ is $(2t+4)$-local, and satisfies
\begin{equation}
    |f(g, o)| \leq 4\kappa^2 \theta_2^{4t} |(g, o)_{2t+4}|^2. \notag 
\end{equation}
On the other hand, by the same arguments as above, we have
\begin{equation}
    \overline f = \frac{d^2 \ind \ind^*}{n} \frac{\Phi^{t}}{d^{2t}}(\phi_i \circ \phi_i), \notag 
\end{equation}
and using \eqref{eq:Phi_scalar_bound}
\begin{equation}
    \left| \sum_{x \in [n]} \overline f(x) \right| \leq K^2 d^2 \theta_1^{4t}. \notag 
\end{equation}
Proposition~\ref{prop:functional_concentration} with a triangular inequality completes the proof.
\end{proof}

\section{Matrix expansion and norm bounds}\label{sec:bulk_radius}
We now move on to show \eqref{eq:BPU_PVB}. Let $\vec{E}_2(V)$ be the oriented wedge set on a complete bipartite graph with vertex sets $V_1, V_2$. We can extend the definition of $B$ to be an operator on $\vec{E}_2(V)$ such that  for any $e=(e_1,e_2,e_3), f=(f_1,f_2,f_3)\in \vec{E}_2(V)$,
\begin{align}\label{eq:defB2}
B_{ef}=\begin{cases} A_{f_1f_2}A_{f_3f_2}\mathbf{1} \{e\to f \} \quad  &\text{ if } e,f\in \vec{E}_2, \\
 0 & \text{ otherwise.}
\end{cases}
\end{align}

For $k\geq 0$, we can define $\Gamma_{e, f}^{2k+2}$ to be the set of non-backtracking walks of length $(2k+2)$ from $e$ to $f$ denoted by $(\gamma_0,\cdots, \gamma_{2k+2})$. We then have
\begin{align}
    B^k_{ef}=\sum_{\gamma \in \Gamma_{ef}^{2k+2}}   X_{e_1e_2}X_{e_3e_2}\prod_{s=1}^{k}A_{\gamma_{2s}\gamma_{2s+2},\gamma_{2s+1}}. \notag 
\end{align}

A bipartite graph spanned by $\gamma$ is given by all vertices and edges from  $\gamma$. We say $\gamma$ is a \textit{tangle-free} path if the bipartite graph $G$ spanned by $\gamma$ contains at most one cycle. Otherwise, we call $\gamma$ a tangled path. A bipartite graph $G$   is called  $\ell$-tangle-free  if  for any $x\in V_1$, there is at most one cycle in 
 the $\ell$-neighborhood of $x$ in $G$ denoted by $(G,x)_{\ell}$.

For $k\geq 0$, let $F_{ef}^{2k+2}\subseteq \Gamma_{ef}^{2k+2}$ be the subset of all tangle-free paths of length $(2k+2)$. 
If the bipartite graph $G$ corresponding to the biadjacency matrix $X$ is $(2\ell+2)$-tangle free, then for all $1\leq k\leq \ell$, we must have $B^k=B^{(k)}$, with  
\begin{align}
    B^{(k)}_{ef}=\sum_{\gamma \in F_{ef}^{2k+2}}   X_{e_1e_2}X_{e_3e_2}\prod_{s=1}^{k}A_{\gamma_{2s}\gamma_{2s+2},\gamma_{2s+1}}. \notag 
\end{align}

We introduce short-hand notations:
\[ A_{e_1e_3,e_2}=A_{e_1e_2}A_{e_3e_2}, \quad X_{e_1e_3,e_2}=X_{e_1e_2}X_{e_3e_2}, \quad M_{e_1e_3,e_2}=M_{e_1e_2}M_{e_3e_2}.
\]
And similarly, 
\begin{align}\label{eq:Qe}
 Q_{e_1e_3,e_2}=Q_{e_1e_2}Q_{e_3e_2}=mnM_{e_1e_3,e_2}^2.   
\end{align}

Define the corresponding centered random variables 
\[\underline{ A}_{e_1e_3,e_2}=A_{e_1e_2}A_{e_3e_2}-M_{e_1e_2}M_{e_3e_2}, \quad \underline{X}_{e_1e_3,e_2}=X_{e_1e_2}X_{e_3e_2}-\frac{d^2}{mn}.  \]
We then define $\underline{B}^{(k)}$, a centered version of $B^{(k)}$, as 
\begin{align}
  \underline{B}^{(k)}_{ef}=  \sum_{\gamma \in F_{ef}^{2k+2}}   \underline X_{e_1e_3,e_2}\prod_{s=1}^{k}\underline A_{\gamma_{2s}\gamma_{2s+2},\gamma_{2s+1}}. \notag 
\end{align}

We also define 
\[ B^{(0)}_{ef}=\mathbf{1}\{ e=f\}X_{e_1e_3,e_2}, \quad  \underline{B}^{(0)}_{ef}=\mathbf{1}\{ e=f\}\underline{X}_{e_1e_3,e_2}.\]

The following  telescoping sum formula holds for any real numbers $a_s, b_s, 0\leq s\leq \ell$:
\begin{equation}\label{eq:prod_difference_identity}
    \prod_{s=0}^{\ell} a_s=\prod_{s=0}^{\ell} b_s+\sum_{t=0}^{\ell} \prod_{s=0}^{t-1}b_s(a_t-b_t)\prod_{s=t+1}^{\ell} a_s.\notag
\end{equation}
Separating the cases $t=0$ in the sum, we can decompose $B^{(\ell)}$ as
\begin{align}
  B^{(\ell)}_{ef}= & \underline{B} ^{(\ell)}_{ef}+\frac{d^2}{mn}\sum_{\gamma\in F_{ef}^{2\ell+2}}\prod_{s=1}^{\ell}A_{\gamma_{2s}\gamma_{2s+2},\gamma_{2s+1}} \notag \\
  &+\sum_{t=1}^{\ell}\sum_{\gamma\in F_{ef}^{2\ell+2}}\underline{X}_{e_1e_3,e_2}\prod_{s=1}^{t-1} \underline{A}_{\gamma_{2s}\gamma_{2s+2},\gamma_{2s+1}}M_{\gamma_{2t}\gamma_{2t+2},\gamma_{2t+1}}\prod_{s=t+1}^{\ell}A_{\gamma_{2s}\gamma_{2s+2},\gamma_{2s+1}}. \notag 
\end{align}

For $1\leq t\leq \ell-1$, define $F_{2t,ef}^{2\ell+2}\subset \Gamma_{ef}^{2\ell+2}$ the set of non-backtracking \textit{tangled} paths $\gamma=(\gamma_0,\dots, \gamma_{2\ell+2})$ such that $(\gamma_0,\dots, \gamma_{2t}) \in F_{e,g}^{2t}$, $(\gamma_{2t+2},\dots, \gamma_{2\ell+2})\in F_{g', f}^{2\ell-2t}$ for some edges $g,g'\in \vec{E}_2(V)$. For $t=0$, $F_{0,ef}^{2\ell+2}\subset \Gamma_{ef}^{2\ell+2}$ is the set of non-backtracking tangled paths $\gamma=(\gamma_0,\dots, \gamma_{2\ell+2})$ such that $(\gamma_0,\gamma_1)=(e_1,e_2)$, and $(\gamma_2,\dots,\gamma_{2\ell+2}) \in F_{gf}^{2\ell}$ for some $g\in \vec{E}_2(V)$. Necessarily, $g_1=e_3$. Similarly, $F_{\ell,ef}^{2\ell+2}\subset \Gamma_{ef}^{2\ell+2}$ is the set of non-backtracking tangled paths $\gamma=(\gamma_0,\dots, \gamma_{2\ell+2})$ such that $(\gamma_0,\dots,\gamma_{2\ell})\in F_{e,g}^{2\ell}$, for some $g\in \vec{E}_2$, and $(\gamma_{2\ell+1},\gamma_{2\ell+2})=(f_2,f_3)$. Necessarily, $g_3=f_1$.

Next, we introduce three matrices $H, H^{(1)}$ and $H^{(2)}$ such that for $e,f\in \vec{E}_2(V)$,
\begin{align}
H_{ef}&=\frac{d^2}{mn}\mathbf{1} \{ e\to f\} \label{def:H}\\
H_{ef}^{(1)}&=\mathbf{1} \{ e\to f\}M_{f_1f_3,f_2} \label{def:H1}\\
H^{(2)}_{ef}&=\sum_{e\to g\to f} M_{e_3f_1,g_2}A_{f_1f_3,f_2},\label{def:H2}
\end{align}
where the sum is over all $g\in \vec{E}_2$ such that $e\to g\to f$ is a non-backtracking walk of length $6$. Then the following decomposition holds:  
\begin{align}
B^{(\ell)}=\underline{B}^{(\ell)}+HB^{(\ell-1)}+\sum_{t=1}^{\ell-1}\underline{B}^{(t-1)}H^{(2)}B^{(\ell-t-1)}+\underline{B}^{(\ell-1)}H^{(1)}-\sum_{t=0}^{\ell} R_t^{(\ell)}, \notag 
\end{align}
where for $1\leq t\leq \ell$,
\begin{align*}
(R_t^{(\ell)})_{ef}&=\sum_{\gamma\in F_{2t,ef}^{2\ell+2}} \underline{X}_{e_1e_3,e_2}\prod_{s=1}^{t-1} \underline{A}_{\gamma_{2s}\gamma_{2s+2},\gamma_{2s+1}}M_{\gamma_{2t}\gamma_{2t+2},\gamma_{2t+1}}\prod_{s=t+1}^{\ell}A_{\gamma_{2s}\gamma_{2s+2},\gamma_{2s+1}}\\
(R_0^{(\ell)})_{ef}&=\frac{d^2}{mn}\sum_{\gamma\in F_{0,ef}^{2\ell+2}} \prod_{s=1}^{\ell}A_{\gamma_{2s}\gamma_{2s+2},\gamma_{2s+1}}.
\end{align*}


From  \eqref{eq:matrix_relation2}, the following holds: 
\begin{align}
(TPS_{\Delta})_{ef}=\sum_{i,j}T_{ei}P_{ij}S_{jf}\Delta_{ff}=P_{e_3f_1}\Delta_{ff}=\sum_{u\in V_2} M_{e_3u}M_{f_1u}A_{f_1f_2}A_{f_3f_2}. \notag 
\end{align}

And from the singular value decomposition of $P$ in \eqref{eq:defPQ}, and the matrix relations from \eqref{eq:matrix_relations},
\begin{align}\label{eq:TPS}
    TP S_{\Delta}=\sum_{i=1}^n \nu_i^2 T\phi_i \phi_i^*S_{\Delta}=\sum_{i=1}^n \nu_i^2 \chi_i \check{\chi}_i^*.
\end{align}
Note that $H^{(2)}$ is close to  $TPS_{\Delta}^*$.
 We define $\tilde H$ such that 
\begin{align}
    H^{(2)}=\sum_{i=1}^r\nu_i^2 \chi_i\check{\chi}_i^*+\tilde{H}. \notag 
\end{align}
Therefore, the following decomposition holds:
\begin{align}
    B^{(\ell)}=&\underline{B}^{(\ell)}+HB^{(\ell-1)}+\sum_{t=1}^{\ell-1}\sum_{k=1}^r\nu_k^2\underline{B}^{(t-1)}\chi_k\check{\chi}_k^* B^{(\ell-t-1)} \notag \\
    &+\sum_{t=1}^{\ell-1} \underline{B}^{(t-1)}\tilde{H} B^{(\ell-t-1)}+\underline{B}^{(\ell-1)}H^{(1)}-\sum_{t=0}^{\ell} R_t^{(\ell)}. \label{eq:telescope_decomposition}
\end{align}

From \eqref{def:H1}, deterministically, 
\[\|H^{(1)}\|\leq \sqrt{\|H^{(1)}\|_{1} \|H^{(1)}\|_{\infty}}\leq nm \|M\|_{\infty}^2= L^2 .\]  

With the decomposition in \eqref{eq:telescope_decomposition}, the following estimate thus holds:
\begin{lemma}\label{lem:bulk}
For any unit vector $x\in \mathbb R^{\vec{E}_2(V)}$:
\begin{align*}
    \|B^{(\ell)}x\|\leq &\|\underline{B}^{(\ell)}\|+\|HB^{(\ell-1)}\|+ \sum_{t=1}^{\ell-1}\sum_{i=1}^r\nu_i^2\|\underline{B}^{(t-1)}\chi_i\| \left| \langle \check{\chi}_i, B^{(\ell-t-1)} x\rangle \right|\\
    &+\sum_{t=1}^{\ell-1} \|\underline{B}^{(t-1)}\tilde{H} B^{(\ell-t-1)}\|+L^2\|\underline{B}^{(\ell-1)}\|+\sum_{t=0}^{\ell} \|R_t^{(\ell)}\|.
\end{align*}
\end{lemma}

The bulk estimates of \eqref{eq:BPU_PVB} stem from Lemma \ref{lem:bulk} and the following proposition:

\begin{proposition}\label{prop:trace}
 Let $\chi$ be any vector among $\chi_1,\dots, \chi_r\in \mathbb C^{\vec{E}_2(V)}$.
  Let $\ell\leq \lfloor \log_{d}(n)\rfloor$. There exists a constant $C_1, c>0$  such that for $n\geq C_1K^{12}$, with  probability at least $1-cn^{-1/4}$, the following norm bounds hold for all $0\leq k\leq \ell$:
  \begin{align}
      \| \underline{B}^{(k)}\| &\leq K^{11} d^{5/2}\log(n)^{10}\theta^{2k}, \label{eq:Delta}\\ 
 \|R_k^{(\ell)}\| &\leq K^{20} d^{3/2}\log(n)^{28} \frac{L^{2\ell}}{\sqrt{mn}}, \label{eq:Rk} \\
      \| HB^{(k)}\| &\leq  d^5 \log(n)^8\, \frac{L^{2k-2}}{mn} , 
      \label{eq:HBk}\\
       \| \underline{B}^{(k)}\chi\| &\leq \kappa K^{11} d^{5/2} \log(n)^{12} \theta^{2k}. \label{eq:Delta_chi}
  \end{align}
  And for all $1\leq k\leq \ell-1$, 
  \begin{align}
      \|\underline{B}^{(k-1)}\tilde{H} B^{(\ell-k-1)}\|\leq K^{11}\kappa^2 d^{3} \log(n)^{18} \frac{\theta^{2(k-1)} L^{2(\ell-k)}}{mn}.\label{eq:Sk}
  \end{align}
\end{proposition}

Before we prove Proposition \ref{prop:trace} in the next section, we show how it implies eq. \eqref{eq:BPU_PVB}.

 \begin{proof}[Proof of \eqref{eq:BPU_PVB}]
With the parity-time invariance from \eqref{eq:parity_time}, 
\begin{align}\notag 
    \langle \check{\chi}_i, B^{t} w\rangle= \langle B^{t}\chi_i, J_{\Delta}w\rangle.
\end{align}
and we have $\|J_{\Delta}w\|\leq \frac{L^2}{d^2}$.
Since $w$ is orthogonal to all $B^{\ell}\chi_i, i\in [r_0]$,
\begin{align*}
    |\nu_i^{-2t} \langle \check{\chi}_i, B^{t} w\rangle|&=|\nu_i^{-2t} \langle \check{\chi}_i, B^{t} w\rangle-\nu_i^{-2\ell} \langle \check{\chi}_i, B^{\ell} w\rangle|\\
    &=\left|\sum_{s=t}^{\ell-1}\nu_i^{-2s} \langle \check{\chi}_i, B^{s} w\rangle-\nu_i^{-2(s+1)} \langle \check{\chi}_i, B^{s+1} w\rangle\right|\\
    &\leq \frac{L^2}{d^2}\sum_{s=t}^{\ell-1}\nu_i^{-2s-2} \| B^{s+1}\chi_i- \nu_i^2 B^s\chi_i\|.
\end{align*}
With Lemma~\ref{lem:partial_telescope}, 
\begin{align}\notag 
    \| B^{s+1}\chi_i- \nu_i^2 B^s\chi_i\|\leq Kd\theta^{2s}+c\kappa d^6\theta^{2s}\log(n)^{7/4}\frac{d^{3s}}{n^{1/4} \wedge \alpha^{1/8} }.
\end{align}
Therefore since $\nu_i\geq \theta$, for $i\in [r_0]$,
\begin{align}\label{eq:ir0}
    |   \langle \check{\chi}_i, B^{t} w\rangle|&\leq  Kd\nu_i^{2t}\sum_{s=t}^{\ell-1}\left(\frac{\theta}{\nu_i} \right)^{2s+2} +\frac{c\kappa d^6\log^{7/4}(n)}{n^{1/4}\wedge \alpha^{1/8}}\nu_i^{2t}\sum_{s=t}^{\ell-1}\left(\frac{\theta}{\nu_i} \right)^{2s+2}d^{3s}\\
    &\leq  \theta^{2t} \left(Kd\log n+\frac{c\kappa d^6\log^{9/4}(n)d^{3\ell}}{n^{1/4}\wedge \alpha^{1/8} }\right).
\end{align}
On the other hand, for $i>r_0$, $\nu_i\leq \theta$, from \eqref{eq:BtBt},
\begin{align}\notag 
   | \langle \check{\chi}_i, B^{t} w\rangle| & \leq \frac{L^2}{d^2} \|B^{t} \chi_i\|\leq L \nu_i^{2t} \sqrt{\Gamma_{ii}^{(t)}}+cd^2L\kappa \theta_2^{2t}\log (n)^{7/4}  \frac{d^{3t}}{n^{1/4} \wedge \alpha^{1/8} }.
\end{align}

From  \eqref{eq:def_Gamma} and \eqref{eq:Phi_scalar_bound},
\begin{align}\notag 
    \Gamma_{ii}^{(t)}\leq \sum_{s=0}^t \frac{K^2 \rho^{2s}}{(\nu_i^2 d)^{2s}}\leq K^2 \log (n) \theta_1^{4t}\nu_i^{-4t},
\end{align}
which implies for $i>r_0$,
\begin{align}\label{eq:ir}
     | \langle \check{\chi}_i, B^{t} w\rangle| & \leq \theta^{2t} \left(L K \log^{1/2}(n) + cd^2L\kappa \log (n)^{7/4}  \frac{d^{3t}}{n^{1/4} \wedge \alpha^{1/8}}\right).
\end{align}
Let $w$ be any unit vector orthogonal to all $\hat{u}_i, i\in [r_0]$ defined in \eqref{eq:def_u_uhat}. By the tangle-free property in Proposition~\ref{prop:locally_tree_like} and Lemma \ref{lem:bulk}, with probability at least $1-cn^{-1/4}$, we can apply the bounds in Proposition~\ref{prop:trace} as well as the two estimates \eqref{eq:ir0} and \eqref{eq:ir} to conclude 
\begin{align}\notag 
    \| B^{\ell} w\|\leq cr\kappa^2 d^6K^{20}\log^{14}(n) \theta^{2\ell},
\end{align}
which  is the first claim of \eqref{eq:BPU_PVB}.  The proof of the second claim in \eqref{eq:BPU_PVB}  follows  the same argument by considering the transpose of the decomposition in \eqref{eq:telescope_decomposition}, as in the proof of \cite[Lemma 26]{stephan2022sparse}.
     \end{proof}

\section{The trace method: proof of \eqref{eq:Delta}} \label{sec: trace_method}
We now show Proposition \ref{prop:trace}. We shall only show eq. \eqref{eq:Delta} in the main text; other inequality are dealt with using similar methods, and relegated to the appendix.
For any integer $s\geq 0$,
\begin{align}
    \|\underline{B}^{(k)}\|^{2s}&\leq \tr \left(\underline{B}^{(k)} {\underline{B}^{(k)}}^*\right)^{s}=\sum_{(e_1,\dots,e_{2s})}\prod_{i=1}^s\underline{B}^{(k)}_{e_{2i-1},e_{2i}}\underline{B}^{(k)}_{e_{2i+1},e_{2i}} \notag \\
    &=\sum_{\gamma\in W_{k,s}}\prod_{i=1}^{2s}\underline{X}_{\gamma_{i,0}\gamma_{i,2},\gamma_{i,1}}\prod_{t=1}^k \underline{A}_{\gamma_{i,2t}\gamma_{i,2t+2},\gamma_{i,2t+1}}, \label{eq:trace_expansion_counting}
\end{align}
where $W_{k,s}$ is the set of sequences of path $(\gamma_1,\dots,\gamma_{2s})$ such that each $\gamma_i=(\gamma_{i,0},\dots,\gamma_{i,2k+2})$ is a non-backtracking tangle-free path of length $2k+2$, with boundary conditions that for all $i\in [s]$,
\begin{align}
    (\gamma_{2i,2k},\gamma_{2i,2k+1},\gamma_{2i,2k+2}) &=(\gamma_{2i-1,2k},\gamma_{2i-1,2k+1},\gamma_{2i-1,2k+2}), \notag \\
     (\gamma_{2i+1,0},\gamma_{2i+1,1},\gamma_{2i+1,2}) &=(\gamma_{2i,0},\gamma_{2i,1},\gamma_{2i,2})\label{eq:summand}
\end{align} 
with the convention that $\gamma_{2s+1}=\gamma_1$. 
For each $\gamma\in W_{k,s}$, we associate with an undirected bipartite graph $G_{\gamma}=(V_{\gamma}^{1}, V_{\gamma}^{2},E_{\gamma})$ of visited vertices and edges. From the boundary condition, $G_{\gamma}$ is connected hence \[|E_{\gamma}|-|V_{\gamma}|+1\geq 0.\]
Let $a(\gamma)=|E_{\gamma}|, v(\gamma)=|V_{\gamma}|.$ We drop the dependence on $\gamma$ for ease of notation.

Now we consider the expectation of \eqref{eq:trace_expansion_counting}. It also depends on the degree profile of vertices in $V_{\gamma}^2$. There are no vertices of degree 1 in $V_{\gamma}^2$ by our construction of $G_{\gamma}$.  let $T_{\gamma}$ be the set of distinct triplets visited by $\gamma$ and let $m_e$ be the number of times a triplet $e$ is visited. If a triplet has an overlapped edge with another triplet, we call it a \textit{bad triplet}; otherwise, we call it a \textit{good triplet}. In addition, we call a vertex in $V_{\gamma}^2$ \textit{bad} if it is a middle point of a certain bad triplet.

Let $k_e$ be the number of times that a triplet $e$ is visited in $\gamma$, which carries the weight $\underline{X}_e$ and denote $\tilde{m}_e=m_e-k_e$.
Let $\overline {W}_{k,s}$ be the subset of ${W}_{k,s}$ such that each good triplet was visited at least twice in $\gamma$.  We can now write the expectation of \eqref{eq:trace_expansion_counting} as 
\begin{align}\label{eq:trace_tildeme}
  & \sum_{\gamma\in \overline W_{k,s}}\mathbb E\prod_{e\in T_{\gamma}}\underline{X}_e^{k_e}\underline{A}_{e}^{\tilde m_e}\\
  = & \sum_{\gamma\in \overline W_{k,s}} \dE \prod_{e\in T_{\gamma}, \text{ good}}\underline{X}_e^{k_e}\underline{A}_{e}^{\tilde m_e}   \prod_{e\in T_\gamma, \text{bad}} \underline{X}_e^{k_e}\underline{A}_{e}^{\tilde m_e} \notag \\
  \leq &\sum_{\gamma\in \overline W_{k,s}}  \prod_{e\in T_{\gamma}} \left( \frac{|M_e|}{p^2}\right)^{\tilde m_e}\cdot \left(\prod_{e\in T_\gamma, \text{good}} p^2\right)\cdot \left |
 \mathbb E\prod_{e\in T_{\gamma},\text{bad}}\underline{X}_e^{m_e}\right| \notag \\
 =&\sum_{\gamma\in \overline W_{k,s}} p^{-4ks} \prod_{e\in T_{\gamma}} |M_e|^{\tilde m_e}\cdot \left(\prod_{e\in T_\gamma, \text{good}} p^2\right)\cdot \left |
 \mathbb E\prod_{e\in T_{\gamma},\text{bad}}\underline{X}_e^{m_e}\right|, \notag 
\end{align} 
where in the last identity, we use that 
\begin{align}\notag
    \sum_{e\in T_{\gamma}} \tilde m_e=2sk.
\end{align}

Notice that when $p\in [0,1/2]$, deterministically, the following inequality holds:
\begin{align}\label{eq:decoupling_inequality}
    |X_{e_1}X_{e_2}-p^2|\leq (1+4p)|(X_{e_1}-p)(X_{e_2}-p)|.
\end{align} 

 Let $T_{\gamma}(v)$ be the collection of bad triplets with the same middle point  $v$. 
Let $d(v)$ be the number of distinct edges connected to $v$ from bad triplets. For each $e\in E_{\gamma}$, with abuse of notation,  let ${m}_e$ be the number of visits of $\gamma$ to the edge $e$. Since contributions of triplets with distinct middle points are independent, with \eqref{eq:decoupling_inequality}, we can consider each contribution from $T_{\gamma}(v)$ as follows
\begin{align}
    \left|\dE\prod_{e\in T_\gamma(v)}\underline{X}_e^{m_e}\right|&=\left|\dE  \prod_{e\in T_\gamma(v)}(X_{e_1e_2}X_{e_3e_2}-p^2)^{m_e}\right| \notag\\
    &\leq  \left(1+4p \right)^{\sum_{e\in T_{\gamma}(v)} m_e} \prod_{e\in E_{\gamma}(v)}\dE |X_e-p|^{m_e}\notag\\
    &\leq \left(1+4p \right)^{\sum_{e\in T_{\gamma}(v)} m_e} 2^{d(v)}p^{d(v)}. \label{eq:dv}
\end{align}

  Notice that $G_{ \gamma}$ is a bipartite graph, and the imbalance between $|V_{\gamma}^1|$ and $|V_{\gamma}^1|$ is bounded  by the genus plus the extra imbalance from the endpoints in each $\gamma_i$, so we have 
 \begin{align}\label{eq:imbalance}
  ||V_{\gamma}^1|-|V_{\gamma}^2||\leq a-v+1+2s.   
 \end{align} 
 
Let $v_i, v_{\geq i}$ be the number of vertices with degree $i$ (at least $i$, respectively) in $V_{\gamma}^2$. Then since there are no vertices of degree 1 in $V_{\gamma}^2$, 
\begin{align}\notag
    v_2+v_{\geq 3}=|V_{\gamma}^2|, \qquand
    2v_2+\sum_{i\geq 3}iv_i\leq a.
\end{align}
Therefore 
\begin{align}\label{eq:bad_edge_multiplicity}
    \sum_{v\in V_{\gamma}^2, \text{bad}} d(v)\leq \sum_{i\geq 3} iv_i\leq \sum_{i\geq 3}3(i-2)v_i\leq 3(a-2|V_{\gamma}^2|)\leq 6(a-v+s+1).
\end{align}
where the last inequality is from \eqref{eq:imbalance}. Similarly, 
\begin{align} \label{eq:bad_vertex_count}
v_{\geq 3}\leq a-2|V_{\gamma}^2|\leq 2(a-v+2s).
\end{align}
 Since each $\gamma_i$ is tangle-free, each bad vertex can be visited at most twice for each $\gamma_i$. Therefore each bad vertex can be visited at most $4s$ times by $\gamma$, which implies
\begin{align}\label{eq:claim_bad}
     \sum_{e\in T_{\gamma}, \text{bad}} m_e\leq 4s\cdot  v_{\geq 3}\leq 8s(a-v+2s).
\end{align}

From \eqref{eq:dv}, we obtain
\begin{align}\label{eq:bad_me}
     \left |
 \mathbb E\prod_{e\in T_{\gamma},\text{bad}}\underline{X}_e^{m_e}\right|\leq \left(1+4p \right)^{8s(a-v+2s)}2^{6(a-v+s+1)}p^{\sum_{v\in V_{\gamma}^2} d(v)}.
\end{align}
Since \[\sum_{v\in V_{\gamma}^2, \text{good}}\deg(v)+\sum_{v\in V_{\gamma}^2} d(v)=a,\]
with \eqref{eq:bad_me}, we can simplify  \eqref{eq:trace_tildeme} to be 
\begin{align}
    \sum_{\gamma\in \overline W_{k,s}}\mathbb E\left[\prod_{e\in T_{\gamma}}\underline{X}_e^{k_e}\underline{A}_{e}^{\tilde m_e}\right]
    &\leq \sum_{\gamma\in \overline W_{k,s}}  \left(1+4p \right)^{8s(a-v+2s)}2^{6(a-v+s+1)}p^{a-4ks}\prod_{e\in E_\gamma}|M_e|^{\tilde m_e},\label{eq:1plusp}
\end{align}
where for each $e\in E_{\gamma}$,  with abuse of notation, $\tilde{m}_e$ is defined as the number of visits of $\gamma$ to the edge $e$ excluding $(\gamma_{i,0},\gamma_{i_1}), (\gamma_{i,1},\gamma_{i_2}), i\in [2s]$.

Define $\mathcal W_{k,s}(v,a)$ to be the set of equivalence classes of paths in $\overline{W}_{k,s}$ with $v$ distinct vertices and $s$ distinct edges. The following lemma holds.
\begin{lemma}\label{lem:Wks}
 Let $v,e$ be integers such that $a-v+1\geq 0$. Then
 \[ |\mathcal W_{k,s}(v,a)|\leq  (4(k+1)s)^{6s(a-v+1)+2s}.
 \]
\end{lemma}
 \begin{proof}[Proof of Lemma \ref{lem:Wks}]
We can use the estimation from \cite[Lemma 17]{bordenave2018nonbacktracking} and notice that each graph is spanned by $2s$ concatenations of non-backtracking walks of length $2k+2$. 
 \end{proof}
 
 We now bound the contribution of paths in each equivalence class.
 \begin{lemma}\label{lem:counting_lemma}
 Let $\gamma\in \overline W_{k,s}$ such that $|V_{\gamma}|=v, |E_{\gamma}|=a$. We have 
 \[ \sum_{\gamma': \gamma'\sim \gamma }\prod_{e\in E_{\gamma'}} |M_e|^{\tilde m_e}\leq K^{6(a-v)+12s+2sk}(\sqrt{mn})^{-4sk+v}\rho^{2sk}\left(\sqrt\frac{m}{n}\right)^{a-v}.
 \]
 \end{lemma}
 \begin{proof}[Proof of Lemma \ref{lem:counting_lemma}]
 Let $\tilde{\gamma}=(\tilde{\gamma}_1,\dots, \tilde{\gamma}_{2s})$ be the sequence of paths from $\gamma\in \overline W_{k,s}$ such that $\tilde{\gamma}_i=(\gamma_{i,2},\dots, \gamma_{i, 2k+2})$, i.e., we remove the initial two steps (the first triplet) in each $\gamma_i, 1\leq i\leq 2s$. From the boundary condition, the associated graph spanned by $\tilde{\gamma}$, denoted by $G_{\tilde \gamma}$, is connected and 
 \[ |E_{\tilde \gamma}|= a, \quad |V_{\tilde \gamma}|=v.\]

 Recall $\tilde{m}_e$ is the number of visits of the edge $e\in E_{\tilde \gamma}$, we have
 \[\sum_{e\in E_{\tilde{\gamma}}}\tilde{m}_e=4ks.\]

 
 Let $H$ be the set of edges $e\in E_{\gamma}$ such that $\tilde{m}_e=1$ and denote $h=|H|$. Note that by our construction, any edge in $H$ is either one of the two initial  or ending edges in $\tilde{\gamma}_i, i\in [2s]$, or from a bad triplet (note that it's impossible to have both edges in a triplet belong to $H$, then the contribution is zero by independence). Therefore with \eqref{eq:bad_edge_multiplicity},
 \begin{align}\label{eq:upperbound_h}
     h\leq 4s+\sum_{v\in V_{\gamma}^2,\text{bad}} d(v) \leq 10s+6(a-v+1).
 \end{align}
 We also have 
 \[ \sum_{e\in E_{\tilde\gamma}\setminus H} \tilde{m}_e=4ks-h,\]
 which implies
 \begin{align}
 \sum_{\gamma': \gamma'\sim \gamma }\prod_{e\in E_{\tilde \gamma}}|M_e|^{\tilde m_e}&\leq  \left( \sqrt\frac{K \rho}{mn}\right)^{4sk-2a+2h}\sum_{\gamma': \gamma'\sim \gamma }\prod_{e\in E_{\tilde\gamma}\setminus H} \frac{Q_e}{\sqrt{mn}}, \notag\\
 &=(K\rho)^{2sk+h-a}(\sqrt{mn})^{a-h-4sk}\sum_{\gamma': \gamma'\sim \gamma }\prod_{e\in E_{\tilde\gamma}\setminus H} {Q_e},\label{eq:Me_bound2}
 \end{align}
 where $Q_e$ is defined in \eqref{eq:def_Q_rho}, and we use \eqref{eq:def_K} to get $|M_e|\leq \sqrt\frac{K\rho}{mn}$.
Next, we  consider an upper bound on 
 \[\sum_{\gamma': \gamma'\sim \gamma }\prod_{e\in E_{\tilde\gamma}\setminus H} {Q_e}.\] 
Let $G'_{\tilde \gamma}$ be the graph spanned by all edges in $E_{\tilde \gamma}\setminus H$. Let $t_i, t_{\geq i}$ be the number of vertices in $V_{\tilde \gamma}'$ with a degree equal to $i$ and at least $i$, respectively. We then have 
\begin{align} \label{eq:degree_count}
v'=t_1+ t_2+ t_{\geq 3} &\geq  v-h, \\
t_1+2t_2+3t_{\geq 3}&\leq \sum_{k\geq 1}kt_k=2(a-h).  \notag
\end{align}
The first inequality is because removing any edge in $H$ can delete at most one vertex from $G_{\tilde\gamma}$. Since  degree-1 vertices can only appear at the endpoints of each $\tilde{\gamma}_i, i\leq 2s$, which are $\gamma_{i,2}, \gamma_{i,2k+2}, i\in [2s]$, we must have  $t_1\leq 2s$. Together with \eqref{eq:degree_count}, we obtain that
\begin{align}\label{eq:t3}
t_{\geq 3}\leq 2( a- v)+t_1\leq 2( a- v)+2s.
\end{align}
We  reduce the graph $G_{\tilde \gamma}'$ to  a multi-graph $\hat{G}_{\tilde \gamma }$ by gluing vertices of degree 2 while keeping the endpoints $\gamma_{i,2}, \gamma_{i,2k+2}, i\in [2s]$. Let $\hat v$ be the number of vertices in $\hat G_{\tilde \gamma}$.
Next, we define the edge set $\hat{E}_{\tilde \gamma}$ for $\hat{G}_{\tilde \gamma}$. We partition edges in ${E}'_{\tilde\gamma}$ into $\hat {a}$ sequences of edges $\hat{e}_j=(e_{j_1},\dots,e_{j_{q_j}})$ where $e_{j,t}=(x_{j,t-1},x_{j,t})\in  E_{\tilde \gamma}'$, $x_{j,0}, x_{j,q_j}\in \hat{V}_{\tilde \gamma}$ and $x_{j_t}\not\in \hat V_{\tilde \gamma}$ for $1\leq t\leq q_j-1$, and let $(x_{j,0},x_{j,q_j})$ be an edge in $\hat{E}_{\tilde \gamma}$.
This implies 
\[ \sum_{j=1}^{\hat a} q_j= a-h.\]
From the definition of $\hat{G}_{\tilde \gamma}$, the genus of the graph is preserved,
\begin{align} \label{eq:genus_hat}
\hat{a}-\hat{v}&= {a}'-{v}'=a-h-v'\geq a-v-h.
\end{align}
and since $a'=a-h$, $v'\geq v-h$, 
\begin{align}\label{eq:genushat}
\hat{a}-\hat{v}=a'-v'\leq a-v.
\end{align}

Moreover, since the number of degree-2 vertices in $\hat{G}_{\tilde \gamma}$ is bounded by the number of distinct endpoints in $\tilde{\gamma}_i, i\in [2s]$ , we have  from \eqref{eq:t3},
\begin{align}\label{eq:hatv}
\hat{v}&\leq t_1+t_{\geq 3}+2s \leq 2(a-v)+6s.
\end{align}

On the other hand, removing edges in $H$ does not increase the genus, so $a'-v'\leq a-v$. And with \eqref{eq:hatv}, it implies 
\begin{align}\label{eq:upperbound_hat_a}
    \hat{a}=\hat{a}-\hat{v}+\hat{v} =a'-v'+\hat{v}\leq a-v+\hat{v}\leq 3(a-v)+6s.
\end{align}


We classify the edges in $\hat{E}_{\tilde \gamma}$ according to endpoints $(x_{j,0}, x_{j,q_j})$ into 3 different groups $\hat{E}_1, \hat{E}_2$, and $\hat{E}_3$:
(1) both are in $V_{\gamma}^1$, (2) both are in $V_{\gamma}^2$, (3) one from  $V_{\gamma}^1$ and one from $V_{\gamma}^2$. 



Let $y_1,\dots,y_{\hat{v}}$ be the elements in $\hat{V}_{\gamma}$, where the first $\hat{v}_1$ vertices are from $\hat{V}_{\tilde\gamma}^1$,  and let $a_j,b_j$ be the indices such that $x_{j,0}=y_{a_j}, x_{j,q_j}=y_{b_j}$. Then 
\begin{align*}
 \sum_{\gamma': \gamma'\sim \gamma }\prod_{e\in E'_{\tilde\gamma}} {Q_e}&\leq \sum_{(y_1,\dots,y_{\hat{v}})\in [n]^{\hat{v}_1}\times [m]^{\hat v_2}}\prod_{\hat e_j \in \hat{E}_1}(QQ^*)_{y_{a_j}y_{b_j}}^{q_j/2} \prod_{\hat e_j \in \hat{E}_2}(Q^* Q)_{y_{a_j}y_{b_j}}^{q_j/2}\prod_{\hat e_j \in \hat{E}_3} \frac{K^2 \rho^{q_j}}{\sqrt{nm}}\\
& \leq  \sum_{(y_1,\dots,y_{\hat{v}})\in [n]^{\hat{v}_1}\times [m]^{\hat v_2}}\prod_{\hat e_j \in \hat{E}_1}\frac{K^2 \rho^{q_j}}{n} \prod_{\hat e_j \in \hat{E}_2}\frac{K^2 \rho^{q_j}}{m}\prod_{\hat e_j \in \hat{E}_3} \frac{K^2 \rho^{q_j}}{\sqrt{nm}} \\
&=K^{2\hat a}\rho^{ a-h}n^{\hat{v}_1-|\hat{E}_1|-\frac{1}{2} |\hat{E}_3|} m^{\hat v_2-|\hat{E}_2|-\frac{1}{2}|\hat{E}_3|}\\
&=(K^2)^{\hat a}\rho^{ a-h} (\sqrt{mn})^{\hat v-\hat{a}}\left(\sqrt \frac{m}{n}\right)^{\hat{v}_2-|\hat E_2|-\hat{v}_1+ |\hat{E}_1|}\\
&\leq (K^2)^{3(a-v)+6s}\rho^{a-h}(\sqrt{mn})^{v+h-a}\left(\sqrt\frac{m}{n}\right)^{\hat{a}-\hat{v}-|\hat{E}_3|-2(|\hat{E}_2|-\hat{v}_2)}\\
&\leq (K^2)^{3(a-v)+6s}\rho^{a-h}(\sqrt{mn})^{v+h-a}\left(\sqrt\frac{m}{n}\right)^{a-v-|\hat{E}_3|-2|\hat{E}_2|+2\hat{v}_2}
\end{align*}
where the first inequality is from \eqref{eq:bipartitemoment},   the second inequality is from \eqref{eq:entrywisePhi} and \eqref{eq:tildeQ}, the third inequality is from \eqref{eq:upperbound_hat_a} and \eqref{eq:genus_hat}, and the last one is from \eqref{eq:genushat}. Since all vertices in $\hat{V}_{\tilde\gamma}^2$ have degree at least 3,
\begin{align}\notag
-|\hat{E}_3|-2|\hat{E}_2|+2\hat{v}_2=\sum_{v\in \hat{V}_{\tilde\gamma}^2} (2-\deg (v))\leq 0.
\end{align}
We get
\begin{align}\notag
   \sum_{\gamma': \gamma'\sim \gamma }\prod_{e\in E_{\tilde\gamma}\setminus H} {Q_e}\leq (K^2)^{3(a-v)+6s}\rho^{a-h}(\sqrt{mn})^{v+h-a}\left(\sqrt\frac{m}{n}\right)^{a-v}.
\end{align}

Therefore from \eqref{eq:Me_bound2},
\begin{align}\label{eq:Me_bound_genus}
    \sum_{\gamma': \gamma'\sim \gamma }\prod_{e\in E_{\tilde \gamma}}|M_e|^{\tilde m_e}&\leq K^{6(a-v)+12s+2sk+h-a}(\sqrt{mn})^{-4sk+v}\rho^{2sk}\left(\sqrt\frac{m}{n}\right)^{a-v}\\
    &\leq K^{12(a-v+1)+22s+2sk-a}(\sqrt{mn})^{-4sk+v}\rho^{2sk}\left(\sqrt\frac{m}{n}\right)^{a-v}, \notag
\end{align}
where the last inequality is from \eqref{eq:upperbound_h} and the fact that $K\geq 1$.
This completes the proof of Lemma \ref{lem:counting_lemma}.
 \end{proof}

 We take 
 \begin{align}\label{eq:ks}
     k\leq \lfloor \log n\rfloor, \quad s=\left\lfloor \frac{\log \left(\frac{n}{8(d\vee K)^3 K^9}\right)}{8\log\log n}\right\rfloor
 \end{align}
 and let $g=a-v+1$. 
 We have the first simple constraint on $v,a$ as
 \[ 0\leq v-1\leq a\leq 4(k+1)s. \]
 Moreover, for any edge in $G_{\gamma}$ not included in any bad triplets must be visited at least twice to have a nonzero contribution in \eqref{eq:trace_expansion_counting}. And for any $\gamma$ with a nonzero contribution in \eqref{eq:trace_expansion_counting}, the number of edges visited only once is bounded by $6(g+s)$ from \eqref{eq:bad_edge_multiplicity}. This implies
$
     a\leq 2ks+3g+5s.$

From \eqref{eq:trace_expansion_counting}, \eqref{eq:1plusp} and \eqref{eq:Me_bound_genus},
\begin{align*}
      \dE \|\underline{B}^{(k)}\|^{2s}\leq & \sum_{a=1}^{4s(k+1)}\sum_{v=1}^{a+1} \mathbf{1}\{ a\leq 2ks+3g+5s\}|\mathcal W_{k,s}(v,a)|\\
      &\cdot \left(1+4p \right)^{8s(a-v+2s)}2^{6(a-v+s+1)}p^{a-4ks}\\
      &\cdot K^{12(a-v+1)+22s+2sk-a}(\sqrt{mn})^{-4sk+v}\rho^{2sk}\left(\sqrt\frac{m}{n}\right)^{a-v}\\
      \leq &n(4(k+1)s)^{2s+2}\left(1+4p \right)^{16s^2}2^{6s}K^{22s}\theta_1^{4sk}\\
      &\cdot  \sum_{g=0}^{\infty} \sum_{a=1}^{(2ks+3g+5s)\wedge 4(k+1)s}\left( \frac{K}{d}\right)^{2sk-a}\left( \frac{2K^{12}\left(1+4p\right)^{8s} (4(k+1)s)^{6s}}{ n}\right)^g \label{eq:genus_trace}
      \end{align*}

      We now consider two cases:  

(1) When $K\geq d$,  
 \eqref{eq:genus_trace} can be further bounded by
\begin{align}
    &n(4(k+1)s)^{2s+3}\left(1+4p \right)^{16s^2}2^{6s}K^{22s}\theta_1^{4sk}\cdot  \left( \frac{K}{d}\right)^{2sk}\sum_{g=0}^{\infty}  \left( \frac{2K^{12}\left(1+4p\right)^{8s} (4(k+1)s)^{6s}}{ n}\right)^g.
\end{align}

(2) When $K\leq d$, \eqref{eq:genus_trace} is bounded by
\begin{align}
n(4(k+1)s)^{2s+3}\left(1+4p\right)^{16s^2}2^{6s}K^{22s}\theta_1^{4sk} (d/K)^{5s}\sum_{g=0}^{\infty} \left( \frac{2d^3K^{9}\left(1+4p\right)^{8s} (4(k+1)s)^{6s}}{ n}\right)^g. \notag
\end{align}

Therefore in both cases, \eqref{eq:genus_trace} is bounded by 
\begin{align}
&n(4(k+1)s)^{2s+3}\left(1+4p\right)^{16s^2}2^{6s}K^{22s}\theta^{4sk} \left(1\vee \frac{d}{K}\right)^{5s}   \notag\\
&\cdot\sum_{g=0}^{\infty} \left( \frac{2(d\vee K)^3K^{9}\left(1+4p\right)^{8s} (4(k+1)s)^{6s}}{ n}\right)^g.  \notag
\end{align}

Since $d\leq n^{1/12}$, from our choices of $k$ and $s$ in \eqref{eq:ks}, for  $n\geq C_1K^{16}$, we have 
\[\left(1+4p\right)^{8s}\leq 2, \quad 
\frac{(d\vee K)^3K^{9} (4(k+1)s)^{6s}}{n}\leq \frac{1}{8},   \quad n^{\frac{1}{2s}}\leq \log(n)^5.
\]
which implies
\begin{align}
       \dE \|\underline{B}^{(k)}\|^{2s}\leq 2n (4(k+1)s)^{2s+3}2^{2s} K^{22s} \theta^{4sk}\left(1\vee \frac{d}{K}\right)^{5s}.  \notag
\end{align}
Then by Markov's inequality and the assumption $C\geq 1$,  with probability at least $1-n^{-1/2}$, the following bound holds:
\begin{align}
    \|\underline{B}^{(k)}\|\leq K^{11} d^{5/2}\log(n)^{19/2}\theta^{2k}.  \notag
\end{align}
This finishes the proof for \eqref{eq:Delta}.










\section{Additional trace methods} \label{sec:additional_trace}
\subsection{Proof of \eqref{eq:Rk} on $R_k^{(\ell)}$}
This is similar to the proof of \eqref{eq:Delta}. We only emphasize the main difference. For any $0\leq k\leq \ell,$
\begin{align}\label{eq:R_kl_bound}
   & \|R_k^{(\ell)}\|^{2s}\leq \textnormal{tr}\left[R_k^{(\ell)}{R_k^{(\ell)}}^*\right]^{s} \notag\\
    &=\sum_{\gamma\in T_{\ell,s,k}}\prod_{i=1}^{2s} \underline{X}_{\gamma_{i,0}\gamma_{i,2}\gamma_{i,1}}\prod_{t=1}^{k-1} \underline{A}_{\gamma_{i,2t}\gamma_{i,2t+2},\gamma_{i,2t+1}}M_{\gamma_{i,2k}\gamma_{i,2k+2},\gamma_{i,2k+1}}\prod_{t=k+1}^{\ell}A_{\gamma_{i,2t}\gamma_{i,2t+2},\gamma_{i,2t+1}}.
\end{align}
where $T_{\ell,s,k}$ is the set of all $(\gamma_1,\dots, \gamma_{2s})$ such that for all $i$, 
\[\gamma_i^1=(\gamma_{i,0},\gamma_{i,1},\gamma_{i,2},\dots, \gamma_{i,2k}) \quad \text{ and } \quad  \gamma_i^2=(\gamma_{i,2k+2},\dots, \gamma_{i,2\ell+1})\] are non-backtracking tangle-free and $\gamma_{i}$ is non-backtracking tangled with the same boundary condition as in \eqref{eq:summand}.

 For any $\gamma\in T_{\ell,s,k}$, denote $G_{\gamma}=(V_{\gamma}, E_{\gamma})$ as the union of the graphs $G_{\gamma_{i}^j}$ for $i\in [2s]$ and $j=1,2$.   Note that the edges $(\gamma_{i,2k},\gamma_{i,2k+1},\gamma_{2k+2})$ are not taken into account in $G_{\gamma}$. For any fixed $\gamma_{i}^1,\gamma_{i}^2,  1\leq i\leq 2s,$ there are at $m^{2s}$ many ways to choose all $\gamma_{i,2k}, 1\leq i\leq 2s$.  Let $v=|V_{\gamma}|$ and $a=|E_{\gamma}|$.
 Since $\gamma_i$ is tangled, each connected component in $G_{\gamma_i}$ contains a cycle, therefore 
 \begin{align}\label{eq:R_va}
 v\leq a.
 \end{align}

 Define the set $\mathcal T_{\ell,s,k}(v,a)$ as the set of all equivalence classes of $T_{\ell,s,k}$ with given $v$ and $a$. The following lemma from \cite{bordenave2018nonbacktracking} holds.
 \begin{lemma} \label{lem:TLSH}
 Let $v\leq a$. Then $|\mathcal T_{\ell,s,k}(v,a)| \leq (8(\ell+1)s)^{12s(a-v+1)+8s}$. 
 \end{lemma}
 Taking expectation of \eqref{eq:R_kl_bound} yields
 \begin{align}\label{eq:R_trace_expansion_count}
   \dE  \|R_k^{(\ell)}\|^{2s}&\leq \frac{L^{4s}}{(mn)^{s}}\sum_{\gamma\in \overline T_{\ell,s,k}}\dE\prod_{i=1}^{2s} \underline{X}_{\gamma_{i,0}\gamma_{i,2}\gamma_{i,1}}\prod_{t=1}^{k-1} \underline{A}_{\gamma_{i,2t}\gamma_{i,2t+2},\gamma_{i,2t+1}}\prod_{t=k+1}^{\ell}A_{\gamma_{i,2t}\gamma_{i,2t+2},\gamma_{i,2t+1}},
 \end{align}
 where $\overline T_{\ell,s,k}$ is the subset of $T_{\ell,s,k}$ where each $\gamma \in \overline T_{\ell,s,k}$ has nonzero contribution in the expectation.  For each $\gamma \in \overline{T}_{\ell,s,k}$,
\begin{align}
   & \dE\prod_{i=1}^{2s} \underline{X}_{\gamma_{i,0}\gamma_{i,2}\gamma_{i,1}}\prod_{t=1}^{k-1} \underline{A}_{\gamma_{i,2t}\gamma_{i,2t+2},\gamma_{i,2t+1}}\prod_{t=k+1}^{\ell}A_{\gamma_{i,2t}\gamma_{i,2t+2},\gamma_{i,2t+1}}
    =\dE \prod_{e\in T_{\gamma}}\underline X_{e}^{k_e}\underline{A}_e^{\tilde m_e}A_e^{t_e},  \notag
\end{align}
where $k_e,\tilde m_e,t_e$ is the number of visits of $\gamma$ to a triplet $e$ that carry the weight $\underline{X}_e, \underline{A}_e,$ and $A_e$, respectively and let $m_e=k_e+\tilde{m}_e$. let $T_{\gamma}$ be the set of all triplets from $\gamma$. We call a triplet \textit{bad} if it overlaps another triplet. Separating good and bad triples yields
\begin{align}
   \dE \prod_{e\in T_{\gamma}}\underline X_{e}^{k_e}\underline{A}_e^{\tilde m_e}A_e^{t_e}&=\prod_{e\in T_{\gamma}}\left( \frac{|M_e|}{p^2}\right)^{\tilde m_e+t_e} \left|\dE \prod_{e\in T_{\gamma} \text{ good}}\underline{X}_e^{m_e}X_e^{t_e}\right|\cdot \left|\dE \prod_{e\in T_{\gamma} \text{ bad}}\underline{X}_e^{m_e}X_e^{t_e}\right|  \notag\\
   &\leq p^{-4(\ell-1)s}\prod_{e\in T_{\gamma}}|M_e| ^{\tilde m_e+t_e} \left(\prod_{e\in T_{\gamma} \text{ good}}p^2 \right)\left|\dE \prod_{e\in T_{\gamma} \text{ bad}}\underline{X}_e^{m_e}X_e^{t_e}\right|. \label{eq:R_triplet_expansion}
\end{align}

Let $T_{\gamma}(v)$ be the collection of bad triplets with the same middle point  $v$. 
Let $d(v)$ be the number of distinct edges connected to $v$ from bad triplets. For each $e\in E_{\gamma}$, with abuse of notation,  let ${m}_e,t_e$ be the corresponding number of visits of $\gamma$ to the edge $e$. 
 Applying  inequality \eqref{eq:decoupling_inequality}, we find
\begin{align}
    \left|\dE\prod_{e\in T_\gamma(v)}\underline{X}_e^{m_e}X_{e}^{t_e}\right|
    &\leq  \left(1+4p \right)^{\sum_{e\in T_{\gamma}(v)} m_e} \prod_{e\in E_{\gamma}(v)}\dE |X_e-p|^{m_e}X_e^{t_e}  \notag\\
    &\leq \left(1+4p \right)^{\sum_{e\in T_{\gamma}(v)} m_e} 2^{d(v)}p^{d(v)}. \label{eq:R_dv} 
\end{align}

Let $v_i, v_{\geq i}$ be the number of vertices with degree $i$ (at least $i$, respectively) in $V_{\gamma}^2$.
Same as  \eqref{eq:bad_edge_multiplicity} and \eqref{eq:bad_vertex_count}, we have 
\begin{align}
    \sum_{v\in V_{\gamma}^2, \text{ bad}}d(v) &\leq 6(a-v+s+1) \label{eq:R_dv_upperbound} \\\
    v_{\geq 3}&\leq 2(a-v+2s).  \notag
\end{align}
Since each $\gamma_{i}^j$ is tangle-free, each bad vertex can be visited at most twice for each $\gamma_{i}^j$. Similar to \eqref{eq:claim_bad}, we obtain 
\begin{align}
    \sum_{e\in T_{\gamma},\text{ bad}} m_e\leq 8s\cdot v_{\geq 3}\leq 16s(a-v+2s).  \notag
\end{align}

With the three inequalities above, \eqref{eq:R_dv} and \eqref{eq:R_triplet_expansion} imply the following bound for any $\gamma\in \overline{T}_{\ell,s,k}$:
\begin{align}
     \dE \prod_{e\in T_{\gamma}}\underline X_{e}^{k_e}\underline{A}_e^{\tilde m_e}A_e^{t_e}\leq (1+4p)^{16s(a-v+2s)}2^{6(a-v+s+1)}p^{a-4(\ell-1)s}\prod_{e\in E_{\gamma}}|M_e| ^{\tilde m_e+t_e}.  \notag
\end{align}

Next, we bound the contribution of each equivalence class.
\begin{lemma}\label{lem:Tlsk}
    Let $\gamma\in \overline{T}_{\ell,s,k}(v,a)$. We have 
    \begin{align}\label{eq:R_Me_bound_genus}
         \sum_{\gamma': \gamma'\sim \gamma }\prod_{e\in E_{\gamma'}} |M_e|^{\tilde m_e+t_e}\leq K^{12(a-v+1)+2s(\ell-1)+40s-a}\rho^{2s(\ell-1)}(\sqrt{mn})^{v-4s(\ell-1)}\left(\sqrt{\frac{m}{n}}\right)^{a-v}.
    \end{align}
\end{lemma}
\begin{proof}[Proof of Lemma \ref{lem:Tlsk}]
    This is similar to the proof of Lemma \ref{lem:counting_lemma}. We only address the differences. Let $\tilde{\gamma}=(\tilde{\gamma_1},\dots, \tilde{\gamma}_{2s})$ be the sequence of paths  from $\gamma\in \overline T_{\ell,s,k}$ with the initial two steps (the first triplet)  removed in each $\gamma_i, 1\leq i\leq 2s$. From the boundary condition, the associated graph spanned by $\tilde{\gamma}$, denoted by $G_{\tilde \gamma}$, satisfies 
 \[ |E_{\tilde \gamma}|= a, \quad |V_{\tilde \gamma}|=v.\] 
Similarly, we have
 \[\sum_{e\in E_{\tilde{\gamma}}}\tilde{m}_e+t_e=4(\ell-1)s.\]

 
 Let $H$ be the set of edges $e\in E_{\gamma}$ such that $\tilde{m}_e+t_e=1$ and denote $h=|H|$.  Therefore with \eqref{eq:R_dv_upperbound},
 \begin{align}\label{eq:R_upperbound_h}
     h\leq 4s+\sum_{v\in V_{\gamma}^2,\text{bad}} d(v) \leq 10s+6(a-v+1).
 \end{align}
 We also have 
 \[ \sum_{e\in E_{\tilde\gamma}\setminus H} \tilde{m}_e+t_e=4(\ell-1)s-h,\]
 which implies
 \begin{align}
 \sum_{\gamma': \gamma'\sim \gamma }\prod_{e\in E_{\tilde \gamma}}|M_e|^{\tilde m_e+t_e}\leq  
 &(K\rho)^{2s(\ell-1)+h-a}(\sqrt{mn})^{a-h-4s(\ell-1)}\sum_{\gamma': \gamma'\sim \gamma }\prod_{e\in E_{\tilde\gamma}\setminus H} {Q_e},\label{eq:R_Me_bound2}
 \end{align}

Let $G'_{\tilde \gamma}$ be the graph spanned by all edges in $E_{\tilde \gamma}\setminus H$. Let $t_i, t_{\geq i}$ be the number of vertices in $V_{\tilde \gamma}'$ with a degree equal to $i$ and at least $i$, respectively. We then have 
\begin{align} \label{eq:R_degree_count}
v'=t_1+ t_2+ t_{\geq 3} &\geq  v-h, \\
t_1+2t_2+3t_{\geq 3}&\leq \sum_{k\geq 1}kt_k=2(a-h).   \notag
\end{align}
Since  degree-1 vertices can only appear at the endpoints of each $\tilde{\gamma}_i^j, i\leq 2s, j=1,2$,  we must have  
\begin{align}\label{eq:R_t1}
    t_1\leq 5s.
\end{align}
Together with \eqref{eq:R_degree_count}, we obtain that
\begin{align}\label{eq:R_t3}
t_{\geq 3} \leq 2( a- v)+5s.
\end{align}
We reduce the graph $G_{\tilde \gamma}'$ to a multi-graph $\hat{G}_{\tilde \gamma }$ by gluing vertices of degree 2 except for the endpoints of $\tilde\gamma_i^j, i\in [2s], j=1,2$. Let $\hat v$ be the number of vertices in $\hat G_{\tilde \gamma}$.  We partition edges in ${E}'_{\tilde\gamma}$ into $\hat {a}$ sequences of edges as before.
Similar to the proof of \eqref{eq:Delta}, 
\begin{align}
     \sum_{j=1}^{\hat a} q_j= a-h, \quand 
a-v-h\leq \hat{a}-\hat{v}\leq a-v.  \notag
\end{align}
Since the number of degree-2 vertices in $\hat{G}_{\tilde \gamma}$ is bounded by the number of distinct endpoints in $\tilde{\gamma}_i^j, i\in [2s], j=1,2$, we have $t_2\leq 5s$. And from \eqref{eq:R_t1} and \eqref{eq:R_t3},
\begin{align}\label{eq:R_hatv}
\hat{v}&\leq t_1+t_{\geq 3}+5s \leq 2(a-v)+15s,\\
    \hat{a}&\leq 3(a-v)+15s.  \notag
\end{align}

Following the same steps in the proof of \eqref{eq:Delta}, we get
\begin{align}\label{eq:R_Q}
   \sum_{\gamma': \gamma'\sim \gamma }\prod_{e\in E_{\tilde\gamma}\setminus H} {Q_e}\leq (K^2)^{3(a-v)+15s}\rho^{a-h}(\sqrt{mn})^{v+h-a}\left(\sqrt\frac{m}{n}\right)^{a-v}.
\end{align}

Therefore from \eqref{eq:R_upperbound_h}, \eqref{eq:R_Me_bound2},  and \eqref{eq:R_Q},
\begin{align}
    \sum_{\gamma': \gamma'\sim \gamma }\prod_{e\in E_{\tilde \gamma}}|M_e|^{\tilde m_e+t_e}&\leq K^{12(a-v+1)+2s(\ell-1)+40s-a}\rho^{2s(\ell-1)}(\sqrt{mn})^{v-4s(\ell-1)}\left(\sqrt{\frac{m}{n}}\right)^{a-v}.  \notag
\end{align}
This completes the proof of Lemma \ref{lem:Tlsk}.
\end{proof}

Now  we take 
 \begin{align}\label{eq:R_ks}
     \ell\leq  \lfloor \log n\rfloor, \quad s=\left\lfloor \frac{\log \left(\frac{n}{8(d\vee K)^3 K^9}\right)}{24\log\log n}\right\rfloor
 \end{align}
 and let $g=a-v+1$. 
From \eqref{eq:R_va}, we have the first simple constraint on $v,a$ as
 $0\leq v\leq a\leq 4(\ell+1)s.$
 Any edge in $\gamma_{i}^1, i\in [2s]$ not included in any bad triplets must be visited at least twice to have nonzero contribution in \eqref{eq:R_trace_expansion_count}. The number of edges visited  only once by $\gamma_{i}^1, i\in [2s]$ is bounded by $6(g+s)$ from \eqref{eq:R_dv_upperbound}. This implies
 \begin{align}\label{eq:R_upperbound_a}
     a\leq \frac{4ks-6(g+s)}{2}+6(g+s)+2s(2\ell-2k)=2s(2\ell-k)+3(g+s).
 \end{align}
 
From \eqref{eq:R_trace_expansion_count},\eqref{eq:R_Me_bound_genus},  \eqref{eq:R_upperbound_a}, and Lemma~\ref{lem:TLSH},
\begin{align*}
     \dE  \|R_k^{(\ell)}\|^{2s}\leq & \frac{L^{4s}}{(mn)^{s}}\sum_{a=1}^{4s(\ell+1)}\sum_{v=1}^{a} \mathbf{1}\{ a\leq 2s(2\ell-k)+3(g+s)\}(8(\ell+1)s)^{12sg+8s}\\
      &\cdot (1+4p)^{16s(g+2s)}2^{6(g+s)}p^{a-4(\ell-1)s}\\
      &\cdot K^{12g+2s(\ell-1)+40s-a}\rho^{2s(\ell-1)}(\sqrt{mn})^{v-4s(\ell-1)}\left(\sqrt{\frac{m}{n}}\right)^{a-v}\\
      \leq &L^{4s}(mn)^{-s}n(8(\ell+1)s)^{8s+1}(1+4p)^{32s^2}2^{6s}K^{40s} \theta^{4s(\ell-1)}\left( 1\vee \frac{d}{K} \right)^{3s}\\
      &\cdot \sum_{g=0}^{\infty}\left( \frac{2^6(d\vee K)^3K^{9}\left(1+4p\right)^{16s} (8(\ell+1)s)^{12s}}{ n}\right)^g.  \notag
      \end{align*}

For  $n\geq C_1K^{16}$ with a constant  $C_1$ depending only on $C$,   from \eqref{eq:R_ks}, we have 
\[\left(1+4p\right)^{16s}\leq 2, \quad 
\frac{(d\vee K)^3K^{9} (8(\ell+1)s)^{12s}}{n}\leq \frac{1}{8},   \quad n^{\frac{1}{2s}}\leq \log(n)^{13}.
\]
which implies
\begin{align}
        \dE  \|R_k^{(\ell)}\|^{2s}\leq L^{4s}(mn)^{-s}n (8(\ell+1)s)^{8s+1}2^{2s} K^{40s} \theta^{4s(\ell-1)}\left(1\vee \frac{d}{K}\right)^{3s}.   \notag
\end{align}
Recall $K, d\geq 1$ and  $\theta\leq L$.
Then by Markov's inequality,  with probability at least $1-n^{-1/2}$, the following bound holds:
\begin{align}
    \|R_k^{(\ell)}\|\leq \frac{L^{2\ell}}{\sqrt{mn}}K^{20} d^{3/2}\log(n)^{28}.  \notag
\end{align}
This finishes the proof for \eqref{eq:Rk}.

\subsection{Proof of \eqref{eq:HBk} on $HB^{(k)}$}

Recall the definition of $H$ from \eqref{def:H}. We have 
\begin{align*}
    \|HB^{(k-1)}\|^{2s}&\leq \tr[(HB^{(k-1)}(HB^{(k-1)})^*]^s \notag\\
    &\leq \left(\frac{d^2}{mn}\right)^{2s} \sum_{\gamma\in W_{k,s}} \prod_{i=1}^{s} X_{\gamma_{2i-1,2}\gamma_{2i-1,4},\gamma_{2i-1,3}}\prod_{t=2}^{k-1}A_{\gamma_{2i-1,2t},\gamma_{2i-1,2t+2},\gamma_{2i-1,2t+1}}\\
    &\cdot \prod_{t=0}^{k-1}A_{\gamma_{2i,2t},\gamma_{2i,2t+2},\gamma_{2i,2t+1}}X_{\gamma_{2i,2k},\gamma_{2i,2k+2},\gamma_{2i,2k+1}},
\end{align*}
where $W_{k,s}$ is the set of paths defined in \eqref{eq:summand}. Let $G_{\gamma}$ be defined as before, and $a=|E_{\gamma}|$, $v=|V_{\gamma}|$. The set of edges that carries weight from $H$ has cardinality at most $2s$ by the boundary conditions. Hence we have 
\begin{align}
  \dE \|HB^{(k-1)}\|^{2s} & \leq  \left(\frac{d^2}{mn}\right)^{2s} \sum_{\gamma\in W_{k,s}} \left( \frac{d}{\sqrt{mn}}\right)^{a-2s}\left( \frac{L^2}{d^2}\right)^{s(2k-4)}  \notag\\
  &=\sum_{\gamma\in W_{k,s}}\left( \frac{d}{\sqrt{mn}}\right)^{a+2s}\left( \frac{L}{d}\right)^{4s(k-2)}.  \notag
\end{align}
Recall the bound from Lemma \ref{lem:Wks} on $\mathcal W_{k,s}(v,a)$. Since each equivalence class contains at most $n^{v_1} m^{v_2}$ many  paths, we obtain 
\begin{align}
  \dE \|HB^{(k-1)}\|^{2s} & \leq \sum_{a=1}^{4ks}\sum_{v=1}^{a+1} (4(k+1)s)^{6s(a-v+1)+2s} n^{v_1}m^{v_2}\left( \frac{d}{\sqrt{mn}}\right)^{a+2s}\left( \frac{L}{d}\right)^{4s(k-2)}.  \notag 
\end{align}

The imbalance between $v_1$ and $v_2$ comes from the number of cycles in $G_{\gamma}$ and the total number of components of the paths (since on one path, the imbalance from $v_1-v_2$ increases by at most $1$, and the first path of length $k$ introduces an imbalance of $1$). We have the following upper bound that 
\begin{align*}
    v_2-v_1 \leq g-1,
\end{align*}
which implies 
\begin{align}\label{eq:nvmv}
    n^{v_1}m^{v_2}\leq  (\sqrt{mn})^{v} \left( \sqrt{\frac{m}{n}}\right)^{g-1}.
\end{align}
Take $s= \left\lfloor \frac{\log n}{12 \log\log n} \right\rfloor$.
With \eqref{eq:nvmv},
\begin{align*}
    \dE \|HB^{(k-1)}\|^{2s} & \leq (4(k+1)s)^{2s+1}L^{4ks-8s}d^{10s} n(\sqrt{mn})^{-2s}\sum_{g=0}^{\infty}\left(\frac{(4(k+1)s)^{6s}}{n} \right)^g\\
    &\leq 2n(\sqrt{mn})^{-2s}(4(k+1)s)^{2s+1}L^{4ks-8s}d^{10s}. 
\end{align*}
Applying Markov's inequality yields \eqref{eq:HBk}.

\subsection{Proof of \eqref{eq:Delta_chi} and  \eqref{eq:Sk}}

The proof of \eqref{eq:Delta_chi} is similar to the proof of \eqref{eq:Delta}, and one can adapt the proof of \cite[Equation (16.21)]{bordenave2020detection}. We skip the proof.

For \eqref{eq:Sk}, first notice that $H_{ef}^{(2)}$ defined in \eqref{def:H2} is equal to $(TPS^*)_{ef}$ defined in \eqref{eq:TPS} except for $g\in \vec{E}_2$ such that $(e,g,f)$ is a back-tracking path of length 6. This happens only when (1) $g_2=e_2$,  (2) $g_2=f_2$, (3) $f_1=e_3$.
Therefore 
\begin{align}
    -\tilde{H}_{ef}&= \left( M_{e_3e_2}M_{f_1e_2} + M_{e_3f_2}M_{f_1f_2}+ \mathbf{1} \{ f_1=e_3\} \sum_{u\in V_2} M_{e_3u}^2\right)A_{f_1f_3,f_2} \nonumber\\
    &:=M^{(1)}_{ef}+M^{(2)}_{ef}+M^{(3)}_{ef}. \nonumber
\end{align}
Then we can write
\begin{align}
    \|\underline{B}^{(k-1)}\tilde{H} B^{(\ell-k-1)}\|\leq  \sum_{i=1}^3 \|\underline{B}^{(k-1)}M^{(i)} B^{(\ell-k-1)}\|\leq \sum_{i=1}^3 \|\underline{B}^{(k-1)}\| \|M^{(i)} B^{(\ell-k-1)}\|. \nonumber
\end{align}
A similar argument of the bound  \eqref{eq:HBk}  applied to each $i\in [3]$ finishes the proof.

\section{Proof of Theorem \ref{thm:frobenius_bound}}\label{sec:app:frobenius} 
For simplicity, in this section, we assume that $\delta_{j, \ell}^{(i)}$ defined in \eqref{eq:delta} for each $\mathrm{Unfold}_i(T)$ satisfies $\delta_{j, \ell}^{(i)} > c>0$ for all $i, j$, and we place ourselves under the high-probability event of Corollary~\ref{cor:orthogonal}. Note that as $d \to \infty$, $L / d \ll \rho^{(i)}/\sqrt{d}$ for all $i$, and as such
$\tau^{(i)} \leq \tau / \sqrt{d}$.

By a simple application of the triangle inequality,
\begin{align*}
    \norm{T - \hat T}_F &\leq \sum_{j=1}^r \norm*{\nu_j \bigotimes_{i=1}^k w_j^{(i)} - \hat\nu_j \bigotimes_{i=1}^k \hat w_j^{(i)}}_F \\
    &\leq \sum_{j=1}^r |\nu_j - \hat \nu_j| + \nu_j \norm*{\bigotimes_{i=1}^k w_j^{(i)} - \bigotimes_{i=1}^k \hat w_j^{(i)}}_F
\end{align*}
The first term is easily bounded through Corollary \ref{cor:orthogonal}:
\begin{align*}
    |\nu_j - \hat \nu_j| &\leq \frac{\nu_j^2 - \lambda_j^{(1)}}{2\min(\nu_j, \hat \nu_j)} \leq C_1 \nu_j \left(\frac\tau{d} \right)^{\ell},
\end{align*}
and hence by the Cauchy-Schwarz inequality
\begin{equation*}
    \sum_{j=1}^r |\nu_j - \hat \nu_j| \leq C_2 \left(\frac\tau{d} \right)^{\ell} \norm{T}_F.
\end{equation*}
For the second term, we also apply the Cauchy-Schwarz inequality to find 
\begin{align*}
    \sum_{j=1}^r \nu_j \norm*{\bigotimes_{i=1}^k w_j^{(i)} - \bigotimes_{i=1}^k \hat w_j^{(i)}}_F &\leq \norm{T}_F \sqrt{\sum_{j=1}^r \norm*{\bigotimes_{i=1}^k w_j^{(i)} - \bigotimes_{i=1}^k \hat w_j^{(i)}}_F^2} \\
    &= \norm{T}_F \sqrt{\sum_{j=1}^r \left(2 - 2 \prod_{i=1}^k \langle w_j^{(i)}, \hat w_j^{(i)} \rangle\right)}
\end{align*}
From Corollary~\ref{cor:orthogonal}, we get 
\[ \langle w_j^{(i)}, \hat w_j^{(i)} \rangle \geq \frac1{\sqrt{\gamma_j^{(i)}}} - C_3 \left(\frac\tau{d} \right)^{\ell}, \]
where
\begin{equation*}
    \gamma_j^{(i)} = \left\langle \ind, \left(I - \frac{Q^{(i)} Q^{(i)*}}{\nu_j^4 d^2} \right)^{-1} \left( w_{j}^{(i)} \circ w_{j}^{(i)} \right) \right\rangle \leq 1 + K^2 \frac{(\tau/d)^2}{1 - (\tau/d)^2}
\end{equation*}
where the last inequality comes from Equation~\eqref{eq:Phi_scalar_bound} and $\langle \ind, w_j^{(i)} \circ w_j^{(i)} \rangle = \norm*{w_j^{(i)}}^2 = 1$.
Now, it is possible to check that for $a \geq 1$ and $x > 0$ small enough,
\begin{equation*}
    \frac1{\sqrt{1 + a \frac{x}{1-x}}} \geq 1 - ax
\end{equation*}
and applying this to the above bound yields
\[ \langle w_j^{(i)}, \hat w_j^{(i)} \rangle \geq 1 - \left(\frac{K\tau}d\right)^2 - C_3 \left(\frac\tau{d} \right)^{\ell}. \]
Finally, for $a_1, \dots, a_k \geq 0$ such that $\sum a_i \leq 1$, we have $\prod (1 - a_i) \geq 1 - \sum a_i$, and hence
\begin{equation*}
    \sum_{j=1}^r \left(2 - 2 \prod_{i=1}^k \langle w_j^{(i)}, \hat w_j^{(i)} \rangle\right) \leq kr \left(\left(\frac{K\tau}d\right)^2 + C_3 \left(\frac\tau{d} \right)^{\ell}\right).
\end{equation*}
Theorem \ref{thm:frobenius_bound} ensues from the bound $\sqrt{a+b} \leq \sqrt{a} + \frac{b}{2\sqrt{a}}$.

As in \cite{bordenave2020detection}, we could improve the bound of Theorem \ref{thm:frobenius_bound} by considering better estimators of the $w_i$, as well as an optimal shrinkage of the $\hat\nu_i$. However, this only gains constant factors over the ``naive'' ones, hence we chose to keep a simpler exposition.

\section{Proof of Theorem \ref{thm:rank_one_optimal}}\label{sec:appendix_A}

We first treat the case where $x \in \{ \pm 1\}^n$. By the natural homeomorphism between $(\{ \pm 1 \}, \times)$ and $(\dF_2, \oplus)$, each equation of the form $\hat x_{i_1} \dots \hat x_{i_k} = \tilde T_{i_1 \dots i_k}$ can be mapped to a linear equation in $\dF_2$. This is known as the $k$-XORSAT problem, and the optimal sample complexity can be found in \cite{creignou.daude_2003_smooth}:
\begin{theorem}[Proposition 4.1 from \cite{creignou.daude_2003_smooth}]
    Let $\cS$ be a random $k$-XORSAT problem, where each clause of the form $y_{i_1} \oplus \dots \oplus y_{i_k} = 0$ is selected with probability $p$. Then, for $k\geq 3$, there exists a constant $c_k$ such that if $p > c_k n^{1-k}$, then with high probability, the only solution to $\cS$ is $x = 0$.
\end{theorem}
By considering $y_i = \hat x_i x_i$, when $p \geq c_k n^{1-k}$, the only solution to $\hat x_{i_1} \dots \hat x_{i_k} = \tilde T_{i_1 \dots i_k}$ for all $i_1, \dots, i_k$ is $\hat x = x$. Note that this system can be solved through Gaussian elimination in $\dF_2$, which is a polynomial algorithm. Hence, the following holds:
\begin{proposition}\label{prop:app:bool_completion}
    If $p \geq c_k n^{1-k}$, and $T = x^{\otimes k}$ with $x \in \{\pm 1 \}^n$, then there exists a polynomial-time algorithm that outputs an estimator $\hat x$ such that $\hat x = x$ with high probability.
\end{proposition}

Now, if $T = x^{\otimes k}$ for an arbitrary $x\in \mathbb R^n$, then $\sign(T) = \sign(x)^{\otimes k}$. Hence, applying Proposition \ref{prop:app:bool_completion} to $\sign(T)$, we obtain an estimator $\hat x = \sign(x)$. Therefore, we get
\begin{equation}
    \langle \hat x, x \rangle = \langle \sign(x), x \rangle = \norm{x}_1. \nonumber
\end{equation}
Theorem \ref{thm:rank_one_optimal} then ensues from the simple inequality
\begin{equation}
    1 = \norm{x}_2^2 = \sum {x_i^2} \leq \norm{x}_\infty \sum |x_i| = \norm{x}_\infty \norm{x}_1.  \nonumber
\end{equation}
\end{document}